\documentclass[11pt]{article}

\usepackage[utf8x]{inputenc}
\usepackage[francais,english]{babel}
\usepackage[a4paper,margin=2.5cm]{geometry}
\usepackage{charter}
\usepackage{microtype}
\usepackage{tikz,pgfplots}
\usepackage{url}
\usepackage[plainpages=false,pdfpagelabels,colorlinks=true,citecolor=blue,hypertexnames=false]{hyperref}
\usepackage{amsmath,amssymb,amsthm,stmaryrd}
\usepackage{enumerate,array,multirow}
\usepackage{comment}
\usepackage{xspace}
\usepackage[all]{xy}

\def\fg{\fg\xspace}

\def\N{\mathbb N}
\def\Z{\mathbb Z}
\def\F{\mathbb F}
\def\Q{\mathbb Q}
\def\R{\mathbb R}
\def\C{\mathbb C}

\def\O{\mathcal O}
\def\m{\mathfrak m}
\def\D{\mathcal D}
\def\G{\mathcal G}

\def\Fp{\F_p}
\def\Zp{\Z_p}
\def\Qp{\Q_p}
\def\Qpbar{\bar \Q_p}
\def\Cp{\C_p}

\def\calX{\mathcal X}

\def\act#1#2{{}^{#1}\hspace{-0.4ex}#2}

\def\mod{\:\textrm{\rm mod}\:}
\def\Card{\textrm{\rm Card}\:}
\def\Tr{\textrm{\rm Tr}}

\def\Hom{\textrm{\rm Hom}}

\def\GL{\textrm{\rm GL}}
\def\det{\textrm{\rm det}}
\def\Frac{\textrm{\rm Frac}\:}
\def\id{\textrm{\rm id}}
\def\Spec{\textrm{\rm Spec}\:}
\def\Frob{\textrm{\rm Frob}}
\def\Gal{\textrm{\rm Gal}}

\def\cycl{\textrm{\rm cycl}}
\def\pcycl{p\textrm{\rm -cycl}}
\def\ppcycl{p'\textrm{\rm -cycl}}

\def\Rep{\textrm{\rm Rep}}
\def\adm{\textrm{\rm adm}}
\def\aadm{\textrm{\rm -adm}}

\def\ur{\textrm{\rm ur}}
\def\tr{\textrm{\rm tr}}
\def\fin{\textrm{\rm f}}

\def\Sen{\textrm{\rm Sen}}

\def\OCpflat{\mathcal R}
\def\Acrys{A_\crys}
\def\Ainf{A_\inf}
\def\Amu{A_\mu}
\def\Amax{A_\max}
\def\Bpinf{B^+_\inf}
\def\Bpcrys{B^+_\crys}
\def\Bpmu{B^+_\mu}
\def\Bpmax{B^+_\max}
\def\BpdR{B^+_\dR}
\def\Bcrys{B_\crys}
\def\Bmu{B_\mu}
\def\Bmax{B_\max}
\def\BdR{B_\dR}
\def\AcrysK{A_{\crys,K}}
\def\AinfK{A_{\inf,K}}
\def\AmuK{A_{\mu,K}}
\def\AmaxK{A_{\max,K}}
\def\BpinfK{B^+_{\inf,K}}
\def\BpcrysK{B^+_{\crys,K}}
\def\BpmuK{B^+_{\mu,K}}
\def\BpmaxK{B^+_{\max,K}}
\def\BcrysK{B_{\crys,K}}
\def\BmuK{B_{\mu,K}}
\def\BmaxK{B_{\max,K}}

\def\NP{\textrm{\rm NP}}
\def\Ueps{\underline \varepsilon}
\def\Up{p^\flat}
\def\Upi{\pi^\flat}
\def\Fil{\textrm{\rm Fil}}
\def\MF{\textrm{\rm MF}}
\def\gr{\textrm{\rm gr}}

\def\et{\textrm{\rm ét}}
\def\crys{\textrm{\rm crys}}
\def\inf{\textrm{\rm inf}}
\def\max{\textrm{\rm max}}
\def\st{\textrm{\rm st}}
\def\dR{\textrm{\rm dR}}
\def\HT{\textrm{\rm HT}}

\def\Vect{\textrm{\rm Vect}}

\def\ph{\vphantom{$\Fil^2_\mu$}}

\newtheorem{theo}{Theorem}[subsection]
\newtheorem{lem}[theo]{Lemma}
\newtheorem{prop}[theo]{Proposition}
\newtheorem{cor}[theo]{Corollary}

\newtheorem{conj}[theo]{Conjecture}
\theoremstyle{definition}
\newtheorem{deftn}[theo]{Definition}
\theoremstyle{remark}
\newtheorem{rem}[theo]{Remark}
\newtheorem{ex}[theo]{Example}

\title{An introduction to $p$-adic period rings}
\author{Xavier Caruso}
\date\today

\begin{document}

\maketitle

\begin{abstract}
This paper is the augmented notes of a course I gave jointly
with Laurent Berger in Rennes in 2014.
Its aim was to introduce the periods rings $\Bcrys$ and $\BdR$
and state several comparison theorems between étale and crystalline
or de Rham cohomologies for $p$-adic varieties.
\end{abstract}

\setcounter{tocdepth}{2}
\tableofcontents


\section*{Introduction}

In algebraic geometry, the word \emph{period} often refers to a 
complex number that can be expressed as an integral of an algebraic 
function over an algebraic domain. One of the simplest periods is
$2i\pi = \int_{\gamma} \frac {dt}t$, where $\gamma$ is the unit
circle in the complex plane.
Equivalently, a period can be seen as an entry of the matrix (in
rational bases) of the de Rham isomorphism:
\begin{equation}
\label{eq:dRcomplex}
\C \otimes_\Q H^r_{\text{sing}}(X(\C), \Q) \simeq
\C \otimes_K H^r_\dR(X)
\end{equation}
for an algebraic variety $X$ defined over a number field~$K$. (Here
$H^r_{\text{sing}}$ is the singular cohomology and $H^r_\dR$
denotes the \emph{algebraic} de Rham cohomology.)

The initial motivation of $p$-adic Hodge theory is the will to
design a relevant $p$-adic analogue of the notion of 
periods. To this end, our first need is to find a suitable $p$-adic 
generalization of the isomorphism~\eqref{eq:dRcomplex}. In the $p$-adic 
setting, the singular cohomology is no longer relevant; it has to be 
replaced by the étale cohomology.
Thus, what we need is a ring $B$ allowing for a canonical isomorphism:
\begin{equation}
\label{eq:dRpadic}
B \otimes_{\Qp} H^r_\et(X_{\bar K}, \Qp) \simeq
B \otimes_K H^r_\dR(X)
\end{equation}
when $K$ is now a finite extension of $\Qp$ and $X$ is a variety 
defined over $K$. Of course, the first natural candidate one thinks at 
is $B = \Cp$, the $p$-adic completion of an algebraic closure $\bar K$
of $K$. 
Unfortunately, this first period ring does not totally fill
our requirements.
More precisely, it turns out that $\Cp \otimes_{\Qp} H^r_\et(X_{\bar K}, 
\Qp)$ is isomorphic to the graded module (for the de Rham filtration) of 
$\Cp \otimes_K H^r_\dR(X)$ but not to $\Cp \otimes_K H^r_\dR(X)$ itself.
The main objective of this lecture is to detail the construction of 
two periods rings, namely $\Bcrys$ and $\BdR$, allowing for the
isomorphism~\eqref{eq:dRpadic} under some additional assumptions on 
the variety~$X$. The ring $\BdR$ (which is the bigger one) is 
often called the \emph{ring of $p$-adic periods}.

Another important aspect of $p$-adic period rings concerns the Galois 
structure of $H^r_\et(X_{\bar K}, \Qp)$. Indeed, we shall see that the 
mere existence of the isomorphism~\eqref{eq:dRpadic} usually has strong 
consequences on the Galois module $H^r_\et(X_{\bar K}, \Qp)$.
In order to give depth to this observation, Fontaine developed a 
general formalism for studying and classifying general Galois 
representations through the notion of period rings. A large part 
of this article focuses on the Galois aspects.

\paragraph{Structure of the article.}

\S\ref{sec:longintro} serves as a second long introduction to this 
article; two results which can be considered as the seeds of $p$-adic 
Hodge theory are presented and discussed. The first one is due to Tate 
and provides a Hodge-like decomposition of the Tate module of a 
$p$-divisible group in the spirit of the isomorphism~\eqref{eq:dRpadic}. 
The second result is a classification theorem of $p$-divisible groups by 
Fontaine. 
Fontaine's general formalism for studying Galois representations is 
also introduced in this section.

In \S\ref{sec:Cp}, we investigate to what extent $\Cp$ meets the
expected properties of a period ring.
We adopt the point of view of Galois representations, which means 
concretely that we will concentrate on isolating those Galois 
representations that are susceptible to sit in an isomorphism of the 
form~\eqref{eq:dRpadic} when $B = \Cp$.
This study will lead eventually to the notion of Hodge--Tate 
representations, which is related to the Hodge-like decompositions of 
cohomology presented in~\S\ref{sec:longintro}.

In \S\ref{sec:dRcrys}, we review the construction of the period rings 
$\Bcrys$ and $\BdR$; it is the heart of the article but also its most 
technical part.
Finally, in \S\ref{sec:crysdRrep}, we state several comparison
theorems between étale and de Rham cohomologies. We also show how
the rings $\Bcrys$ and $\BdR$ intervene in the classification of
Galois representations, through the notions of crystalline and
de Rham representations.

\paragraph{Some advice to the reader.}

Although we will give frequently reminders, we assume that the reader is 
familiar with the general theory of local fields as presented 
in~\cite{serre}, Chapter 1--4. A minimal knowledge of local class field 
theory~\cite{serre} and of the theory of $p$-adic analytic 
functions~\cite{lazard} is also welcome, while not rigourously needed.

To the impatient reader who is afraid by the length of this article
and is not interested in the details of the proofs (at least in first
reading) but only by a general outline of $p$-adic Hodge theory, we 
advise to read \S\ref{sec:longintro}, then the introduction of
\S\ref{sec:dRcrys} until \S\ref{ssec:Ainf} and then finally 
\S\ref{sec:crysdRrep}.

\paragraph{Acknowledgement.}

The author is grateful to the editorial board of 
\emph{Panoramas et Synthèses}, and especially to Ariane Mézard, 
who encourages him to finalize this article after many years.
He also warmly thanks Olivier Brinon for sharing with him his
notes of a master course on $p$-adic Hodge theory he taught in
Bordeaux.

\paragraph{Notations.}

Throughout this article, the letter $p$ will refer to a fixed
prime number. We use the notation $\Zp$ (resp. $\Qp$) for the 
ring of $p$-adic integers (resp. the field of $p$-adic numbers).
We recall that $\Qp = \Frac \Zp = \Zp[\frac 1 p]$. Let also $\Fp$
denote the finite field with $p$ elements, \emph{i.e.} $\Fp = 
\Z/p\Z$.

If $\mathfrak A$ in a ring, we denote by $\mathfrak A^\times$ the 
multiplicative group of invertible elements in $\mathfrak A$.

\section{From Hodge decomposition to Galois representations}
\label{sec:longintro}

After having recalled some basic facts about local fields in 
\S\ref{ssec:setting}, we discuss in \S\ref{ssec:motivations} two 
families of results which are the seeds of $p$-adic Hodge theory. Both 
of them are of geometric nature. The first one concerns the 
classification of $p$-divisible groups over the ring of integers of a 
local field, while the second one concerns the Hodge-like decomposition 
of the étale cohomology of varieties defined over local fields. From 
this presentation, the need to have a good tannakian formalism emerges.

Carried by this idea, we move from geometry to the theory of 
representations and focus on tensor products and scalar extensions. 
Eventually, this will lead us to the notion of $B$-admissibility, 
which is the key concept in Fontaine's vision of $p$-adic Hodge 
theory.
Finally, we briefly discuss the applications we will develop in the 
forthcoming sections: using $B$-admissibility, we introduce the notions 
of crystalline, semi-stable and de Rham representations and explain 
rapidly how the general theory can help for studying these classes of 
representations.

\subsection{Setting and preliminaries}
\label{ssec:setting}

Let $K$ be a finite extension\footnote{We could have considered a more 
general setting where $K$ is a complete discrete valued field of 
characteristic $0$ with perfect residue field of characteristic $p$.
All the results presented in the paper extend to this more general
setting. However the case of finite extensions of $\Qp$ is the main 
case of interest and restricting to this case simplifies the exposition 
at several points.} of $\Qp$.
Let $v_p : K \to \Q \sqcup \{+\infty\}$ be the valuation on $K$
normalized by $v_p(p) = 1$. By our assumptions, $v_p(K^\times)$
is a discrete subgroup of $\Q$ containing $\Z$; hence it is equal 
to $\frac 1 e \Z$ for some positive integer $e$. We recall that
this integer $e$ is called the \emph{absolute ramification index}
of $K$. A \emph{uniformizer} of $K$ is an element
of minimal positive valuation, that is of valuation $\frac 1 e$.
We fix a uniformizer $\pi$ of $K$.

Let $\O_K$ be the ring of integers of $K$, that is the subring of
$K$ consisting of elements with nonnegative valuation. We recall that
$\O_K$ is a local ring whose maximal ideal $\m_K$ consists of 
elements with positive valuation. The residue field $k$ of $K$
is, by definition, the quotient $\O_K/\m_K$. Under our assumptions,
$k$ is a finite field of characteristic $p$.

Let $W(k)$ denote the ring of Witt vectors with coefficients in $k$. Set 
$K_0 = \Frac W(k)$. By the general theory of Witt vectors, there exists 
a canonical embedding $K_0 \to K$. Moreover, through this embedding, 
$K$ appears as a finite totally ramified extension of $K_0$ of degree 
$e$. Therefore, $K_0$ is the maximal subextension of $K$ which is
unramified over $\Qp$.

\subsubsection{The absolute Galois group of $K$}

We choose and fix once for all an algebraic closure $\bar K$ of $K$. We 
recall that the valuation $v_p$ extends uniquely to $\bar K$, so that 
we can talk about the ring of integers $\O_{\bar K}$ of $\bar K$. 
This ring is a local ring whose maximal ideal will be denoted by
$\m_{\bar K}$. The quotient $\O_{\bar K}/\m_{\bar K}$ is identified
with an algebraic closure of $k$; it will be denoted $\bar k$ in
the sequel.

Let $G_K = \Gal(\bar K/K)$ be the absolute Galois group of $K$. Any 
element of $G_K$ acts by isometry on $\bar K$ and therefore stabilizes 
$\O_{\bar K}$ and $\m_{\bar K}$. It thus acts on the residue field 
$\bar k$. This defines a group homomorphism $G_K \to \Gal(\bar k/k)$,
which is surjective. 
The kernel of this morphism is the \emph{inertia} subgroup; we shall
denote it by $I_K$ in the sequel. The subextension of $\bar K$ 
cut out by $I_K$ is the maximal unramified extension of $K$; we
will denote it by $K^\ur$. Summarizing the above discussion, we
find that $G_K$ sits in the following exact sequence:
$$1 \longrightarrow I_K \longrightarrow G_K \longrightarrow
\Gal(\bar k/k) \to 1.$$
The structure of $\Gal(\bar k/k)$ is also known: if $k$ has cardinality 
$q$, $\Gal(\bar k/k)$ is the profinite group generated by the Frobenius 
$\Frob_q : x \mapsto x^q$.

The structure of $I_K$ can be further precised. Indeed a simple 
application of Hensel's lemma shows that any finite extension of 
$K^\ur$ whose degree is not divisible by $p$ has the form $K^\ur
[\sqrt[n]{\pi}]$. The union of all these extensions is called
$K^\tr$; it is the \emph{maximal tamely ramified} extension of~$K$.
Since $K^\ur$ contains all $n$-th roots of unity for $n \nmid p$
(cf the paragraph \emph{The cyclotomic extension} below for more
details),
the extension $K^\tr/K^\ur$ is Galois and its Galois group is
identified with
$\varprojlim_{n, p\nmid n} \Z/n\Z \simeq \prod_{\ell \neq p}
\Z_\ell$. Moreover, any finite extension of $K^\tr$ has degree
$p^m$ for some integer $m$.
On the Galois side, these properties imply that the closed subgroup of 
$I_K$ corresponding to the extension $K^\tr$ is the unique pro-$p$-Sylow 
of $I_K$ (which is then a normal subgroup) and that $I_K$ sits in the
following exact sequence:
$$1 \longrightarrow P_K \longrightarrow I_K \longrightarrow
\varprojlim_{n, p\nmid n} \Z/n\Z \to 1$$
where $P_K$ denotes the pro-$p$-Sylow of $I_K$.

\subsubsection{The cyclotomic extension}
\label{sssec:cyclo}

The cyclotomic extension of $K$ plays a quite important role in
$p$-adic Hodge theory. So we take some time to recall its most 
important properties.
Let $\mu_n \in \bar K$ be a primitive $n$-th root of unity. We recall 
that, by definition, the cyclotomic extension of $K$ is the subextension 
$K_\cycl$ of $\bar K$ generated by the $\mu_n$'s.

The extension $K(\mu_n)/K$ is Galois and its Galois group 
canonically embeds into $(\Z/n\Z)^\times$ through the map
$\chi_n : \Gal(K(\mu_n)/K) \to (\Z/n\Z)^\times$ defined by the 
relation $\chi_n(\mu_n) = \mu_n^{\chi_n(g)}$
for all $g \in \Gal(K(\mu_n)/K)$.
We draw the reader's attention to the fact that $\chi_n$ is in general
not surjective although it is for all $n$ when $K = \Qp$.

When $n$ is coprime with $p$, the extension $K(\mu_n)/K$ is
unramified since the polynomial $X^n - 1$ splits over $\bar k$.
In this case, $K(\mu_n)$ appears as a subextension of $K^\ur$. 
On the other hand, when $n= p^r$ is a power of $p$, the extension 
$K(\mu_{p^r})/K$ is totally ramified. This dichotomy motivates
the introduction of the two following infinite extensions of $K$:
$$K_{\ppcycl} = \bigcup_{n, p\nmid n} K(\mu_n)
\quad \text{and} \quad 
K_{\pcycl} = \bigcup_{r \geq 0} K(\mu_{p^r}).$$
The first one is actually equal to $K^\ur$ since, at the level of
residue fields, $\bar k$ is obtained by $k$ by adding all $p^n$-th 
roots of unity for $p \nmid n$.
As for $K_{\pcycl}$, it is linearly disjoint from $K^\ur$. It
is sometimes called the \emph{$p$-cyclotomic} extension of $K$.
Clearly, the cyclotomic extension of $K$ is the compositum of 
$K_\ur$ and $K_{\pcycl}$.

Let us review briefly the Galois properties of $K_{\pcycl}$. 
First of all, we notice that $K_{\pcycl}/K$ 
is Galois. Its Galois group is equipped with an injective group
homomorphism
$\chi_{p^\infty} : \Gal(K_{\pcycl}/K) \to \Zp^\times$
which is characterized by the relation
$g \mu_{p^m} = \mu_{p^m}^{\chi_{p^\infty}(g)}$
(for all $g \in \Gal(K_{\pcycl}/K)$ and $m \geq 1$).
Let $\chi_\cycl : G_K \to \Zp^\times$ be the homomorphism obtained by 
precomposing $\chi_{p^\infty}$ with the canonical surjection $G_K \to 
\Gal(K_{\pcycl}/K)$. We shall often see $\chi_\cycl$ as a character
and will call it the \emph{($p$-adic) cyclotomic character}. 
As $\chi_{p^\infty}$, it is determined by the relation:
$$g \mu_{p^m} = \mu_{p^m}^{\chi_\cycl(g)}
\quad \text{for all $g \in G_K$ and $m \geq 1$.}$$
By construction, the extension corresponding to $\ker \chi_\cycl$
is $K_{\pcycl}$ and, more generally, for all positive integer $r$,
the extension corresponding to $\ker (\chi_\cycl \mod {p^r})$ is
$K(\mu_{p^r})$.

The logarithm defines a group morphism $\Zp^\times \to \Zp$ where 
the group structure on the target is given by the addition. It sits
in the exact sequence:
\begin{equation}
\label{eq:exactseqlog}
1 \longrightarrow 
\Fp^\times \stackrel{[\cdot]}{\longrightarrow}
\Zp^\times \stackrel{\log}{\longrightarrow}
\Zp \longrightarrow 1
\end{equation}
where $[\cdot]$ denotes the Teichmuller representative function. This 
sequence is split since a retraction of $\Fp^\times \to \Zp^\times$ is 
simply the canonical projection. Therefore $\Zp^\times$ is canonically 
isomorphic to $\Fp^\times \times \Zp$.
Restricting \eqref{eq:exactseqlog} to the image of $\chi_\cycl$, we
find that $\Gal(K_{\pcycl}/K)$ sits in another exact sequence which 
reads as follows:
$$1 \longrightarrow 
H \longrightarrow
\Gal(K_{\pcycl}/K) \stackrel{\log \chi_\cycl}{\longrightarrow}
p^{r_0} \Zp \longrightarrow 1.$$
Here $r_0$ is a nonnegative integer and $H$ can be identified as a 
subgroup of $\Fp^\times$ and thus is cyclic of order divisible by
$p{-}1$. The above sequence splits, so that $\Gal(K_{\pcycl}/K)$
is canonically isomorphic to a direct product $H \times p^{r_0}\Zp
\simeq H \times \Zp$. 

The subextension of $K_{\pcycl}$ cut out by the factor $\Zp$ is nothing 
but $K(\mu_p)$. It is also the maximal tamely ramified subextension of 
$K_{\pcycl}$. 
The Galois group of $K_{\pcycl}/K(\mu_p)$ is canonically isomorphic to
$\Zp$ \emph{via} the additive character $p^{-r_0} \log \chi_\cycl$.
We say that $K_{\pcycl}/K(\mu_p)$ is a $\Zp$-extension. The fact that
$\Gal(K_{\pcycl}/K)$ splits as a direct product means that this 
extension descends to $K$; in particular, $K$ itself admits a 
$\Zp$-extension.

\subsubsection{Characters of $G_{\Qp}$}

The representation theory of $G_K$ is the main object of interest in 
this article. Among all representations of $G_K$, the simplest ones are 
of course characters, which are representations of dimension $1$. We 
have actually already seen an example of such character: the cyclotomic 
character $\chi_\cycl$. From $\chi_\cycl$, we can build the following 
other character:
$$\omega_\cycl :
G_K \stackrel{\chi_\cycl}{\longrightarrow} 
\Zp^\times \stackrel{\mod p}{\longrightarrow} 
\Fp^\times \stackrel{[\cdot]}{\longrightarrow} 
\Zp^\times$$
where the last map takes an element to its Teichmüller representative.
We observe that $\omega_\cycl$ is a finite order character, whose order
divides $p{-}1$. When $K = \Qp$, the order of $\omega_\cycl$ is
exactly $p{-}1$.

Another quite important family of characters are unramified characters,
that are those characters which are trivial on the inertia subgroup.
Since $G_K/I_K \simeq \Gal(\bar k/k)$ is procyclic, continuous 
unramified characters are easy to describe: they are all of the form
$$\mu_\lambda :
G_K \longrightarrow
G_K/I_K \simeq \Gal(\bar k/k) 
\stackrel{\Frob_q \mapsto \lambda}{\longrightarrow} \Zp^\times$$
for $\lambda$ varying in $\Zp^\times$.

Using local class field theory (\emph{cf} \cite{serre}), it is 
possible to describe explicitely all characters of $G_K$. Indeed
such characters all factor through the abelianization of $G_K$, which 
is closely related to $K^\times$ through the Artin reciprocity map.
When $K = \Qp$, this answer is given by the following proposition.

\begin{prop}
\label{prop:charGqp}
We assume $p > 2$.
Let $\chi$ be a character of $G_K$ with values in $\Qp^\times$.

\noindent
Then, there exist unique $\lambda \in \Zp^\times$, $a \in \Zp$ and
$b \in \Z/(p{-}1)\Z$ such that
$\chi = \mu_\lambda \cdot \chi_\cycl^a \cdot \omega_\cycl^b$.
\end{prop}

\begin{proof}
We first observe that, by compacity, $\chi$ must take its values
in $\Zp^\times$.
By the Kronecker--Weber theorem, we know that the maximal abelian
extension of $\Qp$ is the cyclotomic extension. Therefore $\chi$
has to factor through $\Gal(\Q_{p,\cycl}/\Qp)$. In particular,
$\chi_{|I_{\Qp}}$ factors through $\Gal(\Q_{p,\cycl}/\Qp^\ur)$
which is isomorphic to $\Zp^\times$ by the cyclotomic character.
Consequently, $\chi_{|I_{\Qp}} = h \circ \chi_\cycl$ for some
group homomorphism $h : \Zp^\times \to \Zp^\times$.
Moreover, when $p > 2$, we have an isomorphism $\Zp^\times
\simeq \Fp^\times \times \Zp$, $x \mapsto (x \mod p, \log x)$,
the inverse being given by $(a,b)\mapsto [a]{\cdot}\exp b$ where
$[a]$ denotes the Teichmuller representative of $a$. 
From this description, we derive that there exist $a \in \Zp$ 
and $b \in \Z/(p{-}1)\Z$ such that
$h(x) = x^a \cdot [x \mod p]^b$ for $x \in \Zp^\times$.
Thus $\chi_{I_{\Qp}} = \chi_\cycl^a \cdot \omega_\cycl^b$.
The character $\chi \cdot \chi_\cycl^{-a} \cdot \omega_\cycl^{-b}$
is then unramified. Thus it must be of the form $\mu_\lambda$ for
some $\lambda \in \Zp^\times$. The proposition is proved.
\end{proof}

\subsection{Motivations: $p$-divisible groups and étale cohomology}
\label{ssec:motivations}

The starting point of $p$-adic Hodge theory is Tate's paper of 1966 
on $p$-divisible groups~\cite{tate}. In this seminal article, Tate 
establishes a Hodge-like decomposition of the Tate module of a 
$p$-divisible group on $\O_K$.
More precisely, let $\G$ be a $p$-divisible group on
$\O_K$. 
We define the Tate module of $\G$ by $T_p \G = 
\varprojlim_n \G[p^n](\bar K)$. Observe that $T_p
\G$ is naturally endowed with an action of $G_K$. 
The algebraic structure of $T_p \G$ is well-known: it is
a free module of finite rank over $\Zp$. Set $V_p \G = 
\Qp \otimes_{\Zp} T_p \G$. Then, Tate proves the 
following Hodge-like $G_K$-equivariant decomposition:
\begin{equation}
\label{eq:tatehodge}
\Cp \otimes_{\Qp} V_p \G
\,\simeq\, \big(\Cp \otimes_{\O_K} \omega_{\G^\vee}\big)
\oplus \big(\Cp(\chi_\cycl^{-1}) \otimes_{\O_K} \omega_{\G}^\vee\big).
\end{equation}
Here $\G^\vee$ is the Cartier dual of $\G$, the
construction $\omega_{-}$ refers to the cotangent space at the
origin and $\Cp(\chi_\cycl^{-1})$ is $\Cp e$ endowed with the 
action $g(\lambda e) = g \lambda \cdot \chi_\cycl^{-1}(g) \cdot e$
(for $g \in G_K$ and $\lambda \in \Cp$). Note that the Galois action
is trivial over $\omega_{\G^\vee}$ and $\omega_{\G}^\vee$. The
isomorphism \eqref{eq:tatehodge} then reveals the Galois action
on the Tate module. Tate's theorem implies in particular that,
when $A$ is an 
abelian variety over $K$ with good reduction, 
the étale cohomology of $A$ admits the following decomposition:
\begin{equation}
\label{eq:hodgeabelian}
\Cp \otimes_{\Qp} H^1_\et(A_{\bar K}, \Qp) 
\, \simeq \, \big(\Cp \otimes_K H^1(A, \O_A)\big) \oplus
\big(\Cp(\chi_\cycl^{-1}) \otimes_K H^0(A, \Omega_{A/K})\big).
\end{equation}
were $A_{\bar K} = \Spec\:\bar K \times_{\Spec\:K} A$, $\O_A$ is
the structural sheaf of $A$ and $\Omega_{A/K}$ is the sheaf of Kähler
differentials of $A$ over $K$.
We refer to Freixas' lecture in this volume~\cite{freixas} for a 
more detailed discussion---including a sketch of the proof---about
Tate's theorem.

After Tate's results, $p$-divisible groups over various bases were 
studied intensively. In the 1970's, Fontaine~\cite{fontaine-pdiv} 
obtained a complete classification of $p$-divisible groups and finite 
flat group schemes over $\O_K$ when $K/\Qp$ is unramified.
The starting point of Fontaine's theorem is the classification 
of $p$-divisible groups over perfect fields of characteristic $p$ in terms of 
Dieudonné modules~\cite{demazure}. Let us recall briefly how it works. 
If $\G_k$ is a $p$-divisible group over $k$, we define
\begin{equation}
\label{eq:dieudonne}
M(\G_k) = \Hom_{\text{gr}} (\G_k, CW_k)
\end{equation}
where $CW_k$ is the functor of \emph{Witt covectors} and the
notations $\Hom_{\text{gr}}$ means that we are considering the set of 
all natural transformations preserving the group structure. The
space $M(\G_k)$ is a Dieudonné module. This means that it is a module 
over $W(k)$ endowed with a Frobenius $F$ (which is a semi-linear
endomorphism with respect to the Frobenius on $W(k)$) and a 
Verschiebung $V$ (which is a semi-linear endormorphism with respect
to the inverse of the Frobenius) with the property that $FV = VF = p$.
One can show that $M$ realizes an anti-equivalence of categories 
between the category of $p$-divisible groups over $k$ and that of
finite free Dieudonné modules over~$W(k)$, the inverse functor being
given by the formula
$$\G_k(A) = \Hom_{W(k),F,V}\big(M, CW(A)\big)
\quad \text{for any $k$-algebra $A$}$$
which is quite similar to \eqref{eq:dieudonne}.
Now, if $\G$ is a $p$-divisible over $\O_K$ with special fibre $\G_k$, 
Fontaine constructs a submodule $L(\G) \subset M(\G_k)$ and demonstrates 
that
it obeys to a certain list of properties. Taking these properties as axioms, 
Fontaine introduces the notion of \emph{Honda systems} and proves that
the association $\G \mapsto (M(\G_k), L(\G))$ is an anti-equivalence
of categories between the category of $p$-divisible groups over $\O_K$
and the category of finite free Honda systems over $W(k)$.
Moreover, Fontaine establishes a compact formula for the inverse
functor. This formula reads:
\begin{equation}
\label{eq:honda}
V_p \G = \Qp \otimes_{\Zp} \Hom_{\text{honda}} 
\big( (M(\G_k), L(\G)), \, (\mathcal B, L_{\mathcal B})\big)
\end{equation}
where the notation $\Hom_{\text{honda}}$ means that we are taking
the morphisms in the category of Honda system and the target
$(\mathcal B, L_{\mathcal B})$ is a special Honda system\footnote{Its 
construction is subtle and we will not
give it here. However, we would like to encourage the reader to
look at it in Fontaine's paper \cite[Chap. V, \S1]{fontaine-pdiv} 
because it is instructive to realize that it is actually quite
close to the construction of the periods $\BdR$ and $\Bcrys$
we shall detail in \S\ref{sec:dRcrys}.} (the letter $\mathcal B$
refers to the mathematician Barsotti, who first studied 
$p$-divisible groups using this kind of techniques).
Moreover, $(\mathcal B, L_{\mathcal B})$ is endowed with an
action of $G_K$, from which we can recover the $G_K$-action on
$V_p \G$.
Compared to Tate's decomposition formula \eqref{eq:tatehodge},
Fontaine's result is more precise because it describes the
Tate module $V_p \G$ itself, whereas Tate's result only concerns 
its scalar extension to $\Cp$.
For many complements about Fontaine's classification results, we
refer to~\cite{fontaine-pdiv,conrad}.

About ten years later, in 1981, Fontaine came back to Tate's 
decomposition isomorphism~\eqref{eq:hodgeabelian}and gave a different 
proof of it (which is sketched in Freixas' lecture in this 
volume~\cite{freixas}), relaxing at the same time the assumption of 
good reduction. He also became interested in generalizing Tate's 
decomposition theorem to higher cohomology group (\emph{i.e.} 
$H^r_\et(A_{\bar K}, \Qp)$ with $r > 1$) and other types of varieties.
Moreover, noticing that the right hand side of \eqref{eq:hodgeabelian} 
is the graded module of the de Rham cohomology, one may wonder if one 
can make the isomorphism \eqref{eq:hodgeabelian} more precise and relate 
the étale cohomology with the de Rham cohomology equipped with its 
filtration.
All these questions had been a strong motivation for the development of 
$p$-adic Hodge theory for many years. Nowadays, all of them are solved: 
it has been proved independently by Faltings~\cite{faltings} and 
Tsuji~\cite{tsuji} that $\BdR \otimes_{\Qp} H^r_\et(X_{\bar K}, \Qp) 
\simeq \BdR \otimes_{\Qp} H^r_\dR(X)$ whenever $X$ is a proper smooth 
variety over $K$ and $r$ is a nonnegative integer. Here $\BdR$ is the 
so-called \emph{field of $p$-adic periods}. We will introduce it in this 
article in \S\ref{sec:dRcrys}. Taking the grading in the above 
isomorphism, we get the following Hodge-like decomposition: 
$$\Cp \otimes_{\Qp} H^r_\et(X_{\bar K}, \Qp) \, \simeq \, 
\bigoplus_{a+b=r} \Cp(\chi_\cycl^{-a}) \otimes_K H^b(X, 
\Omega^a_{X/K}).$$
We will come back to these results in \S \ref{ssec:comparison}.

Finally, it is interesting to confront the two directions of research 
discussed above, namely classification of $p$-divisible groups and 
Hodge-like decomposition theorems.
As already mentioned, one important feature of the isomorphism
\eqref{eq:honda} is the fact that it gives a complete description
of the Galois action on the Tate module. 
On the other hand, it is apparent that Honda systems have important 
limitations: by design, they can only deal with Tate modules, that
is, roughly speaking, with the first cohomology group. Analyzing
carefully the situation, Fontaine realized that what is missing
to Honda systems is a good tannakian formalism (which is, of
course, a key point in the line of Hodge-like decomposition
theorems).
In more crude terms, the fact that we are limited to the $H^1_\et$ 
should be understood as a reflection of the fact that we are missing a 
good notion of tensor product on $p$-divisible groups.
As explained in the introduction of~\cite{fontaine-rennes}, the period 
ring~$\Bcrys$ and the afferent notion of crystalline representations 
actually emerge when trying to conceal the theory of Honda systems with 
the tannakian formalism inspired from the Hodge-like decomposition 
theorems we have presented above. 

All the developments we will present in the sequel are stamped by this 
simple idea that one wants to keep apparent the tannakian structure 
(\emph{i.e.} the tensor product) everywhere and, even, to use it as 
a main tool. The natural framework in which the theory grows is 
then that of Galois representations, which has a strong tannakian 
structure.

\subsection{Notion of semi-linear representations}

The Hodge-like decomposition theorems discussed previously motivate the 
study of representations of the form $W = \C_p \otimes_{\Q_p} V$ where 
$V$ is a given $\Q_p$-representation of $G_K$. Since $G_K$ does act on 
$\C_p$, we observe that $W$ is \emph{not} a $\C_p$-linear representation 
in the usual sense. Instead, it is a so-called \emph{semi-linear 
representation}. The aim of this subsection is to introduce and study
this notion.

\subsubsection{Definitions}

In what follows, we let $G$ be a topological group\footnote{In the 
application we have in mind, $G$ will be the absolute Galois group
of a $p$-adic field. However, for now, it is better to allow more
flexibility and let $G$ be an arbitrary topological group.} and 
$B$ be a topological ring equipped with a continuous\footnote{By 
continuous, we mean that the map $G \times B \to B, (g,x)\mapsto gx$ is 
continuous.} action of $G$, which is compatible with the ring structure, 
\emph{i.e.} $g \cdot (a+b) = ga + gb$ and $g \cdot (ab) = ga \: gb$ for 
all $g \in G$ and $a,b \in B$.

\begin{deftn}
A \emph{$B$-semi-linear representation of $G$} is the datum of a 
$B$-module $W$ equipped with a continuous action of $G$ such that:
$$g \cdot (x+y) = gx + gy 
\quad \text{and} \quad
g \cdot (ax) = ga \cdot gx$$
for all $g \in G$, $a \in B$ and $x,y \in W$.
\end{deftn}

Clearly, if $G$ acts trivially on $B$, the notion of $B$-semi-linear 
representation of $G$ agrees with the usual notion of $B$-linear 
representation of $G$.

By our assumptions, $B$ itself (endowed with its $G$-action) is a 
$B$-semi-linear representation of $G$. Similarly we can turn $B^n$ into 
a $B$-semi-linear representation by letting $G$ act coordinate by 
coordinate. The latter representation will be called the \emph{trivial 
representation} of dimension $n$.

If $W_1$ and $W_2$ are two $B$-semi-linear representations of $G$, a 
morphism $W_1 \to W_2$ is a $B$-linear mapping which 
commutes with the action of $G$. With this definition, we can form the 
category of $B$-semi-linear representations of $G$ (for $G$ and $B$ 
fixed). In the sequel we will simply denote it $\Rep_B(G)$. It is easily 
seen that $\Rep_B(G)$ is an abelian category. It is moreover endowed 
with a notion of tensor product and internal hom: if $W_1$ and $W_2$ are 
objects of $\Rep_B(G)$, then $W_1 \otimes_B W_2$ (equipped with the 
action $g \cdot (x\otimes y) = gx \otimes gy$) and $\Hom_B(W_1,W_2)$ 
(equipped with the action $g \varphi : x \mapsto g \varphi(g^{-1} x)$) 
are also.

\paragraph{Scalar extension}

There is also a natural notion of scalar extension in the framework of 
semi-linear representations. To explain it, let us consider a closed 
subring $C$ of $B$, which is stable under the action of $G$. Then the 
notion of $C$-semi-linear representations of $G$ makes sense and there 
is a canonical functor $\Rep_C(G) \to \Rep_B(G)$ taking $W$ to $B 
\otimes_C W$.

The latter construction is quite interesting because it allows us to 
build semi-linear representations from classical representations. 
Indeed, assume that we are given a field $E$ and we have chosen $G$ and 
$B$ is such a way that $B$ is an algebra over $E$ and $G$ acts trivially 
on $E$. (As an example, $B$ could be a Galois extension of $E$ with $G = 
\Gal(B/E)$.)
The scalar extension then defines a functor
$\Rep_E(G) \to \Rep_B(G)$.
Moreover, since the action of $G$ on $E$ is trivial, the category 
$\Rep_E(G)$ is just the category of $E$-linear representations of $G$. 
In more concrete terms, if $V$ is a classical representation of $G$ 
defined over $E$, then $W = B \otimes_E V$ is a $B$-semi-linear 
representation. This is actually the prototype of all the
semi-linear representations we are going to consider in this 
article.

Specializing the previous recipe to $1$-dimensional representations,
we obtain a way to construct semi-linear representations of $G$ from
characters of $G$.
Concretely, if $\chi : G \to E^\times$ is a multiplicative character,
we will denote by $B(\chi)$ the $1$-dimensional representation
generated by a vector $e_\chi$ on which $G$ acts by $g e_\chi =
\chi(g) \cdot e_\chi$ for all $g \in G$.
By semi-linearity, we then have
$$g (a \: e_\chi) = g a \cdot \chi(g) \: e_\chi$$
for all $g\in G$ and $a \in B$.

\subsubsection{Recognizing the trivial representation}
\label{sssec:trivialrep}

We keep the setup of the previous subsection: $G$ is a topological
group which acts continuously on a topological ring $B$.

\begin{deftn}
\label{def:trivialrep}
A $B$-semi-linear representation of $G$ is \emph{trivial} if
it is isomorphic to the trivial representation $B^d$ for some
positive integer $d$.
\end{deftn}

While it is in general easy to recognize when a linear representation is 
trivial (it suffices to check that $G$ acts trivially on each vector of 
the representation), the task becomes more complicated in the context of 
semi-linear representations.
Indeed, coming back to the definition, we see that a $B$-semi-linear
representation of $G$ is trivial if and only if it admits a basis of
vectors which are fixed by $G$. 
In particular, it is quite possible that a nontrivial semi-linear 
representation becomes trivial after scalar extension. The latter remark 
is in fact the starting point of Fontaine's strategy for classifying 
Galois representations.

We will discuss Fontaine's strategy in much more details in 
\S\ref{ssec:Fontainestrategy}.
Before this, we have to introduce further notations.
Given $W \in \Rep_B(G)$, we denote by $W^G$ the subset of $W$ consisting 
of fixed points under $G$, that is the subset of elements $x \in W$ such 
that $gx = x$ for all $g \in G$. Clearly $W^G$ is a module over $B^G$.
Moreover scalar extension provides a canonical morphism in $\Rep_B(G)$:
$$\alpha_W : B \otimes_{B^G} W^G \longrightarrow W.$$
This morphism is useful for recognizing trivial representations.
Indeed it is clearly an isomorphism when $W$ is trivial in the sense 
of Definition~\ref{def:trivialrep} (since $(B^d)^G = (B^G)^d$) and 
the converse also holds true when $W$ and $W^G$ are free of finite 
rank over $B$ and $B^G$ respectively.

\subsubsection{Hilbert's theorem 90}

As an introduction to Fontaine's strategy, we propose to discuss an
easy case where trivial semi-linear representations do appear, while
they were not expected at first glance.
The setting here is the following.
We assume that $B$ is a field and, in order to limit confusion, we will
call it $L$. We assume also that $G$ is a finite group, endowed with the 
discrete topology. Under these assumptions, $L^G$ is a subfield of $L$ 
and the extension $L/L^G$ is finite and Galois with Galois group~$G$.

\begin{theo}
\label{theo:H90}
We keep the notations and assumptions above.

\noindent
For all $W \in \Rep_L(G)$, the following assertions hold:

\vspace{-2mm}

\begin{enumerate}
\renewcommand{\itemsep}{0pt}
\item the morphism $\alpha_W$ is surjective,
\item if $W$ is finite dimensional over $L$, then 
$\alpha_W$ is an isomorphism, \emph{i.e.} $W$ is trivial.
\end{enumerate}
\end{theo}

\begin{proof}
Let $\lambda_1, \ldots, \lambda_n$ be a basis of $L$ over $L^G$.
By Artin's linear independence theorem, there exist constants
$\mu_1, \ldots, \mu_n \in L$ such that $\sum_{i=1}^n \mu_i 
g(\lambda_i)$ is $1$ if $g$ is the identity and $0$ otherwise.
Define the trace function $T : W \to W$ by $T(x) = \sum_{g \in 
G} gx$. One easily checks that $T$ takes its values in $W^G$.
Moreover, for the particular $\mu_i$'s we have introduced earlier, 
we have $\sum_{i=1}^n \mu_i T(\lambda_i x) = x$ for all $x \in W$.
This shows the surjectivity of $\alpha_W$.

We now assume that $W$ is finite dimensional over $L$. 
The proof of injectivity is quite similar to the proof of Artin's
linear independence theorem. It is enough to check that every finite 
family of elements of $W^G$ which is linearly independent over $L^G$ 
remains linearly independent over $L$. Let then $(x_1, \ldots, x_m)$
be a linearly independent family over $L^G$ with $x_i \in W^G$ for
all $i$. We assume by contradiction that there exists a nontrivial
relation of linear dependance of the form:
\begin{equation}
\label{eq:lindep}
a_1 x_1 + a_2 x_2 + \cdots + a_m x_m = 0
\end{equation}
with $a_i \in L$. 
We choose such a relation so that the number of 
nonzero $a_i$'s is minimal. Up to reindexing the $a_i$'s and rescaling 
the relation, we may assume that $a_1 = 1$. Let $g \in G$.
Applying $(g-\id)$ to \eqref{eq:lindep}, we get the relation
$(ga_2 - a_2) x_2 + \cdots + (ga_m - a_m) x_m = 0$
which is shorter than \eqref{eq:lindep}. 
From our minimality assumption, we deduce that $ga_i = a_i$ for all 
$i \geq 2$. Since this is valid for all $g \in G$, we deduce that
the linear dependance relation \eqref{eq:lindep}
has coefficients in $L^G$. This is a contradiction since we have
assumed that the family $(x_1, \ldots, x_m)$ is linearly independent
over~$L^G$.
\end{proof}

\begin{rem}
Theorem~\ref{theo:H90} is often referred to as Hilbert's theorem 90. The 
reason is that it can be rephrased in the language of group cohomology, 
then asserting that $H^1(G, \GL_d(L))$ is reduced to one element. This 
latter statement is an extension of the classical Hilbert's theorem 90 
to higher $d$.
\end{rem}

\begin{ex}
We emphasize that Theorem~\ref{theo:H90} does not hold in general
when $G = \Gal(L/K)$ where $L/K$ is an \emph{infinite} extension and
$G$ is equipped with its natural profinite topology.
As an example, take $G = 
G_{\Qp} = \Gal(\Qpbar/\Qp)$ and let it act on $L = \Qpbar$. The fixed 
subfield $L^G$ is $\Qp$. Consider the semi-linear representation 
$\Qp(\chi_\cycl)$ where we recall that $\chi_\cycl$ denotes the 
cyclotomic character $\Gal(\Qpbar/\Qp) \to \Zp^\times \subset 
\Qp^\times$. We claim that $\Qp(\chi_\cycl)$ is not isomorphic to $\Qp$ 
in the category $\Rep_{\Qpbar}(G_{\Qp})$. Indeed, assume by contraction that 
there exists a $G$-equivariant isomorphism $\Qpbar \simeq 
\Qpbar(\chi_\cycl)$. Then there should exist an element $x \in \Qpbar$ 
such that
\begin{equation}
\label{eq:pi}
g x = \chi(g) \: x \quad \text{for all } g \in G_{\Qp}.
\end{equation}
Since $x$ is in $\Qpbar$, it belongs to a finite extension $L$
of $\Qp$. Let $N_{L/\Qp} : L \to \Qp$ be the norm map from $L$
to $\Qp$. Applying it to \eqref{eq:pi}, we get the relation
$N_{L/\Qp}(x) = \chi(g)^{[L:\Qp]} \cdot N_{L/\Qp}(x)$. Since
$N_{L/\Qp}(x)$ does not vanish, we end up with $\chi(g)^{[L:\Qp]} 
= 1$ for all $g \in G_{\Qp}$, which is a contradiction.
\end{ex}

\begin{rem}
Similarly, we shall see later (\emph{cf} Proposition~ 
\ref{prop:HTweights}) that $\Cp$ is not isomorphic to $\Cp(\chi_\cycl)$ 
in the category $\Rep_{\Cp}(G_{\Qp})$.
\end{rem}

\subsection{Fontaine's strategy}
\label{ssec:Fontainestrategy}

We are now ready to explain the general principles of Fontaine's 
strategy for isolating the most interesting representations of the 
Galois group of a $p$-adic field and studying them.
The material presented in this subsection comes from~\cite[Chap. 
II]{fontaine-periodes}.

As before, let $G$ be a topological group. Let also $E$ be a fixed 
topological field. We consider a topological $E$-algebra $B$ on which 
$G$ acts and assume that the $G$-action on $E$ is trivial. Under our 
assumptions, the category $\Rep_E(G)$ is the category of $E$-linear 
representations of~$G$.

\begin{rem}
In fact, in what follows, the topology on $B$ will play no role
since all the forthcoming definitions and results will be purely
algebraic.
Nevertheless, we prefer keeping the datum of the topology on $B$
as it is more natural and all the rings $B$ we shall consider later
on will come equipped with a canonical topology.
\end{rem}

The following definition is due to Fontaine.

\begin{deftn}
Let $V \in \Rep_E(G)$ be finite dimensional over $E$.

\noindent
We say that $V$ is \emph{$B$-admissible} if the $B$-semi-linear
representation $B \otimes_E V$ is trivial.
\end{deftn}

We denote by $\Rep_E^{B\aadm}(G)$ the full subcategory of $\Rep_E(G)$ 
consisting of finite dimensional representations of $E$ which are 
$B$-admissible. It is easy to check that $\Rep_E^{B\aadm}(G)$ is stable 
by direct sums, tensor products, and duals.
Moreover the association $B \mapsto \Rep_E^{B\aadm}(G)$ is increasing 
in the following sense: any 
$B_1$-admissible representation is automatically $B_2$-admissible
as soon as $B_2$ appears as an algebra over $B_1$.

\begin{ex}
Let $L$ be a finite extension of $E$. Take $G = \Gal(L/E)$ and let
it act naturally on $L$. Hilbert's theorem 90 (\emph{cf} Theorem
\ref{theo:H90}) shows that all finite dimension $E$-representation
of $G$ is $L$-admissible.
\end{ex}

\subsubsection{A criterium for $B$-admissibility}
\label{sssec:criterium}

The aim of this paragraph is to establish a numerical criterium
for recognizing $B$-admissible representations. In order to do so,
we make the following assumptions\footnote{Our assumptions are a bit 
stronger than Fontaine's ones. We chose these stronger hypothesis 
because they simplify the exposition and are sufficient for the 
applications we want to discuss here.} on the $E$-algebra $B$:

\vspace{-2mm}

\begin{enumerate}[(H1)]
\renewcommand{\itemsep}{0pt}
\item $B$ is a domain,
\item $(\Frac B)^G = B^G$,
\item if $b \in B$, $b \neq 0$ and the $E$-line $Eb$ is stable
under $G$, then $b \in B^\times$.
\end{enumerate}

\noindent
It is easily seen that the assumption (H3) implies that $B^G$
is a field. Indeed for any $b \in E$, $b \neq 0$, the line $Eb$
is clearly stable under $G$. Thus $b$ has to be invertible in 
$B$. Now we conclude by noticing that its inverse is also fixed
by $G$.
Moreover, by copying the proof of the second part of 
Theorem~\ref{theo:H90}, one shows that the assumptions (H1) and
(H2) ensure that the morphism
$\alpha_W : B \otimes_{B^G} W^G \to W$
is injective for all $W \in \Rep_B(G)$ which are free of finite
rank over $B$. In particular, this property holds true for $W$ of
the form $B \otimes_E V$ where $V$ is finite dimensional $E$-linear
representation of $G$.

\begin{prop}
We assume that $B$ satisfies (H1), (H2) and (H3)

\noindent
Let $V \in \Rep_E(G)$ and set $W = B \otimes_E V$.
We assume that $V$ is finite dimensional over $E$.
Then the following assertions are equivalent:

\vspace{-2mm}

\begin{enumerate}[(i)]
\renewcommand{\itemsep}{0pt}
\item $W$ is trivial,
\item the morphism $\alpha_W$ is an isomorphism,
\item $\dim_{B^G} W^G = \dim_E V$.
\end{enumerate}
\end{prop}

\begin{proof}
Since $B^G$ is a field, the equivalence between (i) and (ii)
is obvious. Moreover the fact that (ii) implies (iii) is also
clear. We then just have to prove that (iii) implies (ii). 

We assume (iii) and denote by $d$ the common dimension of $V$
over $E$ and $W^G$ over $B^G$. The morphism
$\alpha_W : B \otimes_{B^G} W^G \to B \otimes_E V$
is a $B$-linear morphism between two finite free $B$-modules of
rank $d$. It is then enough to prove that its determinant 
is an isomorphism.
Let $v_1, \ldots, v_d$ be a $E$-basis of $V$ and let $w_1, \ldots, 
w_d$ be a $B^G$-basis of $W^G$. Let $b$ be the unique element
of $B$ such that:
\begin{equation}
\label{eq:delta}
\alpha_W(v_1) \wedge \cdots \wedge \alpha_W(v_d)
= b \cdot w_1 \wedge \cdots \wedge w_d.
\end{equation}
From the injectivity of $\alpha_W$, we derive $b \neq 0$. 
Let now $g \in G$. Applying $g$ to \eqref{eq:delta}, we get
$g b = \eta \cdot b$ where $\eta$ is defined by the
identity
$\alpha_W(gv_1) \wedge \cdots \wedge \alpha_W(gv_d) = \eta \cdot
\alpha_W(v_1) \wedge \cdots \wedge \alpha_W(v_d)$. From the fact 
that the $E$-span of $v_1, \ldots, v_d$ (which is $V$) is stable under 
the action of $G$, we deduce that $\eta$ lies in $E$. Hence $g b
\in Eb$. Consequently, the $E$-line $Eb$ is stable by the 
action of $G$. Thanks to hypothesis (H3), we conclude that $b \in 
B^\times$ as wanted.
\end{proof}

\begin{cor}
Under the assumptions (H1), (H2) and (H3), the category 
$\Rep_E^{B\aadm}(G)$ is stable by subobjects and quotients.
\end{cor}

\begin{proof}
Assume that we are given an exact sequence $0 \to V_1 \to V \to V_2 
\to 0$ in the category $\Rep_E(G)$ and assume that $V$ is 
$B$-admissible. Tensoring by $B$ and taking the $G$-invariants,
we obtain the exact sequence
$0 \to (B \otimes_E V_1)^G \to (B \otimes_E V)^G \to
(B \otimes_E V_2)^G$
from which we derive the inequality:
\begin{equation}
\label{eq:dimgeq}
\dim_{B^G} (B \otimes_E V)^G 
 \geq \dim_{B^G} (B \otimes_E V_1)^G + \dim_{B^G} (B \otimes_E V_2)^G.
\end{equation}
Moreover we know that $\dim_{B^G} (B \otimes_E V_i)^G \leq \dim_E V_i$
for $i \in \{1,2\}$. Therefore we get:
\begin{equation}
\label{eq:dimleq}
\dim_{B^G} (B \otimes_E V)^G \leq \dim_E V_1 + \dim_E V_2 = \dim_E V.
\end{equation}
We know also that $\dim_{B^G} (B \otimes_E V)^G = 
\dim_E V$ thanks to the $B$-admissibility of $V$. Combining the
inequalities~\eqref{eq:dimgeq} and~\eqref{eq:dimleq}, we find that
$\dim_{B^G} (B \otimes_E V_i)^G$ has to be equal to $\dim_E V_i$, 
which proves that $V_i$ (for $i \in \{1,2\})$ is $B$-admissible.
\end{proof}

\subsubsection{What's next?}

Until now, we have spent a lot of time at defining a general abstract 
formalism whose main achievement is the notion of $B$-admissibility.
This is certainly nice but still seems to be quite far from the 
applications.
We devote this subsection to our readers who are impatient to connect 
the notion of $B$-admissibility to concrete properties of Galois
representations and cohomology of algebraic varieties.

In the sequel, we will often use the locution \emph{period rings} to 
refer to various rings $B$. This terminology is motivation by the role 
those rings $B$ play in geometry (they often appear in comparison 
theorem between various cohomologies).

\smallskip

From now, we go back to the setting of \S \ref{ssec:setting}.
Precisely, we let $K$ be a finite extension of $\Qp$. We let
$\bar K$ denote a fixed algebraic closure of $K$ and we set
$G_K = \Gal(\bar K/K)$. Let $\Cp$ be the completion of $\bar K$.
The action of $G_K$ on $\bar K$ extends to a continuous action 
of $G_K$ on $\Cp$. 
Finally, we recall that $\chi_\cycl : G_K \to \Zp^\times$ denotes the
$p$-adic cyclotomic character of $G_K$.

\paragraph{$\Cp$-admissible representations.}

The first ring of periods we will consider is $\Cp$ itself,
equipped with the $p$-adic topology and its natural action of
$G_K$. The question of $\Cp$-admissibility of representations
of $G_K$ will be studied in details in \S\ref{sec:Cp};
we will notably prove the following result (\emph{cf} 
Theorem~\ref{theo:Cpadm}).

\begin{theo}
\label{theointro:Cpadm}
Let $V$ be a $\Qp$-linear finite dimensional representation
of $G_K$. Then $V$ is $\Cp$-admissible if and only if the
inertia subgroup of $G_K$ acts on $V$ through a finite
quotient.
\end{theo}

\noindent
In other words, $\Cp$-admissibility detects those representations 
which are potentially unramified. This notion has then a strong
arithmetical meaning.

\paragraph{Hodge--Tate representations.}

Theorem~\ref{theointro:Cpadm} shows that the notion of 
$\Cp$-admissibility is too strong and does not capture all interesting 
representations; for instance, the cyclotomic character is \emph{not} 
$\Cp$-admissible.
A larger class of representations is given by the notion of Hodge--Tate 
representations. By definition, a $\Qp$-linear representation of $G_K$ 
is Hodge--Tate if $\Cp \otimes_{\Qp} V$ decomposes as:
\begin{equation}
\label{eq:HT}
\Cp \otimes_{\Qp} V = 
\Cp(\chi_\cycl^{n_1}) \oplus
\Cp(\chi_\cycl^{n_2}) \oplus \cdots \oplus
\Cp(\chi_\cycl^{n_d})
\end{equation}
for some integers $n_1, \ldots, n_d$.
The condition actually fits very well in the framework of
$B$-admissibility as introduced above. Indeed, set 
$B_\HT = \Cp[t, t^{-1}]$ ($\HT$ stands for Hodge--Tate) and let
$G_K$ act on it by the formula
$g \cdot (a t^i) = g a \cdot \chi_\cycl(g)^i \cdot t^i$
for $g \in G_K$, $i \in \Z$ and $a \in \Cp$. One checks that $V$ is 
Hodge--Tate if and only if it is $B_\HT$-admissible.
Besides, Theorem~\ref{theointro:Cpadm} is the starting point for 
studying Hodge--Tate representations. For example, it implies that the 
integers $n_i$'s that appeared in \eqref{eq:HT} are uniquely determined 
up to permutation (\emph{cf} Proposition \ref{prop:HTweights}). 
They are called the \emph{Hodge--Tate weights} of 
the representation~$V$.

Finally, Hodge-like decomposition theorems show that many 
representations coming from geometry are Hodge--Tate. This class of 
representations then seems particularly interesting.

\paragraph{De Rham and crystalline representations.}

Unfortunately, Hodge--Tate representations have several defaults.
First, they are actually too numerous and, for this reason, it is
difficult to describe them precisely and design tools to work with
them efficiently. 
The second defect of Hodge--Tate representations is of geometric nature. 
Indeed, tensoring the étale cohomology with $\Cp$ (or equivalently, with 
$B_\HT$) captures the \emph{graded} module of the de Rham cohomology. 
However, it does not capture the entire complexity of de Rham 
cohomology, the point being that the de Rham filtration does not split 
canonically in the $p$-adic setting.

In order to work around this issues, Fontaine defined other period 
rings $B$ ``finer'' than $B_\HT$.
The most classical period rings introduced by Fontaine are 
$\Bcrys \subset B_\st \subset \BdR$; the corresponding admissible
representations are called \emph{crystalline}, \emph{semi-stable}
and \emph{de Rham} respectively. Moreover, $\BdR$ is a filtered field 
whose graded ring can be canonically identified with $B_\HT$. This 
property, together with the aforementionned inclusions, imply the 
following implications:
$$\text{crystalline} \,\Longrightarrow\,
\text{semi-stable} \,\Longrightarrow\,
\text{de Rham} \,\Longrightarrow\,
\text{Hodge--Tate.}$$
In~\S\ref{sec:dRcrys}, we will discuss the construction of $\BdR$ and 
$\Bcrys$, while the arithmetical and geometrical interest of these 
refined period rings will be presented in~\S\ref{sec:crysdRrep}.
Rapidly, let us say here that representations coming from the geometry,
\emph{i.e.} of the form $H^r_\et(X_{\bar K}, \Qp)$ where $X$
is a smooth projective algebraic variety over $\Qp$, are all
de Rham.
By definition, this means that the space
$\big(\BdR \otimes_{\Qp} H^r_\et(X_{\bar K}, \Qp)\big)^{G_K}$
has the correct dimension. It turns out that this space has a
very pleasant cohomological interpretation: it is canonically
isomorphic to the de Rham cohomology of $X$, namely $H^r_\dR(X)$.
We thus get an isomorphism:
$$B_\dR \otimes H^r_\dR(X)
\stackrel\sim\longrightarrow
B_\dR \otimes H^r_\et(X_{\bar K}, \Qp)$$
which is a the right $p$-adic analogue of the de Rham comparison 
theorem.
Besides, the de Rham filtration on $H^r_\dR(X)$ can be rebuilt from the 
filtration on $\BdR$. This is the first apparition of the yoga of 
additional structures, which actually is ubiquitous in $p$-adic Hodge 
theory.

The introduction of $\BdR$ resolves elegantly the geometric 
issue we have pointed out earlier. 
However, the class of $\BdR$-admissible representations is still rather 
large and not easy to describe. The ring $\Bcrys$ is a subring of 
$\BdR$ which is equipped with more structures and provides very 
powerful 
tools for describing crystalline representations. On the geometric
side, crystalline representations correspond to the étale cohomology
of varieties with good reduction and the space
$\big(\Bcrys \otimes_{\Qp} H^r_\et(X_{\bar K}, \Qp)\big)^{G_K}$
is related to the crystalline cohomology of (the special fibre of a
proper smooth model of) $X$, equipped with its Frobenius endomorphism.
All in all, we will obtain powerful methods for describing the étale
cohomology of $X$ with comparatively down-to-earth objects.


\section{The first period ring: $\Cp$}
\label{sec:Cp}

After Tate and Fontaine's results on Hodge-like decompositions of 
cohomology, the first natural period ring to consider is $\Cp$ itself. 
In this section, we first study $\Cp$-admissibility and prove 
Theorem~\ref{theointro:Cpadm}. The proof requires some preparation and 
occupies the first two subsections. The last subsection 
(\S\ref{ssec:Sen}) is devoted to expose Sen's theory, as developed in 
\cite{sen}, whose aim is to go further than $\Cp$-admissibly and 
classify general $\Cp$-semi-linear representations.

Our approach differs a bit from usual presentations in that we will not 
use the langage of group cohomology but instead will keep working with 
(semi-linear) representations and vectors throughout the exposition.

\subsection{Ramification in $\Zp$-extensions}

A first important ingredient in the proof of Theorem 
\ref{theointro:Cpadm} is the study of ramification in Galois extensions 
over $K$ whose Galois group is isomorphic to $\Zp$.

Throughout this section, if $L$ is an algebraic extension of $K$, we 
shall denote by $\O_L$ its ring of integers, by $\m_L$ the maximal ideal 
of $\O_L$ and by $k_L = \O_L/\m_L$ its residue field.
If moreover $L/K$ is finite, we shall denote by $v_L : L \to \Z \cup 
\{+\infty\}$ the valuation on $L$ normalized by $v_L(L^\times) = \Z$.
Set $e_L = v_L(p)$; it is the ramification index of the extension
$L/\Qp$.
If $F$ and $L$ are two algebraic extensions of $K$ with $F \subset L$
and $[L:F] < \infty$, we shall use the notation $e_{L/F}$ for the
ramification index of $L/F$ and the notation $\Tr_{L/F}$ for the 
trace map of $L$ over $F$. When $F$ is a finite extension of $K$,
we have $e_{L/F} = \frac{e_L}{e_F}$.

In what follows, it is convenient to allow flexibility and work 
over a base $F$ which is not necessarily $K$ but a finite extension of
it.

\subsubsection{Higher ramification groups}

We first recall briefly the theory of higher ramification groups as 
exposed in \cite{serre}.

Let $L/F$ be a finite Galois extension with Galois group $G$.
For any nonnegative integer $i$, we define the $i$-th higher group of 
ramification of $L/K$ as the subgroup $G_i$ of $G$ consisting of 
elements $g \in G$ such that $g$ acts trivially on the quotient
$\O_L/\m_L^{i+1}$.
One easily checks that the $G_i$'s form a nonincreasing sequence of
normal subgroups of $G$ and that $G_0$ is the inertia subgroup of
$G$. Besides, one proves that
the quotient $G_0/G_1$ naturally embeds into $k_F^\times$ and 
thus is a cyclic of order prime to $p$, and, for $i > 0$, the 
quotient $G_i/G_{i+1}$ embeds into 
$\m_L^i/\m_L^{i+1}$ and thus is a commutative $p$-group killed
by $p$.

The ramification filtration we have just defined is not compatible 
with subextensions. However, we can recover some compatibility after
a suitable reindexation. In order to do so, we first
define the Herbrand function $\varphi_{L/F} : \R^+ \to \R^+$ by:
\begin{equation}
\label{eq:phi}
\varphi_{L/F}(u) = \frac 1 {e_{L/F}} \cdot 
\int_0^u \Card G_t \cdot dt
\end{equation}
where we agree that $G_t = G_{\lceil t \rceil}$ when $t$ is a 
nonnegative real number and $\lceil \cdot \rceil$ is the ceiling
function. Clearly $\varphi_{L/F}$ is increasing and defines a
bijection from $\R^+$ to itself. Let $\psi_{L/F}$ be its inverse.
For $u \in \R^+$, we set $G^u = G_{\psi_{L/F}(u)}$.
One can then show the following property. If $L_1$ and $L_2$ are
two finite Galois extensions of $K$ with $L_1 \subset L_2$, then
the canonical projection $\Gal(L_2/F) \to \Gal(L_1/F)$ maps
surjectively $\Gal(L_2/F)^u$ onto $\Gal(L_1/F)^u$ for all $u \in
\R^+$.
This compatibility allows us the define ramification subgroups
for infinite extensions: if $L$ is a infinite Galois extension
of $K$, we put $\Gal(L/F)^u = \varprojlim_F \Gal(L'/F)^u$
where $L'$ runs over all finite extensions of $F$ included in $L$.

Moreover the $\varphi$'s and $\psi$'s functions verify very pleasant
composition formulae: if $F$, $L_1$ and $L_2$ are extensions of $K$
as above, we have $\varphi_{L_2/F} = \varphi_{L_1/F} \circ
\varphi_{L_2/L_1}$ and thus, passing to the inverse, $\psi_{L_2/F}
= \psi_{L_2/L_1} \circ \psi_{L_1/F}$.

Finally, we observe that the knowledge of the upper numbering of the 
ramification filtration is equivalent to that of the lower numbering. 
Indeed the function $\psi_{L/F}$ can be recovered for the $G^u$'s thanks 
to the formula:
\begin{equation}
\label{eq:psi}
\psi_{L/F}(t) = e_{L/F} \cdot \int_0^u \frac{du}{\Card G^u}.
\end{equation}
Now, taking the inverse of $\psi_{L/F}$, we reconstruct the 
function $\varphi_{L/F}$ and we can finally recover the lower
numbering of the filtration ramification by letting $G_t = 
G^{\varphi_{L/F}(t)}$.

\paragraph{Relation with the different.}

We will often use the higher ramification groups in order to compute (or 
estimate) the different of the extension $L/F$. Let us first recall that 
the different $\D_{L/F}$ of $L/F$ is the ideal of $\O_L$ characterized
by the property that $\Tr_{L/F}(x \O_L) \subset \O_F$ if and only if
$x \in \D_{L/F}^{-1}$, \emph{i.e.} $v_p(x) + v_p(\D_{L/F}) \geq 0$.
Here we have introduced the valuation of an ideal, which is defined
as the minimal valuation of one of its elements (or, equivalently, if
it is generated by one element, the valuation of any generator).
The following formula relates $\D_{L/F}$ with the size of the $G_i$'s:
$$v_L(\D_{L/F}) = \sum_{i\geq 0} (\Card G_i - 1).$$
From this relation, we derive easily the following formulae:
\begin{equation}
\label{eq:difframif}
v_F(\D_{L/F}) 
= \lim_{t \to \infty} \left(\varphi_{L/F}(t) - \frac t{e_{L/F}}\right)
= \lim_{u \to \infty} \left(u - \frac {\psi_{L/F}(u)}{e_{L/F}}\right).
\end{equation}

\paragraph{Ramification and class field theory.}

When the extension $L/F$ is finite and 
abelian, local class field theory gives a nice interpretation
of the ramification filtration. More precisely, recall first
that Artin reciprocity map provides an isomorphism
$\Gal(L/F) \simeq F^\times / N_{L/F}(L^\times)$
where $N_{L/F}$ is the norm of $L$ over $F$. 
Under this isomorphism, the ramification subgroup $\Gal(L/F)^u$ 
corresponds to the image in $F^\times / N_{L/F}(L^\times)$ of
the congruence subgroup:
$$U_F^u = \big\{ \, x \in \O_F^\times \quad \text{s.t.} \quad
x \equiv 1 \pmod{\m_F^u} \, \big\} \, \subset \, F^\times.$$
A nontrivial consequence of this result is the Hasse--Arf
theorem which states that the jumps of the filtration ramification
in upper numbering (\emph{i.e.} the real numbers $u$ for which 
$G^{u+\varepsilon} \neq G^u$ for all $\varepsilon > 0$) are all
integers.

\subsubsection{The case of $\Zp$-extensions}
\label{sssec:Zpext}

We now consider a Galois extension $F_\infty$ of $F$. We assume that 
$F_\infty/F$ is ramified and that we are given an isomorphism $\alpha : 
\Gal(F_\infty/F) \simeq \Zp$. We remark that an extension with these 
properties always exists; it can be cooked up from the cyclotomic 
extension of $F$ as discussed in \S\ref{sssec:cyclo}.

For $r \geq 0$, let $\gamma_r = \alpha^{-1}(p^r) \in \Gal(F_\infty/F)$
and $F_r$ be the finite extension of $F$ cut out by the closed 
subgroup generated by $\gamma_r$ (that is the subgroup $\alpha^{-1}
(p^r\Zp)$). The $F_r$'s then form a tower of 
extensions, in which each $F_{r+1}/F_r$ is a cyclic extension of order 
$p$. More generally, if $s \geq r$, the extension $F_s/F_r$ is cyclic of 
order $p^{s-r}$ and its Galois group is generated by the class of~$\gamma_r$.

Let also $e_F$ be the absolute index of ramification of $F$ defined
as $e_F = v_F(p)$.

\begin{prop}
\label{prop:ramZp}
With the previous notations,
there exists $a \in \Z$ such that, for $u$ large enough,
$\Gal(F_\infty/F)^u $ is the closed subgroup generated by
$\gamma_{\lceil \frac{u-a}{e_F} \rceil}$.
\end{prop}

\begin{proof}
For $u \in \R^+$, let $\rho(u)$ be the unique element $r$ of $\N
\cup \{+\infty\}$ for which
$\Gal(F_\infty/F)^u$ is topologically generated by $\gamma_r$. This 
definition yields a function $\rho : \R^+ \to \N \cup \{+\infty\}$ 
which is nondecreasing and left-continuous. The fact that $F_\infty/F$ 
is ramified shows that $\rho(0)$ is finite. By the Hasse--Arf 
theorem, the points of discontinuity of $\rho$ are all integers. 
Moreover there must be infinitely many of them since the successive
quotients of the ramification filtration are all killed by $p$. This 
implies that $\rho$ takes finite values everywhere.

Let $s$ be a positive integer.
By local class field theory, we know that Artin's isomorphism 
$\Gal(F_s/F) \simeq F^\times / N_{F_r/F}(F_r^\times)$ maps the subgroup 
$\Gal(F_s/F)^u$ onto $U_F^u / (U_F^u \cap N_{F_r/F}(F_r^\times))$.
Note that the group $\Gal(F_s/F)^u$ is generated by the class of
$\gamma_{\rho(u)}$. Its subgroup of $p$-th powers is then generated
by $\gamma_{\rho(u)+1}$. On the other hand, a simple computation shows
that the subgroup of $p$-th powers of $U_F^u$ is equal to $U_F^{u+e_F}$ 
as soon as $u > \frac{e_F}{p-1}$.
Comparing the subgroup of $p$-th powers of both sides, we obtain:
$$\min(s,\, \rho(u) + 1) = \min(s,\, \rho(u+e_F))$$
whenever $u > \frac{e_F}{p-1}$. Letting $s$ go to infinity, we end up 
with $\rho(u + e_F) = \rho(u) + 1$ for $u > \frac{e_F}{p-1}$.
This relation, combined with the facts that $\rho$ is nondecreasing,
left-continuous and takes integral values, implies that there exists
a real constant $a$ such that $\rho(u) =
\lceil \frac{u-a}{e_F} \rceil$ for $u > \frac{e_F}{p-1}$.
The fact that $a$ is indeed an integer is a consequence of the
Hasse--Arf theorem.
\end{proof}

\begin{figure}
\hfill
\begin{tikzpicture}[xscale=2,yscale=0.4]
\draw[->] (0,0)--(4.2,0);
\draw[->] (0,0)--(0,17);
\draw[thick] (0,0)--(1,1)--(2,3)--(3,7)--(4,15)--(4.125,17);
\begin{scope}[dotted]
\draw (1,0)--(1,1)--(0,1);
\draw (2,0)--(2,3)--(0,3);
\draw (3,0)--(3,7)--(0,7);
\draw (4,0)--(4,15)--(0,15);
\end{scope}
\node[scale=0.8,below] at (1,0) { $1$ };
\node[scale=0.8,below] at (2,0) { $2$ };
\node[scale=0.8,below] at (3,0) { $3$ };
\node[scale=0.8,below] at (4,0) { $4$ };
\node[scale=0.8,left] at (0,1) { $1$ };
\node[scale=0.8,left] at (0,3) { $p+1$ };
\node[scale=0.8,left] at (0,7) { $p^2+p+1$ };
\node[scale=0.8,left] at (0,15) { $p^3+p^2+p+1$ };
\end{tikzpicture}
\hfill\null

\caption{The graph of the function $\psi_r$ ($r \geq 4$)}
\label{fig:psir}
\end{figure}
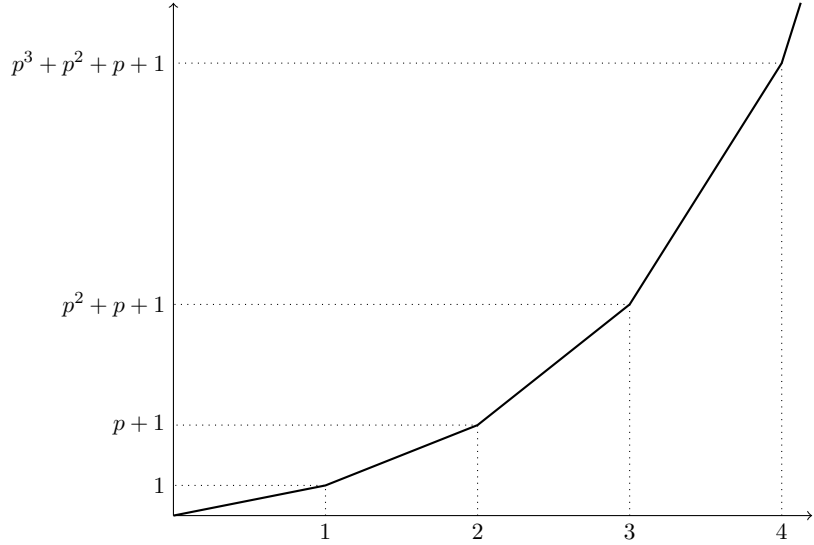

\begin{rem}
\label{rem:ramZp}
Using formula \eqref{eq:psi}, one can rephrase Proposition 
\ref{prop:ramZp} as follows. For a positive integer $r$, let 
$\psi_r : \R^+ \to \R^+$ be the function defined by:
$$\begin{array}{r@{\hspace{0.5ex}}ll}
\psi_r(u) 
 & = u 
 & \text{if } 0 \leq u < 1 \\
 & = p (u - 1) + 1
 & \text{if } 1 \leq u < 2 \\
 & \hspace{1em}\vdots \\
 & = p^{r-1} (u - r + 1) + (1 + p + \cdots + p^{r-2}) 
 & \text{if } r-1 \leq u < r \\
 & = p^r (u - r) + (1 + p + \cdots + p^{r-1})
 & \text{if } u \geq r
\end{array}$$
(\emph{cf} Figure~\ref{fig:psir}).
Then, there exist $u_0 \in \R^+$ and two constants $a$ and $b$ 
such that $\psi_{F_r/F}(u) = e_F \cdot \psi_r(\frac{u-a}{e_F}) 
+ b$ for all integer $r$ and all $u \geq u_0$.
\end{rem}

Proposition \ref{prop:ramZp} has several interesting corollaries that we 
will derive below. We begin with two of them that give information about 
the behavior of the trace map.
The first one (Proposition \ref{prop:trFinf}) concerns extensions living 
inside $F_\infty$ and shows that traces in such extensions tend to
decrease the norm by a large factor.
On the contrary, the second one (Proposition \ref{prop:trFinfperp})
concerns extensions which are ``orthogonal'' to $F_\infty$ and shows
that traces in such extensions have a norm which is close to $1$.
The conceptual meaning of these results is that the extension 
$F_\infty/F$ captures almost all the ramification of $\bar K/F$.

\begin{prop}
\label{prop:trFinf}
There exists a constant $c_1$ (depending only on $F$ and $F_\infty$)
for which the following property holds:
for any positive integers $r$ and $s$ with $r \leq s$ and
for any $x \in \O_{F_s}$, we have:
$$v_p\big(\Tr_{F_s/F_r}(x)\big) \geq v_p(x) + s - r - c_1.$$
\end{prop}

\begin{proof}
Fix a positive integer $r$.
By the reformulation of Proposition~\ref{prop:ramZp} given in 
Remark~\ref{rem:ramZp}, we have:
$$\frac{\psi_{F_r/F}(u)}{e_F p^r} = u - r - \frac a {e_F} + \frac 
b{p^r e_F} - \frac{p^r-1}{p^r(p-1)}$$
when $u$ is sufficiently large. Thanks to formula \eqref{eq:difframif}, 
we find:
$$v_p(\D_{F_r/F}) = r + \frac a {e_F} - \frac
b{p^r e_F} + \frac{p^r-1}{p^r(p-1)}.$$
Given now two integers $r$ and $s$ with $r \leq s$, the transitivity
property of the different implies that:
$$v_p(\D_{F_s/F_r}) = v_p(\D_{F_s/F}) - v_p(\D_{F_r/F}) = s - r
- \frac b{p^s e_F} + \frac b{p^r e_F} + \frac {p^s-1}{p^s(p-1)}
- \frac {p^r-1}{p^s(p-1)}.$$
Then there exists $c \in \N$, not depending on $r$ and $s$,
such that $v_p(\D_{F_s/F_r}) \geq s - r - c$. Going back to the
definition of the different, we obtain the inclusion
$\Tr_{F_s/F_r}(p^{r-s+c} \O_{F_s}) \subset \O_{F_r}$.
Let now $x \in F_s$ and let $v$ be the integer part of $v_p(x)$.
Then $p^{-v} x$ falls in $\O_{F_s}$, so that we get
$\Tr_{F_s/F_r}(p^{r-s+c-v} x) \in \O_{F_r}$, \emph{i.e.}
$\Tr_{F_s/F_r}(x) \in p^{s-r-c+v} \O_{F_r}$. Consequently:
$$v_p(\Tr_{F_s/F_r}(x)) \geq s - r - c + v \geq v_p(x) + s - r - c -1.$$
We can then take $c_1 = c+1$.
\end{proof}

\begin{rem}
\label{rem:normalizedtraces}
For a fixed integer $r$, we can glue the $\Tr_{K_s/K_r}$ (for $s$
varying) and define a function $R_r : K_\infty \to K_r$ by
$R_r(x) = p^{r-s} \: \Tr_{K_s/K_r}(x)$ for $x \in K_s$.
We notice that the above definition makes sense because if $s \leq
t$, the functions $p^{r-s}\:\Tr_{K_s/K_r}$ and $p^{r-t}\:\Tr_{K_t/K_r}$
coincide on $K_s$. Proposition~\ref{prop:trFinf} shows that the
function $R_r$ obtained this way is uniformly continuous. It then
extends (uniquely) to the completion $\hat K_\infty$ of $K_\infty$.

The functions $R_r : \hat K_\infty \to K_r$ are called the
\emph{Tate's normalized traces}.
\end{rem}

\begin{prop}
\label{prop:trFinfperp}
Let $L$ be a finite Galois extension of $F$.
For all $\varepsilon > 0$, there exist a positive integer~$r$
and an element $x \in \O_{L{\cdot}F_r}$ such that
$v_p(\Tr_{L{\cdot}F_r/F_r}(x)) \leq \varepsilon$.
\end{prop}

\begin{proof}
Up to replacing $F$ by $F_\infty \cap L$, we may assume that
$F_\infty$ and $L$ are linearly disjoint over $F$.
We set $L_\infty = L{\cdot}K_\infty$ and $L_r = L{\cdot}F_r$ for
all $r$. The extension $L_\infty/L$ is then a $\Zp$-extension and
the $L_r$'s correspond to the subgroups $p^r \Zp$.

By formula \eqref{eq:difframif} and Remark~\ref{rem:ramZp}, there
exist $u_0 \in \R^+$ and $a,b,a',b' \in \R$ for which:
\begin{align*}
\psi_{L/F}(u) & = e_{L/F} \cdot \big(u - v_F(\D_{L/F})\big) \\
\psi_{L_r/F_r}(u) & = e_{L/F} \cdot \big(u - v_{F_r}(\D_{L_r/F_r})\big) \\
\psi_{F_r/F}(u) & = e_F \cdot \psi_r\Big(\frac{u-a}{e_F}\Big) + b \\
\psi_{L_r/L}(u) & = e_F \cdot \psi_r\Big(\frac{u-a'}{e_L}\Big) + b'
\end{align*}
for all $u \geq u_0$ and all positive integer $r$. 
Writing
$\psi_{L_r/L} \circ \psi_{L/F} = \psi_{L_r/F_r} \circ \psi_{F_r/F}$,
we obtain:
$$e_L \cdot \psi_r\Big(\frac u{e_F} - \frac a{e_F}\Big) 
+ e_{L/F} \cdot \big(b - v_{F_r}(\D_{L_r/F_r})\big) =
e_L \cdot \psi_r\Big(\frac u{e_F} - \frac {v_F(\D_{L/F})}{e_F} - 
\frac{a'}{e_L}\Big) + b'$$
for $u \geq u_0$. When $r$ is sufficiently large, the above identity
of functions implies, by comparing slopes, that 
$\frac a{e_F} = \frac {v_F(\D_{L/F})}{e_F} +
\frac{a'}{e_L}$ and $b - v_{F_r}(\D_{L_r/F_r}) = b'$.
From the latter equality, we derive $v_p(\D_{L_r/F_r}) = \frac{b - 
b'}{e_F p^r}$.

Let $\pi_{F_r}$ be a uniformizer of $F_r$. Let $y$ be an element of 
$\D_{L_r/F_r}$ with $v_p(y) = \frac{b - b'}{e_F p^r}$.
By definition of the different, there exists $z \in \O_{L_r}$ such that 
$\Tr_{L_r/F_r}(\frac z {\pi_{F_r} y}) \not\in \O_{F_r}$. In other words,
$v_p(\Tr_{L_r/F_r}(\frac z y)) < \frac 1{e_F p^r}$. Set $n = \lceil
b{-}b'\rceil$ and $x = \pi_{F_r}^n \: \frac z y$. 
We have $v_p(x) \geq \frac{n - (b-b')}{e_F p^r} \geq 0$; hence 
$x \in \O_{L_r}$. Moreover:
$$v_p(\Tr_{L_r/F_r}(x)) = 
v_p\big(\pi_{F_r}^n \Tr_{L_r/F_r}({\textstyle \frac z y})\big) = 
\frac n{e_F p^r} + v_p\big(\Tr_{L_r/F_r}({\textstyle \frac z y})\big) < 
\frac{n+1}{e_F p^r} < \frac{b-b'+2}{e_F p^r}.$$
We conclude the proof by noticing that, when $r$ goes to infinity, the 
upper bound $\frac{b-b'+2}{e_F p^r}$ converges to~$0$.
\end{proof}

We shall also need the following result which is a refinement of the 
classical additive Hilbert's theorem 90, allowing in addition some control 
on the valuation.

\begin{prop}
\label{prop:H90cont}
There exists a constant $c_2$ (depending only on $F$ and $F_\infty$)
for which the following property holds:
for any positive integers $r$ and $s$ with $r \leq s$ and
for any $x \in \O_{F_s}$ with $\Tr_{F_s/F_r}(x) = 0$, there
exists $y \in F_s$ such that (i)~$\Tr_{F_s/F_r}(y) = 0$, 
(ii)~$x = \gamma_r y - y$ and (iii)~$v_p(y) \geq v_p(x) - c_2$.

Moreover $y$ is uniquely determined by the conditions~(i) and~(ii).
\end{prop}

\begin{proof}
We set $d = s-r$ and:
$$y = -\frac 1{p^d} \cdot \sum_{i=0}^{p^d-1} i \: \gamma_r^i(x).$$
Noticing that $\gamma_r^{p^d}$ is the identity on $F_s$, we find that
$\gamma_r y - y = x + \frac 1{p^d} \Tr_{F_s/F_r}(x) = x$. Moreover
the assumption on the trace of $x$ implies that $\Tr_{F_s/F_r}(y)$
vanishes as well.

Let now $c_1$ be the constant of Proposition~\ref{prop:trFinf} and
set $c_2 = c_1 + 1$.
For $m \in \{0, \ldots, d\}$, we define
$x_m = \frac 1{p^{d-m}} \Tr_{F_s/F_{r+m}}(x)$ and 
$y_m = -\frac 1{p^m} \cdot \sum_{i=0}^{p^m-1} i \: \gamma_r^i(x_m)$.
Obviously $x_d = x$ and $y_d = y$.
Moreover, noticing that any integer between $0$ and $p^m{-}1$ can
be uniquely written as $a + p^{m-1}b$ with $0 \leq a < p^{m-1}$ and
$0 \leq b < p$, we obtain:
$$y_{m-1} - y_m \,\,=\,\, \frac 1 p \cdot \sum_{a=0}^{p^{m-1}-1}\,
\sum_{b=0}^{p-1} \, b \:\gamma_r^{a + p^{m-1}b}(x_m).$$
Therefore $v_p(y_m) \geq \min(v_p(x_m){-}1, \, v_p(y_{m-1}))$ and so
$v_p(y) = v_p(y_d) \geq \min_{1 \leq m \leq d} v_p(x_m) - 1$.
By Proposition~\ref{prop:trFinf}, we end up with $v_p(y) \geq v_p(x) - 
c_2$. The element $y$ we have constructed then satisfies the
requirements (i), (ii) and (iii).

It remains to prove unicity. Assume that we have given $y_1$ and
$y_2$ such that $\Tr_{F_s/F_r}(y_1) = \Tr_{F_s/F_r}(y_2)$ and
$x = \gamma_r y_1 - y_1 = \gamma_r y_2 - y_2$. Set $z = y_1 - y_2$.
The second condition implies that $z$ is fixed by $\gamma_r$. Hence
$z \in F_r$ and $\Tr_{F_s/F_r}(z) = p^{s-r} z$. On the
other hand, one has $\Tr_{F_s/F_r}(z) = \Tr_{F_s/F_r}(y_1) -
\Tr_{F_s/F_r}(y_2) = 0$. We conclude that $p^{s-r} z = 0$ and
hence $y_1 = y_2$.
\end{proof}

\begin{rem}
\label{rem:H90Rr}
Using Tate's normalized traces (\emph{cf} 
Remark~\ref{rem:normalizedtraces}), one may extend 
Proposition~\ref{prop:H90cont} for $x \in \hat K_\infty$.
The result we obtain reads as follows: for all $x \in \hat K_\infty$
with $R_r(x) = 0$, there exists a unique $y \in K_r$ such that
$x = \gamma_r y - y$ and $R_r(y) = 0$. Moreover, this element
$y$ satisfies $v_p(y) \geq v_p(x) - c_2$.
In other terms, the function $(\gamma_r - \id)$ is bijective
on the kernel of $R_r$ and its inverse is continuous.
\end{rem}

\subsection{$\Cp$-admissibility}

We now come to the study of $\Cp$-admissibility of representations of 
$G_K$. The main objective of this subsection is to prove 
Theorem~\ref{theointro:Cpadm} whose statement is recalled below.

\begin{theo}
\label{theo:Cpadm}
Let $V$ be a $\Qp$-linear finite dimensional representation
of $G_K$. Then $V$ is $\Cp$-admissible if and only if the
inertia subgroup of $G_K$ acts on $V$ through a finite
quotient.
\end{theo}

\noindent
As an application, in \S\ref{sssec:HT}, we will explain how 
Theorem~\ref{theo:Cpadm} can be used to understand better the internal 
structure of Hodge--Tate representations.

\subsubsection{Preliminaries}

Before giving the proof of Theorem~\ref{theo:Cpadm}, we need some 
preparation. The first input that we shall use is Ax--Sen--Tate theorem, 
whose purpose is to compute the fixed subspace of $\Cp$ under the action 
of $G_K$.
As in the previous subsection, we shall work over a base $F$ which
is itself a finite extension of $K$. For convenience, we set $G_F = 
\Gal(\bar K/F)$.

\begin{theo}[Ax--Sen--Tate]
\label{theo:axsentate}
We have $\Cp^{G_F} = F$.
\end{theo}

We shall prove a ``finite'' version of Ax--Sen--Tate theorem 
which is more precise.

\begin{theo}
\label{theo:axsentate2}
There exists a constant $c_3$ (depending only on $F$) for which the 
following property holds:
for all real number $v$ and all $x \in \bar K$ such that $v_p(gx - x)
\geq v$ for all $g \in G_F$, there exists $z \in K$ such that $v_p(x-z)
\geq v - c_3$.
\end{theo}

\begin{proof}
Throughout the proof, we fix a $\Zp$-extension $F_\infty$ of $F$.
We recall that such an extension always exists and can be built
from the cyclotomic extension of $F$ as discussed in 
\S\ref{sssec:cyclo}.

Let $L$ be a finite Galois extension of $\Qp$ in which $x$ lies.
Thanks to Proposition \ref{prop:trFinfperp}, one can choose an 
integer $r$ together with an element $\lambda \in L{\cdot}F_r$ with
the property that $v_p(\Tr_{L{\cdot}F_r/F_r}(\lambda)) \leq 1$.
We consider the elements:
$$y = \frac {\Tr_{L{\cdot}F_r/F_r}(\lambda x)}
{\Tr_{L{\cdot}F_r/F_r}(\lambda)} \in F_r
\quad \text{and} \quad
z = \frac 1{p^r} \cdot \Tr_{F_r/F}(y) \in F.$$
The fact that $v_p(g x - x) \geq v$ for all $g \in G_F$ implies that 
$v_p(y - x) \geq v - 1$.

Observe that $\Tr_{F_r/F}(y - z) = p^r z - p^r z = 0$ and $(\gamma_0 - 
\id)(y-z) = \gamma_0 y - y$. In other words, the element $y-z$ has trace 
$0$ and is an antecedent of $\gamma_0 y - y$ by the application 
$(\gamma_0 - \id)$. By Proposition~\ref{prop:H90cont}, it follows that 
$v_p(y-z) \geq v_p(\gamma_0 y - y) - c_2$. Now notice that the 
combinaison of $v_p(\gamma_0 x - x) \geq v$ and $v_p(y - x) \geq v - 1$ 
ensures that $v_p(\gamma_0 y - y) \geq v - 1$. Therefore, we obtain 
$v_p(y-z) \geq v-(c_2 + 1)$. Finally $v_p(x-z) \geq 
\min(v_p(x-y),v_p(y-z)) \geq v-(c_2 + 1)$ and we can take $c_3 = c_2+1$.
\end{proof}

\begin{proof}[Proof of Theorem~\ref{theo:axsentate}]
Let $x \in \Cp^{G_F}$. We consider a positive integer $n$.
Since $\Cp$ is the completion of $\Qpbar$, one can find an
element $x_n \in \Qpbar$ such that $v_p(x - x_n) \geq n$.
We then have $v_p(g x_n - x_n) \geq n $ for all $g \in G_F$.
By Theorem~\ref{theo:axsentate2}, one can find $z_n \in K$
such that $v_p(z_n - x_n) \geq n - c_3$. This implies that
$v_p(z_n - x) \geq n - c_3$ as well. 
The sequence $(z_n)_{n \geq 1}$ then converges to $x$.
Since $z_n \in K$ for all $K$, we obtain $x \in K$.
\end{proof}

\begin{rem}
\label{rem:classfield}
As presented above, it seems that the proof of 
Theorem~\ref{theo:axsentate} uses class field theory (\emph{via}
Proposition~\ref{prop:ramZp}). In fact, it is not the case because we
have the choice on the $\Zp$-extension $F_\infty$. If we decide to
take the $\Zp$-part of the cyclotomic extension, the computation 
of the ramification filtration of $\Gal(F_\infty/F)$ can be carried 
out explicitely, so that Proposition~\ref{prop:ramZp} can be proved in 
this case without making any reference to class field theory.
\end{rem}

The proof we have exposed above is essentialy due to 
Tate~\cite{tate}. A few years later, Ax~\cite{ax} reproves the 
theorem using a more direct and elementary argument. 
We presented Tate's proof because we believe that it serves as a very 
good introduction to the developments we will discuss afterwards, which 
are all modeled on the same strategy: in order to study the action of 
$G_F$ (on some space), we will always first descend to $F_\infty$ using 
Proposition~\ref{prop:trFinfperp} and then to $F$---or possibly only 
$F_r$ for some finite $r$---using Proposition~\ref{prop:trFinf} or 
Proposition~\ref{prop:H90cont}.

Ax's proof provides in addition an explicit value for the constant 
$c_3$, namely $\frac p{(p-1)^2}$. Ax asks for the optimality of this
constant. In \cite{leborgne}, Le Borgne answers this question and
shows that the optimal constant is \emph{not} $\frac p{(p-1)^2}$, but
$\frac 1{p-1}$. Le Borgne's proof follows Tate's strategy but uses a
non Galois extension in place of the cyclotomic extension.

\paragraph{An extension of Hilbert's theorem 90.}

Another input we shall need is a variant of Hilbert's theorem 90
(\emph{cf} Theorem~\ref{theo:H90}) valid for \emph{infinite} 
unramified extensions.
We recall that $K^\ur$ denotes the maximal unramified extension of
$K$ (inside $\bar K$). We define $\hat K^\ur$ as the completion of 
$K^\ur$; it is a field which naturally embeds into $\Cp$ and which
is equipped with a canonical action of $\Gal(K^\ur/K)$.

\begin{prop}
\label{prop:H90unram}
Any finite dimensional $\hat K^\ur$-semi-linear representation of 
$\Gal(K^\ur/K)$ is trivial.
\end{prop}

\begin{rem}
Proposition \ref{prop:H90unram} implies in particular that 
unramified representation of $G_K$ are $\hat K^\ur$-admissible and
then \emph{a fortiori} $\Cp$-admissible. It then appears as a
first step towards the proof of Theorem~\ref{theo:Cpadm}.
\end{rem}

\begin{proof}[Proof of Proposition~\ref{prop:H90unram}]
In order to simplify notations, we denote by $\O$ the ring of integers 
of $\hat K^\ur$, and by $\m$ its maximal ideal. We recall that the 
quotient $\O/\m$ is isomorphic to an algebraic closure $\bar k$ of $k$.
We recall also that $\Gal(K^\ur/K)$ is a procyclic group generated by
the Frobenius $\Frob_q : x \mapsto x^q$ (where $q$ is the cardinality
of $k$).

Let $W$ be a finite dimensional $\hat K^\ur$-semi-linear representation.
We fix $(v_{1,0}, \ldots, v_{d,0})$ a basis of $W$ over $\hat K^\ur$.
Let $\O_W$ be a $\O$-span of $v_{1,0}, \ldots, v_{d,0}$.
We are going to construct a sequence of tuples $(v_{1,n}, \ldots, 
v_{d,n})$ such that $v_{i,n+1} \equiv v_{i,n} \pmod {\m^n}$ and
$\Frob_q (v_{i,n}) \equiv v_{i,n} \pmod{\m^n}$
for all $i \in \{1,\ldots, d\}$ and all $n \in \N$. 

We proceed by induction on $n$. The case $n = 1$ reduces to the fact 
that $\O_W/\m\O_W$ is trivial as a $\bar k$-semi-linear representation 
of $\Gal(K^\ur/K) \simeq \Gal(\bar k/k)$. In order to prove this, we
remark that, using continuity, $\O_W/\m\O_W$ descends at 
finite level: there exist a finite extension $\ell$ of $k$ and a 
$\ell$-semi-linear representation $W_\ell$ of $\Gal(\ell/k)$ such that 
$\bar k \otimes_\ell W_\ell = \O_W/\m\O_W$. The property we want to 
establish then follows from Hilbert's theorem 90 (\emph{cf} 
Theorem~\ref{theo:H90}).

We now assume that $(v_{1,n}, \ldots, v_{d,n})$ has been constructed.
We look for vectors $w_1, \ldots, w_n \in \O_W$ such that $\Frob_q 
(v_{i,n} + \pi^n w_i) \equiv v_{i,n} + \pi^n w_i \pmod{\m^{n+1}}$ 
for all $i$. Letting $\bar w_i$ be the image of $w_i$ in $\O_W/\m\O_W$, 
the system we have to solve can be rewritten $\Frob_q \bar w_i - \bar 
w_i = \bar c_i$ ($1 \leq i \leq d$) where $\bar c_i$ is defined as the 
image of $\frac{\Frob_q v_{i,n} - v_{i_n}}{\pi^n}$ in $\O_W/\m\O_W$. 
It is then enough to prove that $(\Frob_q - \id)$ is surjective on 
$\O_W/\m\O_W$. This follows directly from the triviality of 
$\O_W/\m\O_W$ and the fact that $(\Frob_q - \id)$ is surjective on 
$\bar k$.

We conclude the proof by remarking that, for any fixed $i$, the 
sequence $v_{i,n}$ is Cauchy and hence converges to a vector
$v_i \in \O_W$ on which $\Gal(K^\ur/K)$ acts trivially. Moreover
the family of $v_i$'s is an $\O$-basis of $\O_W$ (because its
reduction modulo $\m$ is a basis of $\O_W/\m\O_W$) and then it is
also a $\hat K^\ur$-basis of $W$.
\end{proof}

\subsubsection{Proof of Theorem~\ref{theo:Cpadm}}

We are now ready to prove Theorem~\ref{theo:Cpadm}.

Write $d = \dim_{\Qp} V$.
We first assume that the inertia subgroup acts on $V$ through
a finite quotient. In other words, there exists a finite extension 
$L$ of $K^\ur$ for which $\Gal(\bar K/L)$ acts trivially on $V$.
By Hilbert's theorem 90 (\emph{cf} Theorem~\ref{theo:H90}), the
$L$-semi-linear representation $L \otimes_{\Qp} V$ admits an 
$L$-basis $(v_1, \ldots, v_d)$ on which the action of $\Gal(L/K^\ur)$
is trivial. Consequently $\Gal(K^\ur/K)$ operates on the 
$\hat K^\ur$-span of $v_1, \ldots, v_d$. 
By Proposition~\ref{prop:H90unram}, this semi-linear representation is 
trivial. Therefore $V$ is $(L{\cdot}\hat K^\ur)$-admissible.
It is then also $\Cp$-admissible.

We now focus on the converse.
We assume that $V$ is $\Cp$-admissible. Then by definition, there exists 
a $\Cp$-basis $(w_1, \ldots, w_d)$ of $\Cp \otimes_{\Qp} V$ with the 
property that $g w_i = w_i$ for all $g \in G_K$ and all $i \in 
\{1,\ldots, d\}$. Let $(v_1, \ldots, v_d)$ be a basis of $V$ over $\Qp$ 
and let $P \in \GL_d(\Cp)$ be the matrix representating the change of 
basis between the $v_i$'s and the $w_i$'s. Up to rescaling the $v_i$'s, 
we may assume without loss of generality that $P \in M_d(\O_{\Cp})$.

From the fact that the $\Qp$-span of the $v_i$'s is stable under the 
action of $G_K$, we derive that the matrix $U_g = P^{-1} \cdot gP$ 
has coefficients in $\Qp$ for all $g \in G_K$. Let $c_3$ be the constant
of Theorem~\ref{theo:axsentate2} and let $v$ be a positive 
integer for which $p^v {\cdot} P^{-1}$ has coefficients in $\O_{\Cp}$. By 
continuity of the action of $G_K$, denoting by $I_d$ the identity
matrix of size $d$, there exists an open subgroup $H$ of 
$G_K$ such that $v_p(U_g - I_d) \geq v + c_3 + 1$ for all $g \in H$. 
Multiplying by $P$ on the left, we get $v_p(P - g P) \geq v + c_3 + 
1$ for all $g \in H$. Applying now Theorem \ref{theo:axsentate2}
(to each entry of $P$), 
we find a matrix $P_0 \in \GL_n(L)$ such that $P \equiv P_0 
\pmod{p^{v+1}}$. multiplying by $P^{-1}$ on the left, we get
$P^{-1} P_0 \equiv I_d \pmod p$. Define $M = P^{-1} P_0$. 
Writing $P = P_0 M^{-1}$, we find $M \cdot g M^{-1} = U_g$
for all $g \in H$. Since $M \equiv I_d \pmod p$, the matrix 
$N = \log M$ is well defined and satisfies the relation
$N - g N = \log U_g$ for all $g \in H$.

Let $F$ be the extension of $K$ cut out by $H$. We are going to
prove that $H \cap I_K$ operates trivially on $N$.
Let $\xi$ be an entry of $N$. The relation 
$N - g N = \log U_g$ ensures that 
$\xi - g\xi \in \Zp$ whenever $g$ is in $H$. Define the function
$\alpha : H \to \Zp$ by $\alpha(g) = g\xi - \xi$. The computation
$$\alpha(g_1 g_2) = g_1 g_2 \xi - \xi = g_1 (g_2 \xi - \xi) + 
(g_1 \xi - \xi) = (g_2 \xi - \xi) + (g_1 \xi - \xi) = 
\alpha(g_2) + \alpha(g_1)$$
shows that $\alpha$ is an additive character. Its kernel defines a Galois 
extension $F_\infty$ of $F$ whose Galois group embeds into $\Zp$. 
Moreover, by construction, $H \cap \ker \alpha$ acts trivially on $\xi$. 
We then need to prove that $F_\infty$ is unramified over~$F$.

We assume by contraction that the extension $F_\infty/F$ is ramified.
In particular, it is not trivial, and hence it is a $\Zp$-extension.
Proposition~\ref{prop:trFinf} then applies and ensures that there exists 
a constant $c_1$ such that:
\begin{equation}
\label{eq:trCpadm}
v_p(\Tr_{F_s/F_r}(z)) \geq v_p(z) + s - r - c_1
\end{equation}
whenever $s \geq r$ and $z \in F_s$.
Let $v$ be a positive real number and let $x$ be an element of
$\bar K$ such that $v_p(x-\xi) \geq v$. From the equality $g\xi - 
\xi = \alpha(g)$, we derive $v_p\big(gx - x - \alpha(g)\big) \geq v$
for all $g \in H$. In particular, if $g \in H \cap \ker\alpha$, we
obtain $v_p(gx - x) \geq v$. By continuity, this estimation is
also correct for $g \in \Gal(\bar K/F_s)$ for some integer $s$.
Repeating the first part of the proof of Theorem~\ref{theo:axsentate2} 
(with the $\Zp$-extension $F_\infty/F$) and possibly enlarging 
$s$, we find that there exists $y \in F_s$ with the property that 
$v_p(x-y) \geq v-1$. Thus $v_p(\xi - y) \geq v - 1$ as well.

Fix now $g \in H$ and set $z = gy - y - \alpha(g)$. By our assumption
on $\xi$, we know that $v_p(z) \geq v - 1$. Using \eqref{eq:trCpadm} 
with $r = 0$, we obtain $v_p(\Tr_{F_s/F_0}(z)) \geq v + s - c_1$. 
On the other hand, a direct computations yields 
$\Tr_{F_s/F_0}(z) = -p^s \alpha(g)$. Combining these two
inputs, we deduce
$v_p(\alpha(g)) \geq v - c_1$. Since this estimation holds for all
$g \in H$ and all $v \in \R^+$, we end up with $\alpha = 0$. 
This means that $F_\infty = F$ and then contradicts
our assumption that $F_\infty/F$ was ramified.

\begin{rem}
\label{rem:Cpadm}
It follows from the proof above (\emph{cf} in particular the first 
paragraph of the proof) that a representation is $\Cp$-admissible if and 
only if it is $(L {\cdot} \hat K^\ur)$-admissible for a
\emph{finite} extension $L$ of $K^\ur$. We will reuse this property in 
\S\ref{ssec:dR} when we will compare $\Cp$-representations with de Rham 
representations.
\end{rem}

\subsubsection{Application to Hodge--Tate representations}
\label{sssec:HT}

Beyond its obvious own interest, Theorem~\ref{theo:Cpadm} can be thought 
of as a first result towards the study of Hodge--Tate representations.
Recall that a finite dimensional $\Qp$-linear representation $V$ of
$G_K$ is Hodge--Tate if $\Cp \otimes_{\Qp} V$ decomposes as:
\begin{equation}
\label{eq:HTbis}
\Cp \otimes_{\Qp} V = 
\Cp(\chi_\cycl^{n_1}) \oplus
\Cp(\chi_\cycl^{n_2}) \oplus \cdots \oplus
\Cp(\chi_\cycl^{n_d})
\end{equation}
for some integers $n_i$'s. Theorem~\ref{theo:Cpadm} implies the
following unicity result.

\begin{prop}
\label{prop:HTweights}
Let $V$ be a finite dimensional Hodge--Tate representation of
$G_K$. Then the integers $n_i$ of Eq.~\eqref{eq:HTbis} are
uniquely determined up to permutation.
\end{prop}

\begin{proof}
We have to show that, if
$\Cp(\chi_\cycl^{n_1}) \oplus \cdots \oplus
\Cp(\chi_\cycl^{n_d}) \, \simeq \,
\Cp(\chi_\cycl^{m_1}) \oplus \cdots \oplus
\Cp(\chi_\cycl^{m_{d'}})$,
then $d = d'$ and the $n_i$'s agree with the $m_i$'s up to permutation. 
For this, it is enough to check that, given two integers $n$ and $m$, 
\begin{equation}
\label{eq:homHT}
\Hom_{\Rep_{\Cp}(G_K)}\big(\Cp(\chi_\cycl^n), \Cp(\chi_\cycl^m)\big)
\end{equation}
is a one dimensional $K$-vector space if $n = m$, and is zero 
otherwise.

Let $W = \Hom_{\Cp}\big(\Cp(\chi_\cycl^n), \Cp(\chi_\cycl^m)\big)
\simeq \Cp(\chi_\cycl^{n-m})$ (equipped with its Galois action).
The space \eqref{eq:homHT} is 
equal to $W^{G_K}$. When $n = m$, it is then $\Cp^{G_K}$ which
is indeed equal to $K$ by Ax--Sen--Tate theorem (\emph{cf}
Theorem~\ref{theo:axsentate}). If $n \neq m$, we need to prove that $W$ 
is not trivial, which means that the representation 
$V = \Qp(\chi_\cycl^{n-m})$ is not $\Cp$-admissible.
By Theorem~\ref{theo:Cpadm}, we are reduced to justify that the inertia 
subgroup of $\Qp$ does not act on $V$ through a finite quotient. This is 
clear because the extension cut out by the kernel of 
$\chi_\cycl^{n-m}$ is the $p$-adic cyclotomic extension which is 
infinitely ramified.
\end{proof}

\begin{rem}
A byproduct of the proof above is that the $\Cp$-semi-linear
representation $\Cp(\chi_\cycl^n)$ has no nonzero invariant vector
when $n \neq 0$. We will reuse repeatedly this property in the sequel.
\end{rem}

\begin{ex}
\label{ex:HTchar}
Recall that, assuming $p > 2$, we have classified the characters of 
$G_{\Qp}$ in Proposition~\ref{prop:charGqp}: they are all of the form 
$\mu_\lambda \cdot \chi_\cycl^a \cdot \omega_\cycl^b$ with $a \in \Zp$ 
and $b \in \Z/(p{-}1)\Z$. Here $\mu_\lambda$ denotes the unramified 
character taking the Frobenius $\Frob_q$ to $\lambda$ and $\omega_\cycl 
= [\chi_\cycl \mod p]$.
Since the representations $\Cp(\mu_\lambda)$ and $\Cp(\omega_\cycl^b)$ 
are $\Cp$-admissible, we obtain:
$$\Cp(\mu_\lambda \cdot \chi_\cycl^a \cdot \omega_\cycl^b)
\simeq \Cp(\chi_\cycl^a).$$
Hence the character $\mu_\lambda \cdot \chi_\cycl^a \cdot 
\omega_\cycl^b$ is Hodge--Tate if and only if $a \in \Z$. In this
case, its Hodge--Tate weight is $a$.
\end{ex}

\begin{ex}
\label{ex:HTdim2}
Let $\alpha : G_K \to \Zp$ be an additive character, \emph{e.g.}
$\alpha = \log \chi_\cycl$.
Consider the two dimensional representation $V$ corresponding to 
the group homomorphism:
$$G_K \to \GL_2(\Qp), \quad g \mapsto 
\left(\begin{matrix} 1 & \alpha(g) \\ 0 & 1 \end{matrix}\right).$$
In order terms, $V = \Qp^2$ and $G_K$ acts to $V$ by $g \cdot (u,v)
= (u,\, v + \alpha(g) u)$. We have an obvious exact sequence
$0 \to \Qp \to V \to \Qp \to 0$ where the action of $G_K$ on the
two copies of $\Qp$ is the trivial action. Tensoring this
sequence by $\Cp$, we get $0 \to \Cp \to V \to \Cp \to 0$.
The representation $V$ is Hodge--Tate if and only if 
the above sequence splits, if and only if $V$ is $\Cp$-admissible.
By Theorem~\ref{theo:Cpadm}, this happens if and only if 
$\alpha(I_K)$ is finite (where $I_K$ is the inertia subgroup
of $G_K$). Since $\alpha(I_K)$ is a subgroup of $\Zp$, the
previous condition is equivalent to the fact that $\alpha(I_K)$
is reduced to $0$. As a conclusion, the representation $V$ is
Hodge--Tate if and only if $\alpha$ is unramified. In this case,
the Hodge--Tate weights of $V$ are $0$ with multiplicity $2$.
\end{ex}

\paragraph{Hodge--Tate representations and admissibility.}

It is important to notice that the class of Hodge--Tate 
representations fits very well in Fontaine's framework presented
in \S\ref{ssec:Fontainestrategy}. Precisely, let us
consider the rings $B_\HT = \Cp[t,t^{-1}]$ and $B'_\HT = \Cp(\!(t)\!)$.
We equip them with the Galois action obtained by letting $G_K$ act 
naturally on $\Cp$ and act on $t$ by $g t = \chi_\cycl(g) \: t$ for 
all $g \in G_K$. In addition, we define a filtration of $B'_\HT$ by 
$\Fil^m B'_\HT = t^m \Cp[[t]]$ for $m$ varying in $\Z$. The graded
ring of $B'_\HT$ is, by definition:
$$\gr\,B'_\HT = \bigoplus_{m \in \Z} \Fil^m B'_\HT / \Fil^m B'_\HT.$$
We observe that it is canonically isomorphic to $B_\HT$. Besides, we
have a natural $G_K$-equivariant inclusion $B_\HT \to B'_\HT$.
Ax--Sen--Tate theorem, together with the fact that $\Cp(\chi_\cycl^n)$
has no nonzero invariant vectors as soon as $n \neq 0$, implies that 
$(B_\HT)^{G_K} = (B'_\HT)^{G_K} = K$.

\begin{prop}
The rings $B_\HT$ and $B'_\HT$ satisfy Fontaine's assumptions (H1), (H2) 
and (H3) (introduced in \S\ref{sssec:criterium}).
\end{prop}

\begin{proof}
This is obvious for $B'_\HT$ since it is a field.
As for $B_\HT$, it is clearly a domain. Moreover since $B'_\HT$ is 
a field, we have $B_\HT \subset \Frac B_\HT \subset B'_\HT$. 
Taking the $G_K$-invariants, we obtain $(\Frac B_\HT)^{G_K} 
= K$; hence $B_\HT$ satisfies (H2). 
Finally, we prove that $B_\HT$ satisfies (H3).
Let $x \in B_\HT$, $x \neq 0$ and assume that the line $\Qp x$ is
stable by $G_K$. We have to prove that $x$ is invertible in $B_\HT$.
Up to multiplying $x$ by some power of $t$, we may assume that $x
\in \Cp[t]$. Write $x = a_0 + a_1 t + \cdots + a_n t^n$ where the
$a_i$'s are in $\Cp$. Our assumption implies that there exists
$\lambda \in \Qp$ such that
$g a_i \cdot \chi(g)^i = \lambda a_i$
for all $g \in G_K$ and all $i \in \{0, 1, \ldots, n\}$.
Let $j$ be an index for which $a_j \neq 0$ and write $\mu_i = 
\frac{a_i}{a_j}$ for all $i$. We then have $g\mu_i = \chi(g)^{j-i} 
\mu_i$ for all $g$ and $i$. If $i \neq j$, this implies that $\mu_i = 0$ 
since $\Cp(\chi_\cycl^{j-i})^{G_K} = 0$. Therefore $x$ has to be equal
to $a_j t^j$, and so is invertible in $B_\HT$.
\end{proof}

\begin{prop}
\label{prop:HTadm}
Let $V$ be a finite dimensional $\Qp$-linear representation.
Then $V$ is Hodge--Tate if and only if it is $B_\HT$-admissible,
if and only if it is $B'_\HT$-admissible.
\end{prop}

\begin{proof}
Write $d = \dim_{\Qp} V$.
Observe that $B_\HT = \bigoplus_{m \in \Z} \Cp(\chi_\cycl^m)$ 
as a $\Cp$-semi-linear representation. Therefore:
$$\big(V \otimes_{\Qp} B_\HT\big)^{G_K} \simeq 
\bigoplus_{m \in \Z}
\big(V \otimes \Cp(\chi_\cycl^m)\big)^{G_K}.$$
Suppose that $V$ is Hodge--Tate. Let $m_1, \ldots, m_s$ be its
Hodge--Tate weights and $e_1, \ldots, e_s$ be the corresponding
multiplicities. 
The space $\big(V \otimes \Cp(\chi_\cycl^{-m_i})\big)^{G_K}$ has
then dimension $e_i$. Summing up all these contributions, we find
that $(V \otimes_{\Qp} B_\HT)^{G_K}$ has dimension $d$, which means
that $V$ is $B_\HT$-admissible.

The converse and the case of $B'_\HT$ are proved in a similar fashion 
and left to the reader.
\end{proof}

\subsection{Complement: Sen's theory}
\label{ssec:Sen}

The aim of this subsection is to expose Sen's theory~\cite{sen} whose 
objective is to provide a systematic study of finite dimensional 
$\Cp$-semi-linear representations of $G_K$. In what follows, we choose 
and fix once for all a $\Zp$-extension $K_\infty$ of $K$. We let $\alpha 
: \Gal(K_\infty/K) \to \Zp$ be the attached group isomorphism. As in 
\S\ref{sssec:Zpext}, we put $\gamma_r = \alpha^{-1}(p^r)$ and let $K_r$ 
be the subextension of $K_\infty$ corresponding to the closed subgroup 
$\alpha^{-1}(p^r\Zp)$.

We recall that one possible choice is $\alpha = \log \chi_\cycl$, in 
which case $K_\infty$ is the $\Zp$-part of the cyclotomic extension of 
$K$ (\emph{cf} \S\ref{sssec:cyclo}). Actually, strictly speaking, Sen's 
theory only concerns this particular choice of $\alpha$. However the 
extension of general $\alpha$'s is straightforward. In what follows,
we do \emph{not} restrict ourselves to $\alpha = \log \chi_\cycl$.

\begin{rem}
Recently, Berger and Colmez~\cite{berger-colmez} generalized Sen's 
theory, allowing $\alpha$ to take its values in any $p$-adic Lie group 
(possibly noncommutative). Their theory relies on the notion of
locally analytic vectors, which is not needed in classical Sen's
theory (finite vectors are enough as we shall explain below). 
We will not expose their generalization in the article and do 
restrict ourselves to homomorphisms $\alpha$ taking their values
in $\Zp$.
\end{rem}

We recall that, given a topological group $G$ and a topological ring $B$ 
on which $G$ acts, we have introduced the notation $\Rep_B(G)$ for the 
category of $B$-semi-linear representations of $G$. 
Let $\Rep^\fin_B(G)$ denote the full subcategory of $\Rep_B(G)$ 
consisting of representations which are finitely generated as a 
$B$-module. When $B$ is the field, $\Rep^\fin_B(G)$ is then the category 
of \emph{finite dimensional} $B$-semi-linear representations of $G$.

\paragraph{Descend.}

The first result towards Sen's theory is 
Proposition~\ref{prop:descendKinfty} just below, which could understood
as an analogue of Hilbert's theorem 90 for the Galois group $\Gal(\bar 
K/K_\infty)$.

\begin{prop}
\label{prop:descendKinfty}
Let $W \in \Rep^\fin_{\Cp}(G_K)$. 
Then there exist an integer $r$ and a $\Cp$-basis $v_1, 
\ldots, v_d$ of $W$ such that $g v_i = v_i$ for all $g \in
\Gal(\bar K/K_\infty)$.
\end{prop}

\begin{proof}
The proof is similar to that of Proposition~\ref{prop:H90unram}.

Let $\O_W$ be any $\O_{\Cp}$-lattice in $W$. 
As a first step, we are going to construct 
a $\Cp$-basis $w_1, \ldots, w_d$ of $W$ with $w_i \in \O_W$, $g w_i 
\equiv w_i \pmod {p^2 \O_W}$ for all $g \in \Gal(\bar K/K_\infty)$
and $p \O_W \subset \O_{\Cp} w_1 \oplus \cdots \oplus \O_{\Cp} w_d$.
By continuity of the Galois action, there exists a finite Galois
extension $L$ of $K$ such that $g w \equiv w \pmod{p^2 \O_W}$ for
all $g \in \Gal(\bar K/L)$ and all $w \in W$. 
For a positive integer $r$, set $L_r = L{\cdot}K_r$. By the proof of 
Proposition~\ref{prop:trFinf}, we know that there exists $r$ for which 
$v_p(\D_{L_r/K_r}) < e_{L/K}^{-1}$. We fix such an $r$.

Let $\lambda_1, \ldots, \lambda_m$ be a $\O_{K_r}$-basis of $\O_{L_r}$
and let $g_1, \ldots, g_m$ be the elements of $\Gal(L_r/K_r)$. 
For $i \in \{1,\ldots, m\}$, choose $\hat g_i \in G_K$ a
lift of $g_i$. We define the elements:
$$y_{i,j} = \sum_{i'=1}^m \hat g_{i'}(\lambda_j x_i) =
\sum_{i'=1}^m g_{i'}(\lambda_j) \cdot \hat g_{i'}(x_i)$$
for $i$ varying between $1$ and $d$ and $j$ varying between $1$ and
$m$.
It is easily seen that $g y_{i,j} \equiv y_{i,j} \pmod{p^2\O_W}$
for all $g \in \Gal(\bar K/K_r)$. Moreover, the 
determinant of the matrix $(g_i(\lambda_j))_{1\leq i,j\leq m}$ is, by
definition, the discriminant of $L_r/K_r$. Its $p$-adic valuation is
then less than $1$ thanks to our assumption on $v_p(\D_{L_r/K_r})$.
We deduce that there exist $\mu_1, \ldots, \mu_m \in \O_{L_r}$
with the property that $\sum_{j=1}^m \mu_j \Tr_{L_r/K_r}(\lambda_j) 
= p$. Hence $\sum_{j=1}^m \mu_j y_{i,j} \equiv p x_i \pmod{p^2\O_W}$ 
for all $i$.
The $\O_{\Cp}$-span of the $y_{i,j}$'s then contains $p \O_W$.
Among these vectors, one can select $d$ of them $w_1, \ldots, w_d$
whose span still contains $p \O_W$. The $w_i$'s satisfy all the 
announced properties.

\smallskip

The second step of the proof consists in lifting the $w_i$'s
by a process of successive approximations.
In order to simplify the notations, we redefine $\O_W$ as the
$\O_{\Cp}$-span of $w_1, \ldots, w_d$. With the new definition,
we have $g w_i \equiv w_i \pmod {p\O_W}$ for all $g \in \Gal
(\bar K/K_\infty)$. We will construct by
induction on $n$ a sequence of families $(v_{1,n}, \ldots, 
v_{d,n})$ satisfying the following congruences: 
$$v_{i,n+1} \equiv v_{i,n} \pmod{p^n \O_W}
\quad \text{and} \quad 
g v_{i,n} \equiv v_{i,n} \pmod{p^n \O_W}$$
for all $i \in \{1,\ldots,d\}$, $n \in \N$ and $g \in \Gal 
(\bar K/K_\infty)$. For $n = 1$, we set $v_{i,1} = w_i$. Now
we assume that the $v_{i,n}$'s have been constructed. 
By continuity there exists a finite Galois extension $L$ of $K$ such 
that $g v_{i,n} \equiv v_{i,n} \pmod{p^{n+2} \O_W}$ for all $g \in 
\Gal(\bar K/L)$. By Proposition~\ref{prop:trFinf}, there exist an
integer $r$ and $\lambda \in \O_{L_r}$ (with $L_r = L{\cdot}K_r$)
such that $v_p(\Tr_{L_r/K_r}(\lambda)) \leq 1$. As in the first
step, we let $g_1, \ldots, g_m$ be the elements of $\Gal(L_r/K_r)$
and we choose a lifting $\hat g_i \in G_K$ of $g_i$. We define:
$$v_{i,n+1} = \frac 1 {\Tr_{L_r/K_r}(\lambda)} \cdot 
\sum_{j=1}^m \hat g_j(\lambda v_{i,n}) =
 \frac 1 {\Tr_{L_r/K_r}(\lambda)} \cdot
\sum_{j=1}^m g_j(\lambda) \cdot \hat g_j(v_{i,n})$$
and check that the $v_{i,n+1}$'s satisfy the desired requirements.

We conclude the proof by taking the limit with respect to $n$.
\end{proof}

Proposition~\ref{prop:descendKinfty} tells us that the $W$ is trivial 
when viewed as a $\Cp$-linear representation of $\Gal(\bar K/K_\infty)$. 
Moreover by the proof of Ax--Sen--Tate theorem, the fixed field 
$\Cp^{\Gal(\bar K/K_\infty)}$ is the completion of $K_\infty$, that we 
shall call $\hat K_\infty$.
By general results of trivial semi-linear representations (\emph{cf}
\S\ref{sssec:trivialrep}), we then have an isomorphism 
$$\Cp \otimes_{\hat K_\infty} W^{\Gal(\bar K/K_\infty)} \simeq W$$
for all $\Cp$-semi-linear representation of $W$. We notice that
$W^{\Gal(\bar K/K_\infty)}$ inherits an action of $\Gal(K_\infty/K)$. 

\paragraph{Finite vectors.}

Set $\Gamma = \Gal(K_\infty/K)$. The second step is Sen's theory is the 
study of $\hat K_\infty$-semi-linear representations of $\Gamma$.
To this attempt, Sen defines the subspace of \emph{finite} vectors
as follows.

\begin{deftn}
\label{def:finitevector}
Let $W \in \Rep^\fin_{\hat K_\infty}(\Gamma)$.
A vector $v \in W$ is \emph{finite} if the $K_\infty$-subspace 
of $W$ generated by the $gv$ for $g$ varying in $\Gamma$ is
finite dimensional over $K_\infty$.
\end{deftn}

As an example, the subspace of finite vectors of the semi-linear
representation $\hat K_\infty$ itself is $K_\infty$. In general,
one easily checks that the subspace of finite vectors is a vector
space over $K_\infty$. 

\begin{prop}
\label{prop:finitevector}
Let $W \in \Rep^\fin_{\hat K_\infty}(\Gamma)$.
Then, there exist an integer $r$ and
a basis $(v_1, \ldots, v_d)$ of $W$ with the property that the
$K_r$-span of the $v_i$'s is stable under the $\Gamma$-action.
\end{prop}

\begin{rem}
Obviously, the $v_i$'s of Proposition~\ref{prop:finitevector}
are finite in the sense of Definition~\ref{def:finitevector}.
Therefore, we deduce that the subspace of finite vectors of $W$
generates $W$ as a $\hat K_\infty$-vector space. Finite vectors
are then numerous.
\end{rem}

\begin{proof}[Proof of Proposition~\ref{prop:finitevector}]
Let $c_2$ be the constant of Proposition~\ref{prop:H90cont}.
It is harmless to assume that $c_2$ is an integer.
To simplify notation, we write $L = \hat K_\infty$. Let $\O_L$
be the ring of integers of $L$. We choose a $\O_L$-lattice $\O_W$ 
inside $W$.
By continuity, there exists an integer $r$ such that $g w \equiv
w \pmod {p^{c_2+1} \O_W}$ for all $g \in \Gal(K_\infty/K_r)$ and
all $w \in W$.
We choose and fix such an $r$. The group $\Gal(K_\infty/K_r)$ 
acts on $\O_W$ and on all the quotients $\O_W/p^n\O_W$ for $n
\in \N$.

We are going to construct, by induction of $n$, a sequence of families 
$(v_{1,n}, \ldots, v_{d,n})_{n \geq 1}$ of elements of $\O_W$ with the 
following properties:

\vspace{-2mm}

\begin{enumerate}[(i)]
\renewcommand{\itemsep}{0pt}
\item for all $n$, the family $v_{1,n}, \ldots, v_{d,n}$ is an
$\O_L$-basis of $\O_W$,
\item for all $n \geq 1$ and all $i$, $v_{i,n+1} \equiv v_{i,n}
\pmod{p^n \O_W}$
\item the $\O_{K_r}$-submodule of $\O_W/p^{n+c_2}\O_W$ generated
by the classes of the $v_{i,n}$'s ($1 \leq i \leq d$) is stable
under the $\Gal(K_\infty/K_r)$-action.
\end{enumerate}

\noindent
For $n=1$, we pick an arbitrary $\O_L$-basis $v_{1,1}, \ldots, v_{d,1}$ 
of $\O_W$. Since $\Gal(K_\infty/K_r)$ acts trivially on $\O_W/p^{c_2+1}
\O_W$, all the requirements are fulfilled. We now assume that $v_{1,n},
\ldots, v_{d,n}$ have been constructed. By the induction hypothesis, 
for all $i$, we can write
$\gamma_r v_{i,n} = v_{i,n} + \sum_{j=1}^d (\lambda_{i,j} + \varepsilon_{i,j})
v_{j,n}$
where the $\lambda_{i,j}$'s lie in $K_r$ and the $\varepsilon_{i,j}$'s 
have $p$-adic valuation at least $n + c_2$. Moreover, since the action of
$\gamma_r$ is trivial modulo $p^{c_2+1}$, we deduce $v_p(\lambda_{i,j}) 
\geq c_2+1$.
Let $R_r : L \to K_r$ be the Tate's normalized trace defined in
Remark~\ref{rem:normalizedtraces}. By Proposition~\ref{prop:H90cont}
(\emph{cf} also Remark~\ref{rem:H90Rr}), for all $i$ and $j$, there exists
$\mu_{i,j} \in L$ with $v_p(\mu_{i,j}) \geq n$ and
$\varepsilon_{i,j} = R_r(\varepsilon_{i,j}) + \gamma_r \mu_{i,j} - \mu_{i,j}$.
For all $i$, define $v_{i,n+1} = v_{i,n} - \sum_{j=1}^d \mu_{i,j} v_{j,n}$.
Since the $\mu_{i,j}$'s have all valuation at least $n$, the items
(i) and (ii) are fulfilled.
Besides, a simple computation gives:
$$\gamma_r v_{i,n+1} = v_{i,n+1} + 
\sum_{j=1}^d \big(\lambda_{i,j} + R_r(\varepsilon_{i,j})\big) 
\cdot v_{j,n} +
\sum_{j=1}^d \gamma_r \mu_{i,j} \cdot 
\big(v_{j,n} - \gamma_r v_{j,n}\big).$$
Since $\gamma_r$ acts trivially modulo $p^{c_2+1}$, the last summand
lies in $p^{n+c_2+1}\O_W$. 
Noting in addition that $v_{j,n} \equiv v_{j,n+1} \pmod
{p^n \O_W}$ and that the $\lambda_{i,j}$'s are all divisible
by $p^{c_2+1}$, we obtain the congruence:
$$\gamma_r v_{i,n+1} \equiv v_{i,n+1} + 
\sum_{j=1}^d \big(\lambda_{i,j} + R_r(\varepsilon_{i,j})\big) 
\cdot v_{j,n+1}
\pmod {p^{n+c_2+1}\O_W}$$
from which the item (iii) follows.

Passing to the limit, we obtain an $L$-basis $v_1, \ldots, v_d$
of $W$ whose $K_r$-span is stable under the action of $\Gal(K_r/K)$.
It remains to prove that it is stable under the whole action of
$\Gamma$. Let $M_0$ and $M_r$ be the matrices that gives the action
of $\gamma_0$ and $\gamma_r$ on $L$ respectively, that are:
\begin{align*}
  (\begin{matrix} \gamma_0 v_1 & \cdots & \gamma_0 v_d \end{matrix})
& = (\begin{matrix} v_1 & \cdots &  v_d \end{matrix}) \cdot M_0 \\
  (\begin{matrix} \gamma_r v_1 & \cdots & \gamma_r v_d \end{matrix})
& = (\begin{matrix} v_1 & \cdots &  v_d \end{matrix}) \cdot M_r.
\end{align*}
We do know that $M_r$ has all its entries in $K_r$ and we want to
prove that the same holds for $M_0$. Actually, from our construction 
of the $v_i$'s, we know further that $M_r$ has integral entries and
that it is congruent to the identity matrix modulo $p^{c_2+1}$.
From the commutation of $\gamma_0$ and $\gamma_r$, we derive the
relation $M_0 \cdot \gamma_0 M_r = M_r \cdot 
\gamma_r M_0$. Define $C = R_r(M_r) - M_r$ where $R_r$ is
the Tate's normalized trace. We want to prove that $C$ vanishes.

Since $R_r$ commutes with $\gamma_r$, we have the relation $C \cdot 
\gamma_0 M_r = M_r \cdot \gamma_r C $, from which we derive
$\gamma_r C - C = M_r^{-1} \cdot C \cdot \gamma_0 M_r$. Set 
$N = M_r^{-1} \cdot C \cdot \gamma_0 M_r$ and let $v$ be the 
smallest valuation of an entry of $C$. We assume by contradiction
that $v$ is finite. The fact that $M_r \equiv I_d \pmod{p^{c_2+1}}$
implies that $N$ is divisible by $p^{v+c_2+1}$.
By unicity in Proposition~\ref{prop:H90cont} (and 
Remark~\ref{rem:H90Rr}), we deduce that $C$ must be divisible by 
$p^{v+1}$. This contradicts the definition of $v$.
\end{proof}

\paragraph{Sen's operator.}

We now put together the results we have established before.
Let $W \in \Rep^\fin_{\Cp}(G_K)$. 
We define $\hat W_\infty = W^{\Gal(\bar K/K_\infty)}$ and 
let $W_\infty$ be the subspace of finite vectors of $\hat W_\infty$.
Combining Propositions~\ref{prop:descendKinfty} 
and~\ref{prop:finitevector}, we find that $W$ admits a $\Cp$-basis
consisting of elements of $W_\infty$. In other words, the canonical
mapping $\Cp \otimes_{K_\infty} W_\infty \to W$ is an isomorphism.
The action of $G_K$ on $W$ is then entirely determined by the
action of $\Gamma$ of $W_\infty$. Using the particularly simple
structure of $\Gamma$, it is possible to describe its action even 
more concretely. 

More precisely, we consider an integer $r$ and a basis
$v_1, \ldots, v_d$ of $\hat W_\infty$ such that the $K_r$-vector
space $W_r = K_r v_1 \oplus \cdots \oplus K_r v_d$ is stable
under the action of $\Gamma$ (or equivalently, $G_K$). 
For $g \in G_K$, we shall denote by $\rho_W(g)$ the endomorphism of 
$W_r$ given by the action of $g$. Note that $\rho_W(g)$ is $K_r$-linear 
as soon as $g \in \Gal(\bar K/K_r)$. In particular $\rho_W(\gamma_s)$ is 
linear whenever $s \geq r$.
Since $\gamma_s$ converges to the identity in $\Gamma$, the 
logarithm of $\rho_W(\gamma_s)$ is well defined for $s$ sufficiently
large. Moreover, we have the relation
$p \cdot \log \rho_W(\gamma_{s+1}) = \log \rho_W(\gamma_s)$ as soon as
the logarithm $\rho_W(\gamma_s)$ is defined. 
The sequence $p^{-s} \log \rho_W(\gamma_s)$ is then ultimately 
constant. Sen's operator $\Phi_W$ is defined as the limit of 
this sequence:
$$\Phi_W = \lim_{s \to \infty} \frac{\log \rho_W(\gamma_s)}{p^s}.$$
We extend $\Phi_W$ to $W_\infty$ by $K_\infty$-linearity.
This extension is canonical in the sense that it does not depend on
the choice of~$r$.
Besides, the exponential map can be used to reconstruct the
representation $W$ we started with, at least on a finite index
subgroup of $G$. Precisely, there exists an integer $s$ such that:
\begin{equation}
\label{eq:expSen}
\rho(g) = \exp\big(\alpha(g) \Phi_W\big)
\quad \text{for all } g \in \Gal(\bar K/K_s)
\end{equation}
where we recall that $\alpha : G_K \to \Zp$ was the character
defining the isomorphism between $\Gal(K_\infty/K)$ and $\Zp$.
From \eqref{eq:expSen}, it follows that the action of $\Gal(\bar K/K_s)$ 
on $W$ can be entirely reconstructed by extending the $\rho(g)$'s 
to $W$ using semi-linearity.

\begin{ex}
\label{ex:Sendim2}
Consider the representation $V$ given by:
$$G_K \to \GL_2(\Qp), \quad g \mapsto 
\left(\begin{matrix} 1 & \alpha(g) \\ 0 & 1 \end{matrix}\right)$$
already discussed in Example~\ref{ex:HTdim2}. Set $W = \Cp \otimes_{\Qp}
V$. It is easily checked that $W^{\Gal(\bar K/K_\infty)} = \hat 
K_\infty^2$ and its subspace of finite vectors is $K_\infty^2$.
Sen's operator $\Phi_V$ is the nilpotent linear morphism 
$(x,y) \mapsto (y,0)$.
\end{ex}

Sen's operator exhibits very interesting properties. Below,
we state the most important ones.

\begin{prop}
\label{prop:SenoverK}
We keep the above notations.
Sen's operator $\Phi_W$ is defined over $K$, in the sense 
that $W_\infty$ admits a basis in which the matrix of $\Phi_W$
has coefficients in $K$.
\end{prop}

\begin{proof}
This follows from the fact that $\Phi_W$ commutes with the action
of $\Gamma$.
\end{proof}

Let $\Sen(K,K_\infty)$ denote the category of finite dimensional
$K_\infty$-vector spaces equipped with an endomorphism defined 
over $K$ (in the sense of Proposition~\ref{prop:SenoverK}). 
The construction $W \mapsto (W_\infty, \Phi_W)$ defines a functor 
$\mathcal S : \Rep^\fin_{\Cp}(G_K) \to \Sen(K, K_\infty)$. Indeed
a morphism $f : W \to W'$ in the category $\Rep^\fin_{\Cp}(G_K)$
necessarily maps $W_\infty$ to $W'_\infty$ and commutes with 
Sen's operators on both sides because it commutes with the
Galois action.
The functor $\mathcal S$ commutes with direct sums, while its behavior 
under tensor products is governed by the Leibniz rule:
\begin{align*}
(W \otimes W')_\infty & = W_\infty \otimes W'_\infty \\
\Phi_{W \otimes W'} & = \Phi_W \otimes \id_{W'} + \id_W \otimes \Phi_{W'}.
\end{align*}
Moreover, the functor $\mathcal S$ is faithful. Indeed, assume that
we are given $W, W' \in \Rep^\fin_{\Cp}(G_K)$, together with a 
morphism $f : W \to W'$ such that $\mathcal S(f) = 0$. Then, by
assumption, $f$ vanishes on the subspace $W_\infty$. Since the latter 
generates $W$ as a $\Cp$-vector spaces, one must have $f = 0$.
In general, $\mathcal S$ is not full. However, it detects 
isomorphisms as shown by the next proposition.

\begin{prop}
\label{prop:conservativeSen}
Let $W, W' \in \Rep^\fin_{\Cp}(G_K)$. We assume that
$\mathcal S(W)$ and $\mathcal S(W')$ are isomorphic in
$\Sen(K,K_\infty)$.
Then $W$ and $W'$ are isomorphic in $\Rep^\fin_{\Cp}(G_K)$.
\end{prop}

\begin{proof}
Let $f : W_\infty \to W'_\infty$ be an isomorphism commuting
with Sen's operators. By $\Cp$-linearity, $f$ extends to an
isomorphism of $\Cp$-vector spaces $f : W \to W'$. Moreover, 
thanks to formula \eqref{eq:expSen}, there exists an integer $s$ 
such that $f$ is $\Gal(\bar K/K_s)$-equivariant.

Let $V$ be the space of $\Gal(\bar K/K_s)$-equivariant $\Cp$-linear
morphisms from $W$ to $W'$. It is endowed with a canonical action 
of $\Gal(K_s/K)$ and thus appears as an object in the category
$\Rep^\fin_{K_s}(\Gal(K_s/K))$. By Hilbert's theorem 90 (\emph{cf}
Theorem \ref{theo:H90}), $V$ admits a basis $(f_1, \ldots, f_m)$
of fixed vectors. In other words the $f_i$'s are $G_K$-equivariant
morphisms $W \to W'$.
It remains to prove that a suitable $K$-linear combination of the
$f_i$'s is invertible. 
For this, we consider the $m$-variate polynomial defined by:
$$P(t_1, t_2, \ldots, t_m) = 
\det(t_1 f_1 + t_2 f_2 + \cdots + t_m f_m).$$
We know that $P$ is not the zero polynomial because the $K_s$-span
of the $f_i$'s contains an isomorphism (namely $f$). Since $K$ is
an infinite field, $P$ cannot vanish everywhere on $K^m$. Hence
there exist $t_1, \ldots, t_m \in K$ such that $t_1 f_1 +
\cdots + t_m f_m$ is an isomorphism.
\end{proof}

\begin{cor}
A representation $W \in \Rep^\fin_{\Cp}(G_K)$ is trivial if and
only if Sen's operation $\Phi_W$ vanishes.
\end{cor}

\begin{proof}
It suffices to apply Proposition~\ref{prop:conservativeSen} with 
$W' = \Cp^{\dim W}$.
\end{proof}

We conclude our exposition of Sen's theory by noticing that 
Sen's operator is closely related to the notion of Hodge--Tate
representations. 
Precisely, a representation $V \in \Rep^\fin_{\Qp}(G_K)$ is
Hodge--Tate if and only if the operator $\Phi_{\Cp \otimes_{\Qp}
V}$ is semi-stable with eigenvalues in $\Z$, these eigenvalues 
being the Hodge--Tate weights of $V$. (Combine Examples~\ref{ex:HTdim2} 
and~\ref{ex:Sendim2} for an illustration of this property.)
Given a general $W \in \Rep^\fin_{\Cp}(G_K)$, the eigenvalues of 
$\Phi_W$ are sometimes called the \emph{generalized Hodge--Tate weights} 
of $W$.


\section{Two refined period rings: $\Bcrys$ and $\BdR$}
\label{sec:dRcrys}

Previously, we have studied the period rings $\Cp$ and $B_\HT$ and 
discussed the attached notion of Hodge--Tate representations.
In the present section, we introduce two new period rings, called
$\Bcrys$ and $\BdR$. As we shall see in \S\ref{sec:crysdRrep},
these rings have a deeper arithmetical and geometrical content that 
$\Cp$ and $B_\HT$.

The definition of $\Bcrys$ and $\BdR$ is a bit elaborated and 
occupies all this section. In order to ease the task of the reader, we 
devote two short paragraphs below to collect the most important 
properties of $\Bcrys$ and $\BdR$ and sketch the main steps of
their construction.

Before this, we need to recall and introduce some notations. 
Throughout this section, $K$ will continue to refer to a finite 
extension of $\Qp$. Its ring of integers (resp. its residuel field) is 
denoted by $\O_K$ (resp. $k$). We define $K_0 = W(k)[\frac 1 p]$; it is 
the maximal unramified extension of $\Qp$ included in $K$. We fix an 
algebraic closure $\bar K$ of $K$ and set $G_K = \Gal(\bar K/K)$. 
Observe that $\bar K$ is also an algebraic closure of $\Qp$ and hence 
does not depend on $K$. We let $K_0^\ur$ (resp. $K^\ur$) be the maximal 
unramified extension of $K_0$ (resp. of $K$) inside $\bar K$. Since 
$K_0$ is unramified over $\Qp$, $K_0^\ur$ is also the maximal extension 
of $\Qp$ inside $\bar K$ and thus is also independent of $K$. We let 
also $\Cp$ denote the $p$-adic completion of $\bar K$. Finally, we 
choose and fix once for all a uniformizer $\pi$ of $K$.

\paragraph{Main properties of $\Bcrys$ and $\BdR$.}

As discussed in \S\ref{ssec:motivations},
the original idea behind the definition of $\Bcrys$ is the wish
to design a variant of Barsotti-type spaces (the $\mathcal B$ of
\S\ref{ssec:motivations}) which includes the tannakian formalism.
On the geometric side, a nice tannakian framework in which 
$p$-divisible groups naturally arise is crystalline cohomology.
Indeed, in many contexts, crystalline cohomology provides powerful
invariants that can be used to classify $p$-divisible groups (and 
more generally finite flat group schemes)~\cite{BBM}.
We then expect the ring $\Bcrys$ to have some ``crystalline 
nature'' and to be eventually related to crystalline cohomology.
Apart from this, recall that another motivation of $p$-adic Hodge 
theory is to compare étale cohomology with de Rham cohomology. The
period ring making the comparison possible---namely $\BdR$---then 
needs to be deeply related to de Rham cohomology.
The algebraic structure of $\Bcrys$ and $\BdR$ is guided by the above 
general expectations: the ring $\Bcrys$ (resp. $\BdR$) will carry, as 
much as possible, the same structures and exhibit similar properties as 
the crystalline (resp. de Rham) cohomology.

Below, we list the main features of $\Bcrys$ and $\BdR$ and, when it 
is possible, we make the parallel with the corresponding properties 
of the cohomology. We start with $\BdR$:

\vspace{-2mm}

\begin{itemize}
\renewcommand{\itemsep}{0pt}
\item $\BdR$ is a discrete valuation field with residue field
$\Cp$;
\item $\BdR$ is an algebra over $\bar K{\cdot}\hat K^\ur$,
but \emph{not} over $\Cp$ (with a defining morphism preserving
the Galois action);
\item $\BdR$ is equipped with a filtration $\Fil^m \BdR$ (which
is nothing but the canonical filtration given by the valuation);
this filtration corresponds to the de Rham filtration on the
cohomology;
\item $\BdR$ has a distinguished element $t$ on which Galois
acts by multiplication by the cyclotomic character; moreover $t$
is a uniformizer of $\BdR$, so that $\Fil^m \BdR = t^m \BdR$;
\item the graded ring of $\BdR$ is $B_\HT = \Cp[t,t^{-1}]$;
\item $(\BdR)^{G_K} = K$; this property corresponds to the fact
that the de Rham cohomology is a vector space over $K$.
\end{itemize}

\noindent
And now for $\Bcrys$:

\vspace{-2mm}

\begin{itemize}
\renewcommand{\itemsep}{0pt}
\item $\Bcrys$ is an algebra over $\hat K_0^\ur$;
\item $\Bcrys$ is equipped with a Frobenius, which is a ring
homomorphism $\varphi : \Bcrys \to \Bcrys$; this structure
corresponds to the action of the Frobenius on the crystalline
cohomology;
\item there is a canonical embedding $K \otimes_{K_0} \Bcrys 
\hookrightarrow \BdR$; this property corresponds to the fact
that the crystalline cohomology defines a canonical $K_0$-structure inside
the de Rham cohomology (this is Hyodo--Kato isomorphism);
\item the distinguished element $t$ of $\BdR$ is in $\Bcrys$;
\item $(\Bcrys)^{G_K} = K_0$; this property corresponds to the 
fact that the de Rham cohomology is a vector space over $K_0$;
\item $\big(\Bcrys \cap \Fil^0 \BdR\big)^{\varphi=1} = \Qp$ (the notation
``$\varphi{=}1$'' means that we are taking the fixed points
under the Frobenius).
\end{itemize}

\paragraph{Sketch of the construction.}

The starting point of the construction of $\Bcrys$ and $\BdR$ is the 
introduction of the rings $\Ainf$ and $\Bpinf = \Ainf[1/p]$. One may 
think of $\Ainf$ as the universal thickening of $\O_{\bar K} / 
p\O_{\bar K}$; it is obtained \emph{via} a general process (detailed 
in \S\ref{ssec:Ainf}) involving a 
perfectization mecanism as a first step and Witt vectors as a second 
step. Beyond this purely algebraic construction, it is 
important to notice that the ring $\Bpinf$ has a strong geometrical 
interpretation. Indeed as observed first by Colmez and then by 
Fargues--Fontaine~\cite{fargues-fontaine} and Scholze~\cite{scholze-weinstein}, $\Bpinf$ 
appears at a mixed characteristic analogue\footnote{This analogy has 
been placed in the framework of Huber geometry by Scholze in~\cite{scholze-weinstein} and 
then takes a very substantial meaning. However, for this article, it 
will be sufficient to keep in mind that elements of $\Bpinf$ behave 
like analytic functions over a nonarchimedian base.} of the ring of 
bounded analytic functions on the open unit disc.
Moreover $\Bpinf$ is equipped with a Frobenius (coming from the
general theory of Witt vectors) and a distinguished geometric point, 
which is materalized by a surjective ring homomorphism $\theta : 
\Bpinf \to \Cp$.

Following the crystalline formalism, we then define the ring $\Bpcrys$ 
as the completion of the divided powers envelope of $\Bpinf$ with 
respect to the ideal $\ker\theta$. 
Unfortunately, $\Bpcrys$ does not have a nice geometrical
interpretation, in the sense that it is not the ring of analytic 
functions on a smaller domain.
In order to tackle this difficulty, we introduce (following Colmez)
some variants 
of $\Bpcrys$. Precisely, given a real parameter $\mu \geq 1$, one
considers the rings $\Bpmu$'s of analytic functions defined over 
some annulus $D_\mu$ included in the open unit disc, and containing
the distinguished point~$\theta$.
The $\Bpmu$'s are closely related to $\Bpcrys$ (we have inclusions
in both directions), so that it is often safe to replace the latter
by the formers.

Another important feature of $\Bpcrys$ and the $\Bpmu$'s is
that they contain a period of the cyclotomic character, that is an
element $t$ on which $G_K$ acts by multiplication by the cyclotomic
character. Geometrically, the divisor of $t$ is the orbit of the point
$\theta$ under the action of the Frobenius, that is the union of
all point $\theta \circ \varphi^n$ for $n$ varying in $\Z$.
The presence of $t$ in $\Bpmu$ will eventually ensure the 
admissibility of the representation $\Qp(\chi_\cycl^{-1})$.
In order to make $\Qp(\chi_\cycl)$ admissible as well (which is of 
course something we really want to have), we need $t$ to be a unit. So 
we finally define $\Bmu = \Bpmu [\frac 1 t]$ and the construction of 
$\Bcrys$ is now complete.

As for the field $\BdR$, it is defined as the fraction field of the 
completion of the local field of $\Bpinf$ (or equivalently, $\Bpmu$) 
at the special point~$\theta$. The filtration on $\BdR$ is nothing but 
the canonical filtration given by the order of the zero (or the pole) 
at~$\theta$.

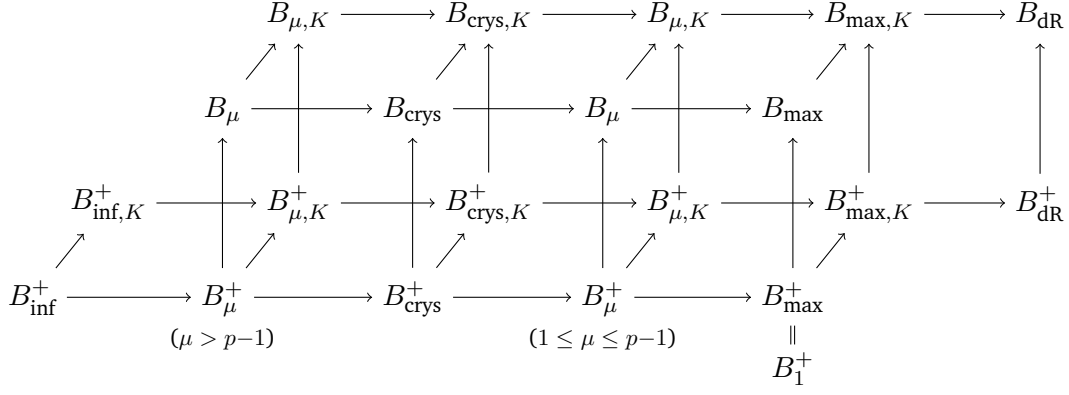
\begin{figure}
\hfill
\begin{tikzpicture}[xscale=2.5,yscale=2.5]

\node[scale=0.8] at (1,0.78) { \ph ($\mu > p{-}1$) };
\node[scale=0.8] at (3,0.78) { \ph ($1 \leq \mu \leq p{-}1$) };
\node at (4,0.62) { \ph $B^+_1$ };
\draw (3.99,0.85)--(3.99,0.75);
\draw (4.01,0.85)--(4.01,0.75);

\node (Bpinf) at (0,1) { \ph $\Bpinf$ };
\node (Bpmu) at (1,1) { \ph $\Bpmu$ };
\node (Bpcrys) at (2,1) { \ph $\Bpcrys$ };
\node (Bpmu2) at (3,1) { \ph $\Bpmu$ };
\node (Bpmax) at (4,1) { \ph $\Bpmax$ };

\begin{scope}[xshift=4mm,yshift=5mm]
\node (BpinfK) at (0,1) { \ph $\BpinfK$ };
\node (BpmuK) at (1,1) { \ph $\BpmuK$ };
\node (BpcrysK) at (2,1) { \ph $\BpcrysK$ };
\node (Bpmu2K) at (3,1) { \ph $\BpmuK$ };
\node (BpmaxK) at (4,1) { \ph $\BpmaxK$ };
\node (BpdR) at (4.9,1) { \ph $B^+_{\dR}$ };
\end{scope}

\node (Bmu) at (1,2) { \ph $\Bmu$ };
\node (Bcrys) at (2,2) { \ph $\Bcrys$ };
\node (Bmu2) at (3,2) { \ph $\Bmu$ };
\node (Bmax) at (4,2) { \ph $\Bmax$ };

\begin{scope}[xshift=4mm,yshift=5mm]
\node (BmuK) at (1,2) { \ph $\BmuK$ };
\node (BcrysK) at (2,2) { \ph $\BcrysK$ };
\node (Bmu2K) at (3,2) { \ph $\BmuK$ };
\node (BmaxK) at (4,2) { \ph $\BmaxK$ };
\node (BdR) at (4.9,2) { \ph $\BdR$ };
\end{scope}

\draw[->] (Bpinf)--(Bpmu);
\draw[->] (Bpmu)--(Bpcrys);
\draw[->] (Bpcrys)--(Bpmu2);
\draw[->] (Bpmu2)--(Bpmax);
\draw[->] (Bmu)--(Bcrys);
\draw[->] (Bcrys)--(Bmu2);
\draw[->] (Bmu2)--(Bmax);

\draw[->] (BpinfK)--(BpmuK);
\draw[->] (BpmuK)--(BpcrysK);
\draw[->] (BpcrysK)--(Bpmu2K);
\draw[->] (Bpmu2K)--(BpmaxK);
\draw[->] (BpmaxK)--(BpdR);
\draw[->] (BmuK)--(BcrysK);
\draw[->] (BcrysK)--(Bmu2K);
\draw[->] (Bmu2K)--(BmaxK);
\draw[->] (BmaxK)--(BdR);

\draw[->] (Bpinf)--(BpinfK);
\draw[->] (Bpmu)--(BpmuK);
\draw[->] (Bpcrys)--(BpcrysK);
\draw[->] (Bpmu2)--(Bpmu2K);
\draw[->] (Bpmax)--(BpmaxK);
\draw[->] (Bmu)--(BmuK);
\draw[->] (Bcrys)--(BcrysK);
\draw[->] (Bmu2)--(Bmu2K);
\draw[->] (Bmax)--(BmaxK);

\draw[->] (Bpmu)--(Bmu);
\draw[->] (Bpcrys)--(Bcrys);
\draw[->] (Bpmu2)--(Bmu2);
\draw[->] (Bpmax)--(Bmax);
\draw[->] (BpmuK)--(BmuK);
\draw[->] (BpcrysK)--(BcrysK);
\draw[->] (Bpmu2K)--(Bmu2K);
\draw[->] (BpmaxK)--(BmaxK);
\draw[->] (BpdR)--(BdR);
\end{tikzpicture}
\hfill\null

\caption{Diagram of period rings; all arrows are injective}
\label{fig:periodrings}
\end{figure}

The diagram presented on Figure~\ref{fig:periodrings} summarizes the 
period rings we will define in this section and the relations between 
them. We see on this diagram that the $\Bpmu$'s and the $\Bmu$'s all 
have a variant denoted by an extra index $K$. They are defined by 
extending scalars from $K_0$ to $K$. These variants are interesting 
because, when $V$ is a $\Bmu$-admissible representation, the de Rham 
filtration becomes visible after extending scalars to $\BmuK$, which 
is much smaller and sometimes more tractable that $\BdR$.

\subsection{Preliminaries: the ring $\Bpinf$}
\label{ssec:Ainf}

In this subsection, we introduce the ring $\Bpinf$ which serves
as a common base upon which all the forthcoming constructions will 
be built.

\subsubsection{Perfectization}

Let $\varphi$ denote the Frobenius morphism $x \mapsto x^p$ acting on 
the quotient $\O_{\Cp}/p\O_{\Cp} \simeq \O_{\bar K}/p \O_{\bar K}$ and 
observe that $\varphi$ is a ring homomorphism since $\O_{\Cp}/p\O_{\Cp}$ 
is annihilated by $p$.

Let $\OCpflat$ be the limit of the projective system\footnote{This 
definition is s special case of a general construction (the \emph{tilt}) 
in the theory of perfectoid spaces~\cite{Sch12}. We refer to Andreatta 
and~al. lecture~\cite{andreatta} in this volume for an introduction to 
perfectoid spaces. The notations $p^\flat$, $\pi^\flat$ and $v_\flat$
that we will introduce later comes from the language of perfectoid
spaces.}:
$$
\O_{\Cp}/p\O_{\Cp} \stackrel \varphi{\longrightarrow}
\O_{\Cp}/p\O_{\Cp} \stackrel \varphi{\longrightarrow}
\O_{\Cp}/p\O_{\Cp} \stackrel \varphi{\longrightarrow}
\cdots \stackrel \varphi{\longrightarrow}
\O_{\Cp}/p\O_{\Cp} \stackrel \varphi{\longrightarrow} \cdots$$
Concretely, an element of $\OCpflat$ is a sequence $(\xi_n)_{n 
\geq 0}$ of elements of $\O_{\Cp}/p\O_{\Cp}$ satisfying the following 
compatibility property: $\xi_{n+1}^p = \xi_n$ for all $n \geq 0$.
Clearly, $\OCpflat$ is a ring of characteristic~$p$. 

In a slight abuse of notation, we continue to write $\varphi$ for the 
Frobenius acting on $\OCpflat$. Over this ring, it is an 
isomorphism, its inverse being given by the shift map $(\xi_0, \xi_1, 
\xi_2, 
\ldots) \mapsto (\xi_1, \xi_2, \xi_3, \ldots)$. We sometimes say that 
$\OCpflat$ is the \emph{perfectization} of $\O_{\Cp}/p\O_{\Cp}$.
Moreover $\OCpflat$ is endowed with an action of $G_K$ coming
from its natural action on $\O_{\Cp}$.

\paragraph{Some distinguished elements.}

Choose a primivite $p$-root of unity in $\O_{\bar K}$ and denote it by 
$\varepsilon_1$. Similarly, choose a $p$-th root of $\varepsilon_1$ and 
denote it by $\varepsilon_2$; obviously, $\varepsilon_2$ is a primitive 
$p^2$-th root of unity. Repeating inductively this process, we construct 
elements $\varepsilon_3, \varepsilon_4, \ldots \in \O_{\bar K}$ such 
that $\varepsilon_{n+1}^p = \varepsilon_n$ for all $n$. 
Let $\bar\varepsilon_n \in \O_{\bar K}/p\O_{\bar K}$ be the image of 
$\varepsilon_n$. The compabitility property ensures that the sequence 
$(1, \bar\varepsilon_1, \bar\varepsilon_2, \ldots)$ defines an element 
in $\OCpflat$; we shall denote it by $\Ueps$.
We emphasize that $\Ueps$ does depend on the choice of the 
$\varepsilon_n$'s. However, the dependency is easy to write down
explicitely: if $(\varepsilon'_n)_{n \geq 0}$ is another compatible 
sequence of primitive $p^n$-th roots of unity, one can always find an 
element $g \in G_K$ such that $\varepsilon'_n = g \varepsilon_n = 
\varepsilon_n^{\chi_\cycl(g)}$. Hence the element of $\OCpflat$ 
defined the $\varepsilon'_n$'s is $\Ueps^{\chi_\cycl(g)}$. In what 
follows, we fix once for all an element $\Ueps$ as above.

In a similar fashion, we choose a compatible system $(p_n)_{n \geq 1}$
of $p^n$-root of $p$, \emph{i.e.} $p_1^p = p$ and $p_{n+1}^p = p_n$
for all $n \geq 1$. If $\bar p_n \in \O_{\bar K}/p\O_{\bar K}$ is 
the reduction of $p_n$ modulo $p$, the sequence $(0, \bar p_1, 
\bar p_2, \ldots)$ defines an element of $\OCpflat$ that we
will denote by $\Up$. Again, $\Up$ depends on the choice of the 
$p_n$'s but we can check that another choice would finally lead to 
an element of the form $\Up{\cdot}\Ueps^a$ for some $a \in \Zp$.
The same construction works more generally if we start for any 
element $x \in \O_{\Cp}$ in place of $p$; it leads to an element
$x^\flat \in \OCpflat$, which is well defined up to 
multiplication by $\Ueps^a$ with $a \in \Zp$.
Besides $\Up$, we will fix a choice of $\Upi$ (where we recall that 
$\pi$ is a fixed uniformizer of $K$) for future use.

\paragraph{Valuation.}

The ring $\OCpflat$ is equipped with a derivation $v_\flat$
that we are going to define now. We start with the following
observation: if $x$ is a nonzero element in $\O_{\Cp}/p\O_{\Cp}$, 
the $p$-adic valuation of $\hat x$ does not depend on the lifting 
$\hat x$ of $x$. The valuation $v_p$ then induces a well defined
function $v_p : \O_{\Cp}/p\O_{\Cp} \to \Q \cup \{+\infty\}$ where
we agree that $v_p(0) = +\infty$ as usual.
For $\xi = (\xi_n)_{n \geq 0}$ in $\OCpflat$, we define:
$$v_\flat(\xi) = \lim_{n \to \infty} p^n v_p(\xi_n).$$
The compatibility condition $\xi_{n+1}^p = \xi_n$ implies that
the sequence $\big(p^n v_p(\xi_n)\big)_{n \geq 0}$ is ultimately
constant; so the limit is well defined. The function $v_\flat$
satisfies the following properties for $\xi, \xi' \in 
\OCpflat$:

\vspace{-2mm}

\begin{enumerate}[(1)]
\renewcommand{\itemsep}{0pt}
\item $v_\flat(\OCpflat) = \Q \cup \{+\infty\}$,
\item $v_\flat(\xi) = \infty$ if and only if $\xi = 0$,
\item $v_\flat(\xi) = 0$ if and only if $\xi$ is invertible,
\item $v_\flat(\xi+\xi') \geq \min(v_\flat(\xi), v_\flat(\xi'))$ and
equality holds if $v_\flat(\xi) \neq v_\flat(\xi')$,
\item $v_\flat(\xi\xi') = v_\flat(\xi) + v_\flat(\xi')$.
\end{enumerate}

\vspace{-2mm}

\noindent
Combining (2) and (5), we find that $\OCpflat$ is a domain.
Indeed if $\xi$ and $\xi'$ are nonzero elements of $\OCpflat$,
then $v_\flat(\xi)$ and $v_\flat(\xi')$ are finite, and so 
$v_\flat(\xi\xi') = v_\flat(\xi) + v_\flat(\xi')$ is also finite. The
existence of $v_\flat$ implies that $\OCpflat$ is a local
ring with maximal ideal $\m_{\OCpflat}$ consisting of elements
of positive valuation. The residue field $\OCpflat/
\m_{\OCpflat}$ is canonically isomorphic to $\bar k$.
We observe in addition that the projection $\OCpflat \to
\bar k$ has a canonical splitting defined by:
$$a \mapsto \big([a] \mod p, [a^{1/p}] \mod p,
[a^{1/p^2}] \mod p, \ldots\big)$$
where the notation $[\cdot]$ stands for the Teichmuller 
representative.
Besides, the valuation $v_\flat$ equips $\OCpflat$ with a 
distance, and hence a topology. The Galois action on $\OCpflat$ 
preserves~$v_\flat$; in particular, it is continuous.

An easy consequence of the existence of a valuation is the following
result.

\begin{lem}
\label{lem:ocpflat}
The projection onto the first coordinate 
$\OCpflat \to \O_{\Cp}/p\O_{\Cp}$ induces an isomorphism
$\OCpflat/\Up\OCpflat \simeq \O_{\Cp}/p\O_{\Cp}$.
\end{lem}

\begin{proof}
Let $f : \OCpflat \to \O_{\Cp}/p\O_{\Cp}$, $(\xi_0, \xi_1,
\ldots) \mapsto \xi_0$. The surjectivity of $f$ is a consequence
of the fact that $\Cp$ is algebraically closed. On the other hand,
it is clear that the kernel of $f$ consists of elements $\xi$ such
that $v_\flat(\xi) \geq 1$. Since $v_\flat(\Up) = 1$, we deduce
that $\ker f$ is the principal ideal generated by $\Up$. This
proves the lemma.
\end{proof}

\subsubsection{Witt vectors}

We set $\Ainf = W(\OCpflat)$ (where $W(-)$ stands for the Witt 
vectors functor) and $\Bpinf = \Ainf[\frac 1 p]$.
For $x \in \OCpflat$, we let $[x]$ denote its representative 
Teichmüller in $\Ainf$. Since $\OCpflat$ is perfect, an element 
of $\Ainf$ can be written uniquely as a convergent series $\sum_{i 
\geq 0} [\xi_i]\:p^i$ with $\xi_i \in \OCpflat$ for all $i$. 
A similar decomposition holds for elements in $\Bpinf$: each such
element $x$ has a unique expansion of the form $\sum_{i \geq i_0}
[\xi_i]\:p^i$ (with $\xi_i \in \OCpflat$) where $i_0$ is a
(possibly negative) integer, which depends on $x$.

The inclusion $\bar k \to \OCpflat$ provides by functoriality a 
ring morphism $W(\bar k) \to \Ainf$.
Thus $\hat K_0^\ur$ embeds into $\Bpinf$.
The ring $\Ainf$ is a local ring whose maximal ideal is the kernel of 
the composition $\Ainf \to W(\bar k) \to \bar k$ where the first map is 
induced by the projection $\OCpflat \to \bar k$ and the second
map is the reduction modulo $p$. Concretely, it consists of series
$\sum_{i \geq 0} [\xi_i]\:p^i$ for which $\xi_0 \in \m_{\OCpflat}$.

\medskip

We set $\AinfK = \O_K \otimes_{W(k)} \Ainf$ and $\BpinfK
= K \otimes_{K_0} \Bpinf$. These tensors products make sense
because we saw that $\Ainf$ is an algebra over $W(k)$.
The elements of $\BpinfK$ have a canonical 
expansion of the form $\sum_{i \geq i_0} [\xi_i]\:\pi^i$ with $i_0
\in \Z$ and $\xi_i \in \OCpflat$ for all~$i \geq i_0$.
Moreover $\AinfK$ is a local ring and its maximal
ideal consists of series as above such that $\xi_0 \in \m_{\OCpflat}$;
its residue field is $\bar k$.

\paragraph{Additional structures.}

By definition of the Witt vectors, $\Bpinf$ carries an action of
a Frobenius, that we shall continue to call $\varphi$. On the above
representation, it is given by the simple formula:
\begin{equation}
\label{eq:FrobAinf}
\varphi\Bigg(\sum_{i=i_0}^\infty [\xi_i]\:p^i\Bigg) =
\sum_{i=i_0}^\infty [\xi_i^p]\:p^i
\qquad (i_0 \in \Z, \, \xi_i \in \OCpflat).
\end{equation}
We emphasize that $\varphi$ does not admit a \emph{canonical} 
extension to $\BpinfK$ as there is no canonical Frobenius on $K$.

The ring $\Bpinf$ is also equipped with an action of $G_K$ by
functorialiy of Witt vectors. Again, this action has a simple expression,
namely:
\begin{equation}
\label{eq:GaloisAinf}
g\Bigg(\sum_{i=i_0}^\infty [\xi_i]\:p^i\Bigg) =
\sum_{i=i_0}^\infty [g \xi_i]\:p^i
\qquad (i_0 \in \Z, \, \xi_i \in \OCpflat)
\end{equation}
for all $g \in G_K$. The $G_K$-action extends to $\BpinfK$
by letting $G_K$ act trivially on $\O_K$.

Finally, we equip $\Ainf$ and $\AinfK$ with the \emph{weak 
topology}, which is the topology defined by the ideal~$(p, [\Up])$ (or 
equivalently, by the ideal~$(p, [x])$ for any element $x \in 
\m_{\OCpflat}$). Concretely, if
$$x_n = \sum_{i=0}^\infty [\xi_{i,n}]\:p^i \in \Ainf
\quad \text{and} \quad
x = \sum_{i=0}^\infty [\xi_i]\:p^i \in \Ainf$$
the sequence $(x_n)_{n \geq 0}$ converges to $x$ if $\xi_{i,n} \to
\xi_i$ for each fixed index~$i \in \N$, and a similar property holds
for $\AinfK$.
The topology on $\Ainf$ induces a topology on the subset $p^{-v}
\Ainf$ of $\Bpinf$ for all~$v$. Gluing them, we obtain a topology
on $\Bpinf = \bigcup_{v \geq 0} p^{-v} \Ainf$. In concrete terms,
a sequence $(x_n)_{n \geq 0}$ of elements on $\Bpinf$ converges
to $x \in \Bpinf$ if and only if there exists an integer $v$ such
that $p^v x_n \in \Ainf$ for all~$n$ and $p^v x_n$ tends to $p^v
x$ in $\Ainf$ as $n$ goes to infinity.
The topology on $B^+_{K,\inf}$ is defined similarly.

From the above descriptions, it follows that the Frobenius 
acts continuously on $\Ainf$ and $G_K$ acts continuously on
$\Ainf$ and $\AinfK$.

\paragraph{Newton polygons.}

In~\cite{fargues-fontaine}, Fargues and Fontaine argue that elements of 
$\Ainf$ (resp. $\AinfK$) should be thought of as analytic functions 
of the variable $p$ (resp. $\pi$); indeed, they share many properties 
with bounded analytic functions on the open unit disc. Similarly, 
elements of $\Bpinf$ and $\BpinfK$ resemble to bounded analytic 
functions on the punctured open unit disc.

In particuler, there is a well-defined notion of Newton polygons for 
series in $\Bpinf$ and $\BpinfK$. Precisely, if $x = \sum_{i \geq 
i_0} [\xi_i]\:p^i \in \Ainf$, its \emph{Newton polygon} is defined as 
the convex hull in $\R^2$ of the points $(i, v_\flat(\xi_i))$ together 
with two points at infinity in the direction of the positive $x$-axis 
and the direction of the positive $y$-axis respectively.
Similarly, the Newton polygon of $x = \sum_{i \geq i_0} [\xi_i]\:\pi^i 
\in \AinfK$ is the convex hull of the points $(\frac i e, 
v_\flat(\xi_i))$ and the same points at infinity. 
Using that $\pi^e = up$ for some invertible element $u \in \O_K$, one 
easily proves that the above definition coincides with that of Newton 
polygons on $\Bpinf$ when $x$ is in $\Bpinf$. Let $\NP_\inf(x)$ denote the 
Newton polygon of $x \in \BpinfK$.

Fargues and Fontaine prove that Newton polygons satisfy many
excepted properties. For example, they are multiplicative in the
sense that $\NP_\inf(xy) = \NP_\inf(x) + \NP_\inf(y)$ where the plus 
sign on the right hand side denotes the Minkowski sum.
Moreover Fargues and Fontaine prove
an analogue of the Weierstrass preparation and factorization 
theorems in this context, showing that Newton polygons serve as a
guide for factorization in the rings $\Bpinf$ and $\BpinfK$
as they do for usual
analytic functions. We do not reproduce their proofs here because
we will only use Newton polygons for visualizing our forthcoming 
constructions, and not for proving results. In any case, we refer
to \cite[\S1--3]{fargues-fontaine} for many developments in this
direction.

We conclude this discussion by examining the action of the additional 
structures at the level of Newton polygons.
Since $G_K$ acts on $\OCpflat$ by isometries, it follows from the 
formula~\eqref{eq:GaloisAinf} that $\NP_\inf(gx) = \NP_\inf(x)$ whenever $g$ is
in $G_K$ and $x$ is in $\BpinfK$. 
As for Frobenius, formula~\eqref{eq:FrobAinf} shows that, for any $x \in 
\Bpinf$, we have $\NP_\inf(\varphi(x)) = \varphi_{\R^2}(\NP_\inf(x))$ where 
$\varphi_{\R^2} : \R^2 \to \R^2$ takes $(i,v)$ to~$(i,pv)$.

\subsubsection{The ``sharp'' construction}
\label{sssec:sharp}

In \S\ref{ssec:Ainf}, starting with $x \in \O_{\Cp}$, we 
have constructed an element $x^\flat \in \OCpflat$ (which was only 
well-defined up to multiplication by an element of the form $\Ueps^a$ 
with $a \in \Zp$). Let us recall more precisely that the element 
$x^\flat = (x_0 \mod p, x_1 \mod p, x_2 \mod p, \ldots)$ where
$x_0 = x$ and $x_{n+1}$ is a $p$-th root of $x_n$ for $n \geq 0$.

It turns out that the datum of $x^\flat$ entirely determines 
$x$. Precisely, if we write $x^\flat = (\xi_0, \xi_1, \xi_2, \ldots)$ 
and if we choose a lifting $\hat \xi_n \in \O_{\Cp}$ of $x_n$ for all $n$, 
we have $x = \lim_{n \to\infty} \hat \xi_n^{p^n}$
independently of the choices of the liftings. Indeed, following the
definitions, we find that $x_n \equiv \hat \xi_n \pmod p$ and then, 
raising to the $p^n$-th power, $x \equiv \hat \xi_n^{p^n} \pmod
{p^{n+1}}$. This motivates the following definition.

\begin{deftn}
For $\xi = (\xi_0, \xi_1, \xi_2, \ldots) \in \OCpflat$,
we put
$$\xi^\sharp = \lim_{n \to \infty} \hat \xi_n^{p^n}$$
where $\hat \xi_n$ is a lifting of $\xi_n$.
\end{deftn}

One checks immediately that the function $\OCpflat \to \O_{\Cp}$, 
$\xi \mapsto \xi^\sharp$ is surjective and multiplicative. Its
``kernel'' is the closed subgroup of $\OCpflat^\times$ 
generated by $\Ueps$; it is isomorphic to $\Zp$. Besides, we
observe that $v_p(\xi^\sharp) = v_\flat(\xi)$ for all $\xi
\in \OCpflat$ and that $\xi^\sharp$ is the Teichmuller
representative of $\xi$ if $\xi$ is in $\bar k$.
By the general properties of Witt vectors, the ``sharp'' function 
extends to a surjective homorphism of $\hat K_0^\ur$-algebras
$\theta: \Bpinf \to \Cp$ which commutes with the $G_K$-action.
Concretely, it is given by:
$$\theta : \quad \sum_{i=i_0}^\infty [\xi_i]\: p^i \,\, 
\mapsto \,\, \sum_{i=i_0}^\infty \xi_i^\sharp\: p^i \qquad 
(i_0 \in \Z, \, \xi_i \in \OCpflat).$$
Note that the latter series converges in $\Cp$ since its 
$i$-th summand is a multiple of~$p^i$.
The morphism $\theta$ extends by $K$-linearity to a surjective 
$G_K$-equivariant homomorphism of 
$\hat K^\ur$-algebras $\theta_K : \BpinfK \to \Cp$.

\begin{prop}
\label{prop:kertheta}
\begin{enumerate}[(i)]
\renewcommand{\itemsep}{0pt}
\item Let $z \in \Ainf$ be an element such that $\theta(z) = 0$ and 
$v_\flat(z \mod p) = 1$. Then $z$ generates $\Ainf \cap \ker 
\theta$, viewed as an ideal of $\Ainf$.
\item Let $z \in \AinfK$ be an element such that $\theta_K(z) = 0$ 
and $v_\flat(z \mod \pi) = \frac 1 e$. Then $z$ generates
$\AinfK \cap \ker \theta_K$, viewed as an ideal of $\AinfK$.
\end{enumerate}
\end{prop}

\begin{rem}
In particular, an element $z$ satisfying the condition of the
first item (resp. the second item) of Proposition~\ref{prop:kertheta}
is a generator of the ideal $\ker\theta$ (resp. $\ker\theta_K$).
\end{rem}

\begin{proof}[Proof of Proposition~\ref{prop:kertheta}]
Let $z \in \Ainf$ such that $\theta(z) = 0$ and $v_\flat(\zeta) = 1$ 
with $\zeta = z \mod p$.
Let $x \in \ker\theta \cap \Ainf$. Write
$x = \sum_{i\geq 0} [\xi_i] \: p^i$ with $\xi_i \in \OCpflat$. 
From $\theta(x) = 0$, we derive that $v_p(\xi_0^\sharp) \geq 1$ and
then $v_\flat(\xi_0) \geq 1$. From the assumption $v_\flat(\zeta) 
= 1$, we find that $\zeta$ divides $\xi_0$ in $\OCpflat$.
Thus, we can write $x = z y_0 + p x_1$ with $y_0, x_1 \in 
\Ainf$. From this above equality, we derive $\theta(x_1) = 0$
and we can then repeat the argument with $x_1$, ending up with
a writing of the form
$x = z{\cdot}(y_0 + p y_1) + p^2 x_2$ with $y_1, x_2 \in \Ainf$.
Repeating this process again and again, we construct a sequence
$(y_n)_{n \geq 0}$ of elements of $\Ainf$ such that:
$$x \equiv z\cdot(y_0 + p y_1 + \cdots + p^n y_n) \pmod{p^n \Ainf}$$
for all $n$. Passing to the limit we find that $x \in z \Ainf$,
which proves~(i).

The statement~(ii) is proved similarly.
\end{proof}

We remark that there do exist elements in $\Ainf$ satisfying the 
condition of Proposition~\ref{prop:kertheta}. The simplest one is 
$[\Up]{-}p$, which is then a generator of $\ker\theta$. 
Similarly $[\Upi]{-}\pi \in \AinfK$ satisfies the
condition of Proposition~\ref{prop:kertheta} and so is a generator
of $\ker \theta_K$.
Another generator of $\ker\theta$ is $E([\Upi])$ where $E$ is minimal 
polynomial of $\pi$ over $K_0$. Indeed, on the one hand, we have 
$\theta(E([\Upi])) = E(\theta([\Upi])) = E(\pi) = 0$ and, on the other 
hand, $E([\Upi])$ reduces modulo $p$ to the constant coefficient of $E$, 
which has valuation $1$.

The next proposition gives another quite interesting generator of
$\ker\theta$.

\begin{prop}
\label{prop:omega}
The element 
$$\omega = \frac{[\Ueps]-1}{[\Ueps^{1/p}]-1}
= [\Ueps^{1/p}] + [\Ueps^{1/p}]^2 + \cdots + [\Ueps^{1/p}]^{p-1}$$
satisfies the condition of Proposition~\ref{prop:kertheta}.(i).
\end{prop}

\begin{proof}
We want to check that $\theta(\omega) = 0$ and $v_\flat(\omega\mod p) = 1$. 
The first equality follows from the fact that $\theta([\Ueps]) = 1$ and 
the fact that $\theta([\Ueps^{1/p}])$ is a primitive $p$-th root of 
unity. Let us now prove that $v_\flat(\omega\mod p) = 1$. Reducing
modulo $p$, we find that $\omega \mod p = \frac{\Ueps-1}{\Ueps^{1/p}-1}$. 
Write $\Ueps = (\varepsilon_0, \varepsilon_1, \varepsilon_2, \ldots)$ 
where $\varepsilon_n$ is the reduction modulo $p$ of a primitive 
$p^n$-th root of unity. Coming back to the definition of~$v_\flat$, we 
find:
\begin{equation}
\label{eq:vflateta}
v_\flat(\omega \mod p) = 
\lim_{n \to \infty} p^{n+1} \cdot 
v_p\left(\frac{\varepsilon_n - 1} {\varepsilon_{n+1} - 1}\right).
\end{equation}
By the standard properties of the cyclotomic extension (\emph{cf}
\cite[Chap. IV, \S 4]{serre}), we know that the $p$-adic valuation of 
$\varepsilon_n{-}1$ is $\frac 1{p^n(p-1)}$. Injecting this in
\eqref{eq:vflateta}, we obtain $v_\flat(\omega \mod p) = 
\frac p{p-1} - \frac 1{p-1} = 1$.
\end{proof}

\begin{rem}
\label{rem:NPeps}
Since two generators of $\ker\theta$ differ by multiplication by a unit, 
they have to share the same Newton polygon up to translation by a 
horizontal vector. If in addition, they satisfy the conditions of 
Proposition~\ref{prop:kertheta}, the Newton polygons must coincide since 
they both admit $(0,1)$ as an extremal point.
Clearly, the Newton polygon of $[\Up] - p$ is the convex polygon whose 
vertices are $(0,+\infty)$, $(0,1)$, $(1,0)$ and $(+\infty,0)$. The 
Newton polygon of $\omega$ is then the same. Writing
\begin{equation}
\label{eq:prodeps}
[\Ueps]-1 = \prod_{n=0}^\infty \varphi^{-n}(\omega)
\end{equation}
and using the multiplicative properties of the Newton polygons, we
find that $\NP_\inf([\Ueps]-1)$ starts at $(0,\frac 1{p-1})$ and then
has a segment of length $1$ of slope $p^{-n}$ for each nonnegative
integer $n$ (\emph{cf} Figure~\ref{fig:NewtonUeps}).
\begin{figure}
\hfill
\begin{tikzpicture}[rotate=90,xscale=5,yscale=-2]
\draw[->] (-0.1,0)--(1.2,0);
\draw[->] (0,-0.5)--(0,5.1);
\draw[thick] (1.2,0)--(1,0)--(0.5,1)--(0.25,2)--(0.125,3)--(0.0625,4);
\draw[thick,dashed] (0.0625,4)--(0.03125,5);
\begin{scope}[dotted]
\draw (0.5,0)--(0.5,1)--(0,1);
\draw (0.25,0)--(0.25,2)--(0,2);
\draw (0.125,0)--(0.125,3)--(0,3);
\end{scope}
\node[below left,scale=0.8] at (0,0) { $0$ };
\node[scale=0.5] at (1,0) { $\bullet$} ;
\node[left,scale=0.8] at (1,0) { $\phantom {\frac p{p-1} ={}}\nu = \frac p{p-1}$ };
\node[left,scale=0.8] at (0.5,0) { $\nu/p$ };
\node[left,scale=0.8] at (0.25,0) { $\nu/p^2$ };
\node[left,scale=0.8] at (0.125,0) { $\nu/p^3$ };
\node[below,scale=0.8] at (0,1) { $1$ };
\node[below,scale=0.8] at (0,2) { $2$ };
\node[below,scale=0.8] at (0,3) { $3$ };
\node[scale=0.7,rotate=-51] at (0.75,0.61) { slope: $-1$ };
\node[scale=0.7,rotate=-33] at (0.4,1.56) { slope: $-1\hspace{-0.1ex}/\hspace{-0.1ex}p$ };
\end{tikzpicture}
\hfill\null

\caption{The Newton polygon of $[\Ueps] - 1$}
\label{fig:NewtonUeps}
\end{figure}
\end{rem}

\begin{prop}
\label{prop:kerthetaphi}
The element $[\Ueps]-1$ is a generator of the ideal
$\bigcap_{n \geq 0} \ker(\theta \circ \varphi^n)$.
\end{prop}

\begin{rem}
Proposition~\ref{prop:kerthetaphi} is not surprising after 
formula~\eqref{eq:prodeps}. Indeed if $x$ is such that $\theta
\circ \varphi^n(x) = 0$ for all $n \geq 0$, then $x$ must to divisible
by $\varphi^{-n}(\omega)$ for all $n \geq 0$. It is then reasonable
to expect to $[\Ueps]-1 = \prod_{n=0}^\infty \varphi^{-n}(x)$ 
divides $x$ since the Newton polygon of the factors do not share
any common slope (and thus the factors look pairwise coprime). It is 
possible to turn this vague idea into a rigourous proof. However, we 
prefer giving below a more direct argument, which is easier to write 
down.
\end{rem}

\begin{proof}
Clearly $[\Ueps]-1 \in \bigcap_{n \geq 0} \ker(\theta \circ \varphi^n)$.
Repeating the second part of the proof of 
Proposition~\ref{prop:kertheta}, we are reduced to show that any element 
$x \in \Ainf$ such that $\theta \circ \varphi^n(x) = 0$ verifies 
$v_\flat(x \mod p) \geq \frac p{p-1}$. From $\theta(x) = 0$, we
deduce that $x$ can be written $\omega x_1$ with $x_1 \in \Ainf$.
Since $\theta\circ\varphi(\omega) \neq 0$, we deduce that
$\theta\circ\varphi(x_1)$ must vanish. Therefore there exists $x_2
\in \Ainf$ such that $x_1 = \varphi^{-1} (\omega) x_2$, \emph{i.e.}
$x = \omega \varphi^{-1}(\omega) x_2$. By induction, we find that
$x$ has to be divisible by $x = \omega \varphi^{-1}(\omega) \cdots 
\varphi^{-n}(\omega)$ in $\Ainf$ for all $n$. Reducing modulo $p$,
this implies
$v_\flat(x \mod p) \geq 1 + \frac 1 p + \cdots + \frac 1{p^n}$
for all $n$. Passing to the limit, we find $v_\flat(x \mod p)
\geq \frac p{p-1}$ as expected.
\end{proof}

\subsection{The ring $\Bcrys$ and some variants}

In this subsection, we introduce the ring $\Bcrys$ and its variants 
$\Bmu$'s. The former is interesting because it fits very well in the 
crystalline framework and therefore is well suited for studying 
cohomology. 
Nevertheless, as we shall see, $\Bcrys$ does not behave
very well from the purely algebraic point of view. The $\Bmu$'s are
substitutes to $\Bcrys$ which share its most important features and, 
in addition, exhibit better algebraic (and analytic) properties, and 
hence are easier to work with.

\subsubsection{Divided powers}

Given $x \in \Ainf$, we denote by $\Ainf\left<x\right>$ the 
sub-$\Ainf$-algebra of $\Bpinf$ generated by the elements 
$\frac{x^n}{n!}$ for $n$ varying in $\N$. Obviously if $y$ divides $x$ 
in $\Ainf$, we have $\Ainf\left<x\right> \subset 
\Ainf\left<y\right>$. In particular, $\Ainf\left<x\right>$ only
depends on the principal ideal $x \Ainf$.
By the proof of Proposition~\ref{prop:kertheta}, we know that
$\Ainf \cap \ker\theta$ is a principal ideal of $\Ainf$. The
following definition then makes sense.

\begin{deftn}
We define $\Acrys$ as the $p$-adic completion of 
$\Ainf\left< z\right>$ where $z$ is some generator of the
ideal $\Ainf \cap \ker\theta$.
We set $\Bpcrys = \Acrys[\frac 1 p]$.
\end{deftn}

Rephrasing the definition, we can write:
$$\Acrys = \Ainf\left<\right.\![\Up]-p\!\left.\right>^\wedge
= \Ainf\left<\omega\right>^\wedge$$
where the exponent ``$\wedge$'' means the $p$-adic completion
and the element $\omega$ is the one of Proposition~\ref{prop:omega}.

\begin{lem}
For $x, y \in \Ainf$ with $x \equiv y \pmod {p \Ainf}$,
we have $\Ainf\left<x\right> = \Ainf\left<y\right>$.
\end{lem}

\begin{proof}
By symmetry, it is enough to prove that 
$\Ainf\left<x\right> \subset \Ainf\left<y\right>$, \emph{i.e.}
that $\frac{x^n}{n!} \in \Ainf\left<y\right>$ for all positive
integer~$n$. Writing $x = y + pz$ with $z \in \Ainf$, we have:
\begin{equation}
\label{eq:xnnfact}
\frac{x^n}{n!} = \frac{(y+pz)^n}{n!}
= \sum_{i=0}^n \frac{p^i z^i}{i!} \cdot \frac{y^{n-i}}{(n{-}i)!}.
\end{equation}
We recall that $v_p(i!) = \frac{i - s_p(i)}{p-1}$ where $s_p(i)$
denotes the sum of the digits of $i$ in radix $p$. In particular,
we observe that $v_p(i!) \leq i$, so that the fraction $\frac{p^i}
{i!}$ is in $\Zp$. The formula~\eqref{eq:xnnfact} then presents
$\frac{x^n}{n!}$ as an $\Ainf$-linear combination of elements of
the form $\frac{y^j}{j!}$ for $j$ between $0$ and $n$. Therefore,
$\frac{x^n}{n!} \in \Ainf\left<y\right>$ and we are done.
\end{proof}

The above lemma shows that $\Acrys = 
\Ainf\left<\right.\![\Up]\!\left.\right>^\wedge$. Since 
$\Ainf\left<x\right>$ depends only on the ideal generated by $x$, we 
also have $\Acrys = \Ainf\left<\right.\![\xi]\!\left.\right>^\wedge$ 
for any element $\xi \in \OCpflat$ with $v_\flat(\xi) = 1$.

\paragraph{Topology and additional structures.}

Since $\Acrys$ is defined as a $p$-adic completion, it is quite natural 
to endow $\Acrys$ (and $\Bpcrys$) with the $p$-adic topology. 
Noticing that we can obviously write $[\Up]^n = n! \cdot 
\frac{[\Up]^n}{n!}$, it follows from the defintion of $\Acrys$ that 
$[\Up]^n$ tends to zero when $n$ goes to infinity (since $n!$ goes to 
zero for the $p$-adic topology). 
In particular the inclusion $\Ainf \to \Acrys$ is continuous. 
Inverting~$p$, we find that the inclusion $\Bpinf \to \Bpcrys$ is 
continuous as well.

Besides, we observe that the Frobenius extends canonically to a ring 
homomorphism $\varphi : \Acrys \to \Acrys$. This can be checked by 
noticing that $\Ainf\left<\right.\![\Up]\!\left.\right>$ is stable 
under the Frobenius since $\varphi\big(\frac{[\Up]^n}{n!}\big) = 
[\Up]^{np-n} \cdot \varphi\big(\frac{[\Up]^n}{n!}\big)$.
Inverting~$p$, we obtain an extension of the Frobenius to $\Bpcrys$.
We shall continue to denote it by $\varphi$ in the sequel.
Similarly, the action of $G_K$ extends to $\Bpcrys$.

The embedding $W(\bar k) \to \Ainf \to \Acrys$ endows $\Acrys$ with
a structure of $W(\bar k)$-algebra. Similarly $\Bpcrys$ is an
algebra over $\hat K_0^\ur$. It then makes sense to define
$\AcrysK = \O_K \otimes_{W(k)} \Acrys$ and 
$\BpcrysK = K \otimes_{K_0} \Bpcrys = \AcrysK[\frac 1 p]$.

\subsubsection{Some analytic analogues of $\Acrys$}
\label{sssec:Amu}


In \S\ref{ssec:Ainf}, we saw that elements of $\Ainf$ admitted a nice
series expansion, allowing for an analytic interpretation of the
ring $\Ainf$.
To some extent, this point of view is also meaningful for $\Acrys$.
Indeed, it follows from 
$\Acrys =
\Ainf\left<\right.\![\Up]\!\left.\right>^\wedge$, that any
element $x \in \Acrys$ has a unique expansion of the form:
\begin{equation}
\label{eq:convAcrys}
\begin{array}{rl}
& \displaystyle
x = \sum_{i \in \Z} [\xi_i] \: p^i \qquad (\xi_i \in \OCpflat) 
\smallskip \\
\text{with} & 
v_\flat(\xi_i) - \nu(i) \geq 0
\quad \text{and} \quad
\displaystyle
\lim_{i \to -\infty} v_\flat(\xi_i) - \nu(i) = +\infty
\end{array}
\end{equation}
where, for $i \geq 0$, $\nu(i) = 0$ and, for $i < 0$, $\nu(i)$ denotes 
the smallest integer $n$ such that $v_p(n!) + i \geq 0$. From the
formula $v_p(n!) = \frac{n - s_p(n)}{p-1} = \frac n{p-1}$,
we derive that, for $i \ll 0$, we have the estimation:
\begin{equation}
\label{eq:estimnui}
- i\cdot(p-1) \leq \nu(i) \leq - i\cdot(p-1) + O\big(\log |i|\big).
\end{equation}
We insist on the fact that the term $O(\log |i|)$ is not bounded (it may 
have order of magnitude $(p{-}1){\cdot} \frac{\log|i|}{\log p}$); hence, 
we cannot replace $\nu(i)$ by $-i{\cdot}(p{-}1)$ 
in~\eqref{eq:convAcrys}. We will circumvent this difficulty later on.
Figure~\ref{fig:convAcrys} illustrates the convergence conditions
discussed above: the grey part is the region on which $v_\flat(\xi_i) 
- \nu(i) \geq 0$.

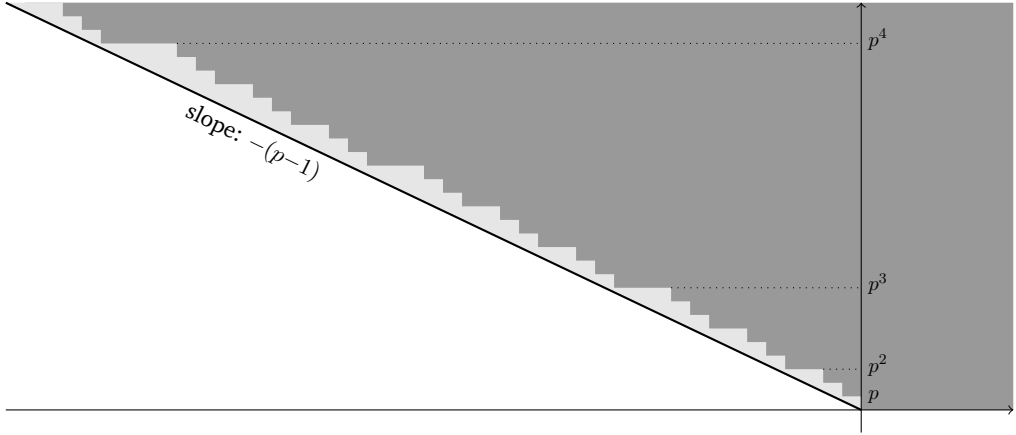
\begin{figure}
\hfill
\begin{tikzpicture}[rotate=90,xscale=0.06,yscale=-0.25]
\fill[black!10] (0,0)--(90,-45)--(90,0);
\fill[black!40]
      (0,8)--(0,0)
    --(3,0)--(3,-1)
    --(6,-1)--(6,-2)
    --(9,-2)--(9,-4)
    --(12,-4)--(12,-5)
    --(15,-5)--(15,-6)
    --(18,-6)--(18,-8)
    --(21,-8)--(21,-9)
    --(24,-9)--(24,-10)
    --(27,-10)--(27,-13)
    --(30,-13)--(30,-14)
    --(33,-14)--(33,-15)
    --(36,-15)--(36,-17)
    --(39,-17)--(39,-18)
    --(42,-18)--(42,-19)
    --(45,-19)--(45,-21)
    --(48,-21)--(48,-22)
    --(51,-22)--(51,-23)
    --(54,-23)--(54,-26)
    --(57,-26)--(57,-27)
    --(60,-27)--(60,-28)
    --(63,-28)--(63,-30)
    --(66,-30)--(66,-31)
    --(69,-31)--(69,-32)
    --(72,-32)--(72,-34)
    --(75,-34)--(75,-35)
    --(78,-35)--(78,-36)
    --(81,-36)--(81,-40)
    --(84,-40)--(84,-41)
    --(87,-41)--(87,-42)
    --(90,-42)--(90,-44)
    --(90,8);
\draw[->] (-5,0)--(90,0);
\draw[->] (0,-45)--(0,8);
\draw[thick] (0,0)--(90,-45);
\begin{scope}[dotted]
\draw (9,-2)--(9,0);
\draw (27,-10)--(27,0);
\draw (81,-36)--(81,0);
\end{scope}
\node[scale=0.7,right] at (3,0) { $p$ };
\node[scale=0.7,right] at (9,0) { $p^{\smash{2}}$ };
\node[scale=0.7,right] at (27,0) { $p^{\smash{3}}$ };
\node[scale=0.7,right] at (81,0) { $p^{\smash{4}}$ };
\node[scale=0.8,rotate=-26] at (59,-32) { slope: $-(p{-}1)$ };
\end{tikzpicture}
\hfill\null

\caption{Convergence conditions for elements in $\Acrys$}
\label{fig:convAcrys}
\end{figure}

\paragraph{Analytic functions on annuli.}

The function $\nu$ that appeared in the formula~\eqref{eq:convAcrys} 
has a very erratic behavior. This is unfortunate for two reasons: the 
ring $\Acrys$ we defined do not have pleasant algebraic properties (for 
instance, it is not noetherian), nor a nice analytic interpretation
(its elements are not anayltic functions defined on a nice domain).
In order to get around these difficulties, we introduce a variant
of $\Acrys$ which does not have these defaults.
More precisely, given a positive real number $\mu$, we introduce
the ring $\Amu$ consisting of series of the form:
$$\begin{array}{rl}
& \displaystyle
x = \sum_{i \in \Z} [\xi_i] \: p^i \qquad (\xi_i \in \OCpflat) 
\smallskip \\
\text{with} & 
v_\flat(\xi_i) + \mu i \geq 0
\quad \text{and} \quad
\displaystyle
\lim_{i \to -\infty} v_\flat(\xi_i) + \mu i = +\infty.
\end{array}$$
When $\mu$ is rational\footnote{Otherwise, an element $\xi$ with the 
required properties does not exist.}, $\Amu$ is the $p$-adic completion 
of $\Ainf\big[\frac{[\xi]}p\big]$ for any $\xi \in \Ainf$ with 
$v_\flat(\xi) = \mu$.
We let $\Bpmu = \Amu[\frac 1 p]$. The elements of $\Bpmu$ are 
series of the form:
$$\begin{array}{rl}
& \displaystyle
x = \sum_{i \in \Z} [\xi_i] \: p^i \qquad (\xi_i \in \OCpflat) 
\smallskip \\
\text{with} & 
\displaystyle
\lim_{i \to -\infty} v_\flat(\xi_i) + \mu i = +\infty
\end{array}$$
\emph{i.e.} the same conditions as for $\Amu$ except that the condition 
of positivity has been dropped.
From the analytic point of view, elements of $\Bpmu$ should be 
considered as bounded analytic functions (of the variable $p$) on the 
annulus $\{ 0 \leq v_\flat(\cdot) < \mu\}$.

It is clear that $\Amu \subset A_{\mu'}$ (resp. $\Bpmu \subset 
B^+_{\mu'}$) as soon as $\mu \geq \mu'$. However, the reader should be 
careful that the functions $\mu \mapsto \Amu$ and $\mu \mapsto \Bpmu$ 
are not continuous in the sense that $\Amu$ is strictly included in 
$\bigcap_{\mu' < \mu} A_{\mu'}$, and similarly for the $\Bpmu$'s. 
In the analytic language, a function in $\bigcap_{\mu' < \mu} 
B^+_{\mu'}$ is analytic on the annulus $\{ 0 \leq v_\flat(\cdot) < 
\mu\}$ but not necessarily bounded. Similarly $\bigcap_{\mu > 0} 
\Bpmu$ is strictly greater than the ring $\Bpinf$ we have introduced 
in \S\ref{ssec:Ainf}; actually, we shall construct soon a quite important
element $t$ lying in the former ring but not in the latter.
The relation between $\Bpcrys$ and the $\Bpmu$'s is also simple
to understand. Indeed, the estimation~\eqref{eq:estimnui} shows 
that $\Bpmu \subset \Bpcrys \subset B^+_{p-1}$ for all $\mu >
p{-}1$  (\emph{cf} also Figure~\ref{fig:convAcrys}). 
At the integral level, we have $A_p \subset \Acrys \subset A_{p-1}$.

For $\mu > 0$, we also define $\AmuK = \O_K \otimes_{W(k)}
\Amu$ and $\BpmuK = K \otimes_{K_0} \Bpmu$. Elements in
$\BpmuK$ are series of the form $\sum_{i \in \Z} [\xi_i] \: 
\pi^i$ with $\lim_{i \to -\infty} v_\flat(\xi_i) + \frac {\mu i} e 
= +\infty$. The subring $\AmuK$ is characterized by the
positivity condition $v_\flat(\xi_i) + \mu
\big\lfloor\frac i e\big\rfloor \geq 0$ for all $i \in \Z$.

The notion of Newton polygons, which was defined for elements of 
$\Bpinf$ (resp. $\BpinfK$) in \S\ref{ssec:Ainf}, admits a
straightforward extension to $\Bpmu$ (resp. $\BpmuK$). 
Precisely, if $x = \sum_{i \in \Z} [\xi_i] \: p^i \in \Bpmu$ (resp. $x 
= \sum_{i \in \Z} [\xi_i] \: \pi^i \in \BpmuK$), we define $\NP_\mu(x)$ 
as the convex hull of the points $(i, v_\flat(i))$ (resp. $(\frac i e, 
v_\flat(i))$) together with the two points at infinity $(0, +\infty)$ 
and $+\infty{\cdot}(-1,\mu)$. When $x \in \bigcap_{\mu > 0} \Bpmu$
(resp. $x \in \bigcap_{\mu > 0} \BpmuK$), we define 
$\NP_\inf(x) = \bigcap_{\mu > 0} \NP_\mu(x)$. One checks easily 
that this definition agrees with the definition of $\NP_\inf$ on
$\Bpinf$ (resp. on $\BpinfK$) we gave earlier.

Finally, we observe that the Galois action and the Frobenius are 
well-defined on the $\Amu$'s and $\Bpmu$'s. Even better, for all 
$\mu > 0$, the Frobenius induces isomorphisms of rings $A_{\mu} \to 
A_{p\mu}$, $\Bpmu \to B^+_{p\mu}$ and $\BpmuK \to B^+_{p\mu,K}$.
As for Newton polygons, they are preserved under the action of $G_K$
and we have the following transformation formula under Frobenius:
$$\NP_{\mu p}(\varphi(x)) = \varphi_{\R^2}(\NP_\mu(x))
\quad \text{where} \quad
\varphi_{\R^2} : \R^2 \to \R^2, \, (i,v) \mapsto (i,pv)$$
for $x \in \BpmuK$. Passing to the limit on $\mu$, we find that 
$\NP_\inf(\varphi(x)) = \varphi_{\R^2}(\NP_\inf(x))$ for all $x \in
\bigcap_{\mu > 0} \BpmuK$.

\subsubsection{The element $t$}

An essential property of $\Acrys$ is that it contains a period for the 
cyclotomic character, that is a special element on which Galois acts by 
multiplication by $\chi_\cycl$. This distinguished element is:
$$t = \log\:[\Ueps] = \sum_{i=1}^\infty (-1)^{i-1} {\cdot}
\frac{([\Ueps]{-}1)^i} i.$$
Observe that the latter sum converges in $\Acrys$ since its
$i$-th summand is equal to:
$$(-1)^{i-1} \cdot (i{-}1)! \cdot \big([\Ueps^{1/p}] - 1\big)^i \cdot
\frac{\omega^i}{i!}$$
and therefore goes to $0$ in $\Acrys$, thanks to the factor $(i{-}1)!$
which converges to $0$ for the $p$-adic topology. A similar computation
shows that $t$ actually lies in $\Bpmu$ for all $\mu > 0$ and in
$\Amu$ for $\mu \geq 1 - \frac 1 p$.

Recall that the Frobenius and the group $G_K$ act on $[\Ueps]$ by 
$\varphi([\Ueps]) = [\Ueps]^p$ and $g [\Ueps] = [\Ueps]^{\chi_\cycl(g)}$ 
for $g \in G_K$. Taking logarithms, we find $\varphi(t) = pt$ and
$gt = \chi_\cycl(g)\: t$ for all $g \in G_K$. The latter relation is
what we expected: the element $t$ is a period for the cyclotomic
character.

\smallskip

The Newton polygon of $t$ can also be computed\footnote{The computation 
can be carried out as follows. By Remark~\ref{rem:NPeps}, we know that 
$\NP_\inf([\Ueps]-1))$ is the set $\mathcal P^+$ defined as the convex 
hull of the points $A_n = \big(n,\frac 1{(p-1)p^n}\big)$ for $n$ varying 
in $\N$. By the multiplicativity property of Newton polygons, we find 
that $\NP\big( \frac{([\Ueps]-1)^i} i\big) = \tau_{v_p(i)} (i \mathcal 
P^+)$ where $\tau_u$ is the translation of vector $(0,-u)$. We now 
observe that each $A_n$ ($n \in \Z$) belongs to exactly one 
$\tau_{v_p(i)} (i \mathcal P^+)$: when $n \geq 0$, we have $i = 0$ and 
when $n < 0$, we have $i = p^{-n}$. Therefore the Newton polygon of $t$ 
is the convex hull of the $A_n$'s for $n$ varying in $\Z$.}.
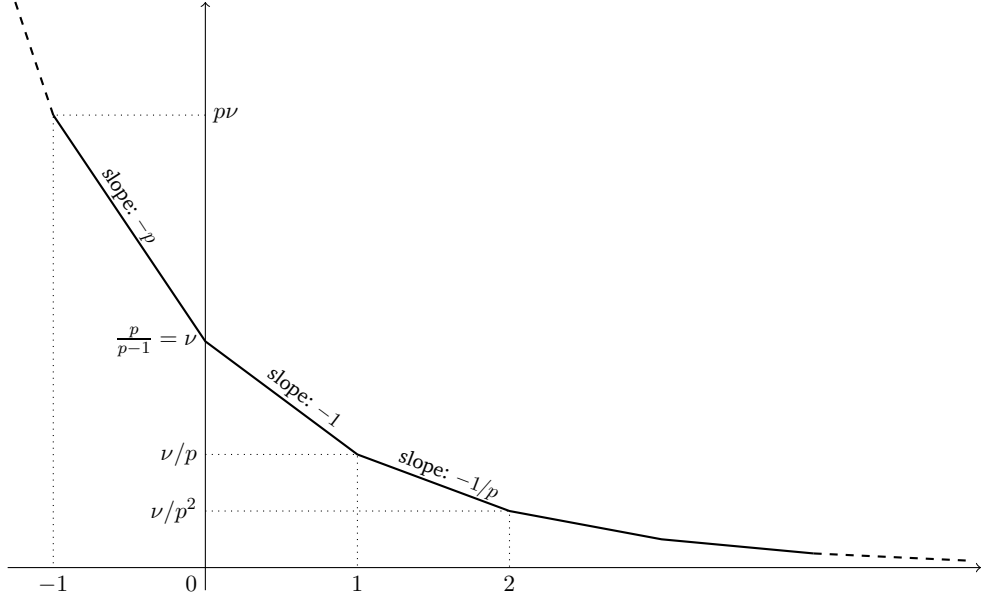
\begin{figure}
\hfill
\begin{tikzpicture}[rotate=90,xscale=3,yscale=-2]
\draw[->] (-0.1,0)--(2.5,0);
\draw[->] (0,-1.3)--(0,5.1);
\draw[thick,dashed] (2.5,-1.25)--(2,-1);
\draw[thick] (2,-1)--(1,0)--(0.5,1)--(0.25,2)--(0.125,3)--(0.0625,4);
\draw[thick,dashed] (0.0625,4)--(0.03125,5);
\begin{scope}[dotted]
\draw (2,0)--(2,-1)--(0,-1);
\draw (0.5,0)--(0.5,1)--(0,1);
\draw (0.25,0)--(0.25,2)--(0,2);
\end{scope}
\node[below left,scale=0.8] at (0,0) { $0$ };
\node[right,scale=0.8] at (2,0) { $p\nu$ };
\node[left,scale=0.8] at (1,0) { $\frac p{p-1} = \nu$ };
\node[left,scale=0.8] at (0.5,0) { $\nu/p$ };
\node[left,scale=0.8] at (0.25,0) { $\nu/p^2$ };
\node[below,scale=0.8] at (0,-1) { $-1$ };
\node[below,scale=0.8] at (0,1) { $1$ };
\node[below,scale=0.8] at (0,2) { $2$ };
\node[scale=0.7,rotate=-57] at (1.6,-0.5) { slope: $-p$ };
\node[scale=0.7,rotate=-37] at (0.75,0.65) { slope: $-1$ };
\node[scale=0.7,rotate=-21] at (0.41,1.6) { slope: $-1\hspace{-0.1ex}/\hspace{-0.1ex}p$ };
\end{tikzpicture}
\hfill\null

\caption{The Newton polygon of $t$}
\label{fig:Newtont}
\end{figure}
The result we find is displayed on Figure~\ref{fig:Newtont}; we notice
in particular that its slopes are unbounded, reflecting the fact that
$t$ is in $\Bpmu$ for all $\mu > 0$. It also remains unchanged under 
the transformation $(i,v) \mapsto (i{-}1, pv)$, reflecting the fact
that $\varphi(t) = pt$. 
In fact, the special shape of $\NP_\inf(t)$ is explained by the
existence of a decomposition of $t$ as an infinite convergent product 
(in all $\Bmu$'s), precisely:
\begin{equation}
\label{eq:prodt}
t = \prod_{n=0}^\infty \varphi^{-n}(\omega) \cdot
\prod_{n=1}^\infty \frac{\varphi^n(\omega)} p
\end{equation}
where $\omega = \frac{[\Ueps]-1}{[\Ueps]^{1/p}-1}$ is the element of
Proposition~\ref{prop:omega}. The factorization~\eqref{eq:prodt} should 
be paralleled with~\eqref{eq:prodeps}.

\begin{deftn}
For $\mu > 0$ or $\mu = \crys$, we set $\Bmu = \Bpmu[\frac 1 t]$.
\end{deftn}

\noindent
The Frobenius and the Galois action extend to $\Bmu$ without 
difficulty: for $g \in G_K$, $x \in \Bpmu$ and $m \in \N$, we put 
$\varphi(\frac x{t^m}) = \frac{\varphi(x)}{p^m t^m}$ and $g(\frac 
x{t^m}) = \frac{g x}{\chi_\cycl(g)^m \cdot t^m}$.

\subsection{The de Rham filtration and the field $\BdR$}

We recall that, in \S\ref{sssec:sharp}, we have constructed a ring 
homomorphism $\theta : B^+_{\inf} \to \C_p$, which was given by 
the explicit formula:
$$\theta : \quad \sum_{i=-\infty}^\infty [\xi_i]\: p^i \,\, 
\mapsto \,\, \sum_{i=-\infty}^\infty \xi_i^\sharp\: p^i \qquad 
(\xi_i \in \OCpflat),\, \xi_i = 0 \text{ for } i \ll 0).$$
The filtration by the power of the ideal $\ker \theta$ will play
an important role because it will eventually correspond to the de
Rham filtration on the cohomology. 
We devote this subsection to the study of its main properties.
This will lead us to the definition of the period ring $\BdR$.

\subsubsection{Definition and main properties of the de Rham filtration}

First of all, we will need to extend the morphism $\theta$ to the rings 
$\Bmu$'s we have introduced earlier.
Actually, just noticing that $v_p(\xi_i^\sharp p^i) = v_\flat(\xi_i) + 
i$, we deduce that $\theta$ extends readily to a ring homomorphism 
$\theta_\mu : \Bpmu \to \C_p$ whenever $\mu \geq 1$. Extending scalars 
to $K$, we obtain a ring homomorphism $\theta_{\mu,K} : \BpmuK \to 
\C_p$ for all $\mu \geq 1$. We observe that $\theta_\mu$ (resp. 
$\theta_{\mu,K}$) maps the subring $\Amu$ (resp. $\AmuK$) to 
$\O_{\Cp}$.

The condition $\mu\geq 1$ for the existence of $\theta_\mu$ suggests 
that the rings $A_1$, $A_{1,K}$, $B^+_1$ and $B^+_{1,K}$ will play a 
particular role. In the literature, they are often denoted by $\Amax$, 
$\AmaxK$, $\Bpmax$ and $\BpmaxK$ respectively; we will also 
use this notation in the sequel and will set $\theta_\max = \theta_1$,
$\theta_{\max,K} = \theta_{1,K}$ accordingly. Similarly, we will use
the notation $\theta_\inf$ and $\theta_{\inf,K}$ for $\theta$ and
$\theta_K$ respectively.

We recall that $\Amax$ is the $p$-adic completion of 
$\Ainf\big[\frac{[\Up]}p\big]$. 
In particular, we have canonical isomorphisms:
\begin{equation}
\label{eq:Amaxmodp}
\textstyle
\Amax / p \Amax
\, \simeq \,
\big(\OCpflat / \Up \OCpflat\big)[X] 
\, \simeq \,
\big(\O_{\Cp} / p \O_{\Cp}\big)[X], \quad
[\Up]/p \,\mapsfrom\, X
\end{equation}
the second isomorphism coming from Lemma~\ref{lem:ocpflat}.
Similarly, we have:
\begin{equation}
\label{eq:AmaxKmodpi}
\textstyle
\AmaxK / \pi \AmaxK
\, \simeq \,
\big(\OCpflat / \Upi \OCpflat\big)[X] 
\, \simeq \,
\big(\O_{\Cp} / \pi \O_{\Cp}\big)[X], \quad
[\Upi]/\pi \,\mapsfrom\, X.
\end{equation}
We can also identify the kernels of $\theta_\max$ and 
$\theta_{\max,K}$ (as we did for $\theta$ and $\theta_K$ in 
Proposition~\ref{prop:kertheta}).

\begin{prop}
\label{prop:kerthetamax}
The ideal $\ker\theta_\max$ (resp. $\ker\theta_{\max,K}$) is the principal
ideal generated by the element $[\Up] - p$ (resp.
$[\Upi] - \pi$).
\end{prop}

\begin{proof}
We only give the proof for $\theta_\max$, the case of $\theta_{\max,K}$
being absolutely similar.
We will prove that $1 - [\Up]/p$ is a generator of the ideal
$\Amax \cap \ker \theta_\max$ of $\Amax$.
Let $\bar\theta_\max : \Amax/p\Amax \to \O_{\Cp}/p\O_{\Cp}$ be the 
morphism induced by $\theta_\max$.
Repeating the argument of Proposition~\ref{prop:kertheta}, it is enough 
to show that $\ker \bar\theta_\max$ is the principal ideal generated by 
$1 - [\Up]/p $.
Under the isomorphism~\eqref{eq:Amaxmodp}, $\bar\theta_\max$ acts by the 
identity on $\O_{\Cp}/p\O_{\Cp}$ and takes $X$ to $1$. Hence, its kernel 
is the principal ideal generated by $X{-}1$.
\end{proof}

\begin{deftn}
For $\mu \geq 1$, $\mu = \crys$ or $\mu = \inf$ and for
$m \in \N$, we define:
$$\Fil^m \Bpmu = (\ker \theta_\mu)^m
\quad \text{and} \quad
\Fil^m \BpmuK = (\ker \theta_{\mu,K})^m.$$
\end{deftn}

\noindent
We will use the notation $\gr$ to refer to the graded ring of a 
filtered ring: if $m$ is an integer and $\mathfrak A$ is a filtered ring, 
we put $\gr^m \mathfrak A = \Fil^m \mathfrak A / \Fil^{m+1} \mathfrak A$ 
and $\gr\,\mathfrak A = \bigoplus_{m \geq 0} \gr^m \mathfrak A$. 
We recall that $\gr^0 \mathfrak A$ is a ring and that $\gr^m \mathfrak
A$ is a module over $\gr^0 \mathfrak A$ for all $m \geq 0$. As for
$\gr\:\mathfrak A$, is it a graded algebra over $\gr^0\mathfrak A$.
In our case, we have $\gr^0 \Bpmu = \gr^0 \BpmuK = \Cp$ (since
$\theta_\mu$ and $\theta_{\mu,K}$ are surjective). Hence
$\gr^m \Bpmu$ and $\gr^m \BpmuK$ have a natural structure of
$\Cp$-vector space. They moreover inherit a Galois action, so that 
they are actually $\Cp$-semi-linear representations of $G_K$, 
\emph{i.e.} objects of the category $\Rep_{\Cp}(G_K)$.
From Proposition~\ref{prop:kertheta} and 
Proposition~\ref{prop:kerthetamax}, we deduce that $\gr^m \Bpinf$, 
$\gr^m \BpinfK$, $\gr^m \Bpmax$ and $\gr^m \BpmaxK$ are all 
one dimensional over $\Cp$. As we shall see below (\emph{cf} 
Proposition~\ref{prop:compkertheta}), this property also holds for 
$\gr^m \Bpmu$ and $\gr^m \BpmuK$ for any $\mu$.

\smallskip

The next proposition shows that the de Rham filtration is separated.

\begin{prop}
\label{prop:kerthetasep}
For $\mu \geq 1$, $\mu = \crys$ or $\mu=\inf$, we have
$\bigcap_m \Fil^m \Bpmu = \bigcap_m \Fil^m \BpmuK = 0$.
\end{prop}

\begin{proof}
Since $\Bpmu$ and $\BpmuK$ contain $\BpmaxK$, it is enough 
to prove the proposition for $\theta_{\max,K}$.
After Proposition~\ref{prop:kerthetamax}, we are reduced to check that if 
$x \in \BpmaxK$ is divisible by $\big(1 - \frac{[\Upi]}{\pi}\big)^m$ 
for all $m$, then $x = 0$. Multiplying $x$ by the adequate power of 
$\pi$, we may assume that $x \in \AmaxK$ and in addition, if $x 
\neq 0$, that $x \not\in \pi \Amax$. Using 
isomorphism~\eqref{eq:Amaxmodp}, we find that $x$ vanishes in 
$\AmaxK / \pi \AmaxK$, \emph{i.e.} $x \in \pi \AmaxK$. By 
our assumption, this implies that $x = 0$.
\end{proof}

\begin{prop}
\label{prop:compkertheta}
For $\mu \geq 1$ or $\mu = \crys$ and for $m \in \N$,
the inclusion $\Bpinf \to \Bpmu$ (resp.
$\BpinfK \to \BpmuK$)
induces a $G_K$-equivariant isomorphism 
$$\Bpinf / \Fil^m \Bpinf \, \simeq \,
\Bpmu / \Fil^m \Bpmu
\qquad \text{(resp. }
\BpinfK / \Fil^m \BpinfK \, \simeq \,
\BpmuK / \Fil^m \BpmuK
\text{).}$$
\end{prop}

\begin{proof}
In the case $\mu = 1$, the proposition follows by combining 
Propositions~\ref{prop:kertheta} and~\ref{prop:kerthetamax}. Before 
moving to the case of a general $\mu$, we will prove an additional
continuity property of the isomorphism $\Bpinf / \Fil^m \Bpinf 
\simeq \Bpmax / \Fil^m \Bpmax$, which will be useful later. 
Precisely, we claim that, for all $m \geq 0$, there exists a
nonnegative integer $v_m$ such that: 
\begin{equation}
\label{eq:continuity}
p^{v_m} \cdot \Amax / \Fil^m \Amax \subset \Ainf / \Fil^m \Ainf
\end{equation}
We prove the claim by induction on $m$. For $m = 0$, there is nothing 
to prove (we can take $v_0 = 0$). 
We now assume that \eqref{eq:continuity} is proved for $m$. We 
consider the following commutative diagram with exact rows:
$$\xymatrix @C=2em @R=1.5em {
0 \ar[r] &
\gr^m \Ainf \ar[r] \ar[d]^\sim &
\Ainf / \Fil^{m+1} \Ainf \ar[r] \ar[d]^\sim &
\Ainf / \Fil^m \Ainf \ar[r] \ar[d]^\sim & 0 \\
0 \ar[r] &
\gr^m \Amax \ar[r] &
\Amax / \Fil^{m+1} \Amax \ar[r] &
\Amax / \Fil^m \Amax \ar[r] & 0. }$$
From the fact that $\gr^m \Bpinf \to \gr^m \Bpmax$ is a $\Cp$-linear 
mapping between two one-dimensional $\Cp$-vector spaces, we deduce that 
there exists an integer $v$ such that $p^v{\cdot}\gr^m \Amax \subset 
\gr^m \Ainf$. A diagram chase then shows that \eqref{eq:continuity} 
holds with $v_{m+1} = v_m + v$.

We now go back to the proof of the proposition. We pick $\mu \in
(1,+\infty) \sqcup \{\crys\}$ and $m \in \N$.
Let $f : \Bpinf / (\ker \theta)^m \to \Bpmu / (\ker \theta_\mu)^m$
be the morphism of the proposition.
Let $A'_\mu$ be the sub-$\Ainf$-algebra of $\Bpinf$ generated by all 
the elements of the form $\frac{[\xi]}{p^i}$ ($\xi \in \OCpflat$, 
$i \in \N$), which belong to $\Amu$. Then $A'_\mu \subset \Bpinf$
and $\Amu$ appears as the $p$-adic completion of $A'_\mu$. 
The former property implies that we have a morphism
$g' : A'_\mu / (\ker \theta_\mu)^m \to \Bpinf / (\ker \theta)^m$.
We claim that $g'$ is continuous. Indeed, by~\eqref{eq:continuity},
we have $\Amax \subset p^{-v_m} \Ainf + (\ker \theta_\max)^m$. Since 
$A'_\mu
\subset \Amax$, we deduce that $g'$ maps $A'_\mu / (\ker \theta_\mu)^m$
to $p^{-v_m} \Ainf / (\ker \theta)^m$, which implies its continuity.
Now, passing to the $p$-adic completion and inverting~$p$, we find that 
$g'$ induces a ring morphism $g : \Bpmu / (\ker \theta_\mu)^m \to 
\Bpinf / (\ker \theta)^m$, which is an inverse of $f$. Therefore $f$ 
is an isomorphism.

The identification $\BpinfK / \Fil^m \BpinfK
\simeq \BpmuK / \Fil^m \BpmuK$ is obtained similarly.
\end{proof}

\noindent
Proposition~\ref{prop:compkertheta} implies that for all $m$, all
the maps of the commutative square below are isomorphisms of 
$\Cp$-semi-linear representations:
\begin{equation}
\label{eq:diaggrm}
\raisebox{0.5\depth}{
\xymatrix @C=3em {
\gr^m \Bpinf \ar[r]^-{\sim} \ar[d]_-{\sim} &
\gr^m \Bpmu \ar[d]^-{\sim} \\
\gr^m \BpinfK \ar[r]^{\sim} &
\gr^m \BpmuK}}
\end{equation}
We can moreover entirely elucidate the Galois action.
Indeed we have the following proposition.

\begin{prop}
\label{prop:tmgen}
For $\mu \geq 1$ or $\mu = \crys$ and for $m \in \N$, the
spaces $\gr^m \Bpmu$ and $\gr^m \BpmuK$ are generated
by the class of $t^m$.
\end{prop}

\begin{proof}
Thanks to the diagram~\eqref{eq:diaggrm}, it is enough to prove the 
proposition for $\gr^m \Bpmu$.
We already know that $\gr^m \Bpmu$ is one dimensional over $\Cp$. 
We observe that $t$ lies in $\ker \theta_\mu$ since $\theta_\mu(t) = 
\log \theta_\mu([\Ueps]) = \log 1 = 0$. Thus $t^m \in \Fil^m \Bpmu$ 
and we are reduced to prove that $t^m$ is not zero in $\gr^{m+1} 
\Bpmu$.
Noting that $t \equiv [\Ueps]-1 \pmod{\Fil^2 \Bpmu}$, we can replace 
$t$ by $[\Ueps]-1$. Using again the diagram~\eqref{eq:diaggrm}, it is
enough to show that $([\Ueps]{-}1])^m$ does not vanish in $\gr^{m+1} 
\Bpinf$.
Let $\omega$ be the element of Proposition~\ref{prop:omega}, so that we 
can write $([\Ueps] - 1)^m = \omega^m \cdot ([\Ueps]^{1/p}-1)^m$.
We know that the class of $\omega^m$ is a generator of $\gr^m \Bpinf$. 
It is enough to check $\theta$ does not vanish on $([\Ueps]^{1/p}-1)^m$. 
But a direct computation gives $\theta(([\Ueps]^{1/p}-1)^m) = 
(\varepsilon_1-1)^m$ where $\varepsilon_1 \in \Cp$ is a 
primitive $p$-th root of unity. We conclude by noticing that
$\varepsilon_1 \neq 1$.
\end{proof}

\begin{rem}
\label{rem:genkertheta}
We strongly insist on the fact that $t^m$ is \emph{not} a generator of 
$\Fil^m \Bpmu$ (resp. $\Fil^m \BpmuK$) since this is often the
source of confusion.
Let us clarify this point by examining a bit the case where $\mu = 1$. 
Then, by Proposition~\ref{prop:kerthetamax}, we know that $\Fil^1 
\Bpmax$ is generated by the element $\gamma = [\Up] - p$. 
Thus, we can write $t = \gamma \gamma'$ for some $\gamma' \in \Bpmax$. 
It turns out that $\gamma'$ is not invertible in $\Bpmax$ but is a
unit in 
$\Bpmax / \Fil^1 \Bpmax$ (which is isomorphic to $\Cp$), reflecting
the fact that $t^m$ does not generate $\Fil^m \Bpmax$ but 
generates $\gr^m \Bpmax$.
The situation is quite similar to the following one which is very 
familiar to the number theorists: pick an odd prime number $p$, equip 
$\Z$ with the filtration $\Fil^m \Z = p^m \Z$ and consider the element 
$t = 2p$. Then $t^m$ is not a generator of $\Fil^m \Z$ but it does
generate $\gr^m \Z$ because $2$ is invertible modulo $p$.
\end{rem}

It follows from Proposition~\ref{prop:tmgen} that $\gr^m \Bpmu$ 
and $\gr^m \BpmuK$ are both isomorphic to $\Cp(\chi_\cycl^m)$
in the category $\Rep_{\Cp}(G_K)$. Passing to the graduation, we 
obtain $G_K$-equivariant isomorphisms of rings:
\begin{equation}
\label{eq:gradedplus}
\gr\, \Bpmu \simeq \gr\, \BpmuK
\simeq \Cp[t] \qquad (\mu \in [1,+\infty) \sqcup \{\inf,\crys, \max\})
\end{equation}
where the letter $t$ on the right hand side is a new variable
(corresponding to the special element $t \in \Bpmu$) on which
Galois acts by multiplication by the cyclotomic character.

\subsubsection{Completion with respect to the de Rham filtration}

After what we have done previously, it is natural to introduce the 
completion of the $\Bpmu$'s (resp. the $\BpmuK$'s) with respect 
to the de Rham filtration. This actually leads to the definition of the 
period ring $\BpdR$.

\begin{deftn}
We define $\BpdR$ as the completion of $\Bpinf$ for
the $(\ker\theta)$-adic topology:
$$\BpdR = \varprojlim_m \, \Bpinf / (\ker\theta)^m 
= \varprojlim_m \, \Bpinf / \Fil^m \Bpinf.$$
\end{deftn}

Since each quotient $\Bpinf / \Fil^m \Bpinf$ has a Galois
action, $\BpdR$ inherits an action of $G_K$.
Besides, the algebraic structure of $\BpdR$ is very pleasant. Indeed, 
from the fact that $\ker\theta$ is a principal ideal of $\Bpinf$, we 
deduce that $\BpdR$ is a discrete valuation ring. Its maximal ideal is 
the ideal generated by $\ker\theta$ and its residue field is canonically 
isomorphic to $\Bpinf/\Fil^1 \Bpinf \simeq \Cp$. Therefore, as a 
ring, $\BpdR$ is isomorphic $\Cp(\!(t)\!)$, that is to the ring
$B'_\HT$ we introduced in \S\ref{sssec:HT}.
However, we strongly insist on the fact that there is no such 
isomorphism preserving the Galois action.
The sole connection between $\BpdR$ and $\Cp(\!(t)\!)$ is that they 
share the same graded ring, namely $B_\HT$.

Observe that, by Proposition~\ref{prop:compkertheta}, we could have 
defined alternatively $\BpdR$ as the completion of $\Bpmu$ or 
$\BpmuK$, \emph{i.e.} we have the following canonical 
identifications:
$$\BpdR = \varprojlim_m \, \Bpmu / \Fil^m \Bpmu
= \varprojlim_m \, \BpmuK / \Fil^m \BpmuK.$$
for any $\mu \in [1,+\infty) \sqcup \{\inf,\crys, \max\}$.
Combining this with the fact that the de Rham filtration is separated 
(\emph{cf} Proposition~\ref{prop:kerthetasep}), we deduce that the 
canonical maps $\Bpmu \to \BpmuK \to \BpdR$ are injective for 
all $\mu$ as before.
In particular $t \in \BpdR$ and $\BpdR$ contains a copy of $K$. 
Since the definition of $\BpdR$ does not actually depend on $K$,
it follows that $\BpdR$ contains (in a coherent way) a copy of any 
finite extension of $\Qp$, that is a copy of $\bar K$.
Denote by $\iota: \bar K \to \BpdR$ the resulting embedding. It turns 
out that $\iota$ can be understood in more down-to-earth terms. Indeed 
observe first that $\bar K$ naturally embeds into the residue field of 
$\BpdR$ since the latter is canonically isomorphic to $\Cp$. 
By Hensel lemma, this embedding admits a unique lifting $\iota: \bar K 
\to \BpdR$ which is a homomorphism of $K_0$-algebras: concretely, for 
$x \in \bar K$ whose minimal polynomial over $K_0$ is denoted by $P$, 
$\iota(x)$ is the unique root of $P$ that lifts the image of $x$ in 
$\Cp$. In particular, the composite $\bar K \to \BpdR \to \Cp$
is the natural inclusion.

\smallskip

The map $\theta$ extends to $\BpdR$ easily: we define $\theta_\dR$ as 
the composite $\BpdR \to \Bpinf / (\ker \theta) \to \Cp$ where the 
first map is the projection onto the first component and the second map is 
induced by~$\theta$. 
We set $\Fil^m \BpdR = (\ker\theta_\dR)^m$ for $m \in \N$. 
Observe that the kernel of $\theta_\dR$ is nothing but the maximal ideal 
of $\BpdR$. As a consequence, the de Rham filtration of $\BpdR$
coincides with the canonical filtration on the discrete valuation
ring $\BpdR$, given by the valuation. Its graded ring is isomorphic to
$\Cp[t]$ (compare with \eqref{eq:gradedplus}).
Moreover, any generator of $\Bpmu$ or $\BpmuK$ (for $\mu \in 
\{\inf, \max\}$) is a generator of $\BpdR$, \emph{i.e.} a uniformizer 
of $\BpdR$. Even better, by completeness, an element of $\BpdR$
is a uniformizer if and only if it does not belong to $\Fil^1 \BpdR$
or, equivalently, it does not vanish if $\gr^1 \BpdR$.
In particular, the special element~$t$ is a uniformizer of $\BpdR$
by Proposition~\ref{prop:tmgen}.

\begin{rem}
\label{rem:genkertheta2}
Continuing Remark~\ref{rem:genkertheta} (and importing notations from 
there), we observe that the element $\gamma' \in \Bpmax \subset 
\BpdR$ is invertible in $\BpdR$
since it is nonzero in the residue field; thus $t = \gamma\gamma'$ is 
a generator of $\ker\theta_\dR$ as $\gamma$ is.
This contrasts with the fact that $t$ did not generate $\ker\theta_\max$ 
because $\gamma'$ was not invertible in $\Bpmax$.
\end{rem}

\paragraph{Topology on $\BpdR$.}

As $\BpdR$ is defined as a completion, the first natural topology on 
$\BpdR$ is the $(\ker \theta)$-adic topology: a sequence $(x_n)_{n 
\geq 0}$ of elements of $\BpdR$ converges to $x \in \BpdR$ if and 
only if, for all $m$, the sequence $x_n \mod \Fil^m \BpdR$ is 
eventually constant. This topology is actually not nice because it
does not see the $p$-adic topology: it induces the discrete topology
both on the subfield $\bar K{\cdot}\hat K^\ur$ and on the residue
field $\Cp$.

A coarser topology can be defined as follows.
Observe that the quotients $\Bpinf / \Fil^m \Bpinf$ have finite 
length and hence are equipped with a canonical topology. 
This topology can be described by remarking that the lattice $\Ainf / 
\Fil^m \Ainf$ defines a valuation $v_{m,\inf}$ on $\BpdR / \Fil^m 
\BpdR$: given $x \in \BpdR / \Fil^m \BpdR \simeq \Bpinf / \Fil^m 
\Bpinf$, we define $v_{m,\inf}(x)$ as the largest (possible negative) 
integer $n$ for which $x \in p^n \Ainf / \Fil^m \Ainf$. The valuation 
$v_{m,\inf}$ defines a norm on $\Bpinf / \Fil^m \Bpinf$, and hence a 
topology.

\begin{rem}
Alternatively, instead of $\Bpinf$, one could have worked with 
$\Bpmu$ for a different $\mu$. We would have ended up this way with a
valuation $v_{m,\mu}$ on $\BpdR / \Fil^m \BpdR$ 
for which there exists a constant $v_{m,\mu}$ with the property that:
\begin{equation}
\label{eq:vm}
v_{m,\inf}(x) - v_{m,\mu} \leq v_{m,\mu}(x) \leq v_{m,\inf}(x)
\end{equation}
for all $x \in \BpdR / \Fil^m \BpdR$ (see the first part of the proof 
of Proposition~\ref{prop:compkertheta}). Therefore the topology induced 
by $v_{m,\mu}$ agrees with that defined by $v_{m,\inf}$ for all $m$.
\end{rem}

We extend $v_{m,\inf}$ to $\BpdR$ by precomposing by the natural 
projection $\BpdR \to \BpdR / \Fil^m \BpdR$.
When $m$ varies, the $v_{m,\inf}$'s define a family of semi-norms
on $\BpdR$, giving it the structure of a Frechet space.
The attached topology will be called (in this article) the 
\emph{standard topology} on $\BpdR$. Concretely, a sequence
$(x_n)$ of elements of $\BpdR$ converges to $x \in \BpdR$ 
for the standard topology if and only if, for all integer $m$, 
the image of $x_n$ is $\BpdR / \Fil^m \BpdR \simeq \Bpinf / 
\Fil^m \Bpinf$ converges to the image of $x$.
Clearly, the standard topology induces the usual $p$-adic topology on 
the residue field $\BpdR / \Fil^1 \BpdR \simeq \Cp$. Colmez proved 
in~\cite{colmez} that $\bar K$ is dense in $\BpdR$ for the standard 
topology.

We point out that there is no good notion of $p$-adic topology 
on $\BpdR$. Indeed, if there were, the inclusion $\iota : \bar K \to 
\BpdR$ would extend to an inclusion $\Cp \to \BpdR$ which would 
imply that $\BpdR$ would be isomorphic to $\Cp(\!(t)\!) = B'_\HT$ and 
we have already seen that this does not happen. Yet, $\BpdR$ admits 
kinds of lattices, \emph{e.g.}
$$A_{\mu,\dR} = \varprojlim_m \, \Amu / (\Amu \cap \Fil^m \Bpmu)
\quad \text{or} \quad
A_{\mu,K,\dR} = \varprojlim_m \, \AmuK / (\AmuK \cap 
\Fil^m \BpmuK)$$
though we have to be careful that $A_{\mu,\dR}[\frac 1 p]
\subsetneq B_\dR$ and similarly for $A_{\mu,K,\dR}$.
These ``lattices'' do define topologies on $\BpdR$ (which might be 
considered as sort of $p$-adic topologies). However, these topologies 
are all different (and different from the standard topology) and they 
all have bad properties; for instance, the inclusion $\iota : \bar K \to 
\BpdR$ is not continuous for any of them. The point behind this is 
that the constant $v_{m,\mu}$ of Eq.~\eqref{eq:vm} is not bounded uniformly
when $m$ grows.

\begin{rem}
The situation is quite similar to that $\Qp[[t]]$.
The analogue of the standard topology on $\Qp[[t]]$ is the
standard Fréchet topology on this ring: a sequence 
$(f_n)_{n \geq 0}$ converges to $f$ if and only if $f_n
\mod t^m$ converges to $f \mod t^m$ in $\Qp[t]/t^m$ for 
all~$m \in \N$.
This is further equivalent to the fact that, for all fixed $m
\in \N$, the $m$-th coefficient of $f_n$ converges to the the 
$m$-th coefficient of $f$.
Another topology on $\Qp[[t]]$ is that defined by the ``lattice''
$\Zp[[t]]$, for which the sequence $(f_n)_{n \geq 0}$ converges to 
$f$  when, for each $A \geq 0$, there exists an index $n_0$ with
the property that $f_n \equiv f \pmod{p^A \Zp[[x]]}$ for all
$n \geq n_0$. This notion of convergence is stronger than the
previous one because we impose here that the coefficients of 
$f_n$ converge \emph{uniformly} to that of $f$ (the index $n_0$
has to the same for all~$m$).
\end{rem}

\paragraph{Inverting $t$.}

Recall that we have defined $\Bmu$ and $\BmuK$ as $\Bpmu[\frac 1 
t]$ and $\BpmuK[\frac 1 t]$ respectively for $\mu \geq 1$ or $\mu
= \crys$ (recall that this definition does not make sense for $\mu =
\inf$ because $t \not\in \Bpinf$).
Similarly we set $\BdR = 
\BpdR[\frac 1 t]$. Since $\BpdR$ is a discrete valuation ring with 
uniformizer $t$, $\BdR$ is also the fraction field of $\BpdR$; in 
particular, it is a field. Moreover since localization is exact,
the rings $\Bmu$ and $\BmuK$ appear as subrings of $\BdR$.

The de Rham filtration extends readily to~$\BdR$ by letting $\Fil^m 
\BdR = t^m \BpdR$ for $m \in \Z$. The graded ring of $\BdR$ is then
canonically isomorphic to $\Cp[t,t^{-1}] = B_\HT$.
If $B$ is any subring of $\BdR$, we define:
\begin{equation}
\label{eq:filmB}
\Fil^m B = B \cap \Fil^m \BdR \qquad (m \in \Z).
\end{equation}
Observe that $\Fil^0 B$ is the intersection of two rings and thus
is a ring as well.
It is easily checked that, when $B = \Bpmu$ or $\BpmuK$ (for $\mu 
\geq 1$ or $\mu = \crys$), the above definition leads to 
the de Rham filtration $\Fil^m B$ we have defined earlier by different 
means. Yet, the definition~\eqref{eq:filmB} is new and interesting for 
$B = \Bmu$ and $B = \BmuK$.
The filtrations obtained this way sit in the following diagram
(and a similar diagram for $\BmuK$):

\begin{center}
\begin{tikzpicture}[xscale=2.2,yscale=1.2]
\node at (-2.8,0) { \ph $\cdots$ };
\node at (-2.5,0) { \ph $\subset$ };
\node at (-2,0) { \ph $\Fil^2 \Bmu$ };
\node at (-1.5,0) { \ph $\subset$ };
\node at (-1,0) { \ph $\Fil^1 \Bmu$ };
\node at (-0.5,0) { \ph $\subset$ };
\node at (0,0) { \ph $\Fil^0 \Bmu$ };
\node at (0.5,0) { \ph $\subset$ };
\node at (1,0) { \ph $\Fil^{-1} \Bmu$ };
\node at (1.5,0) { \ph $\subset$ };
\node at (2,0) { \ph $\Fil^{-2} \Bmu$ };
\node at (2.5,0) { \ph $\subset$ };
\node at (2.8,0) { \ph $\cdots$ };
\node at (-2.8,-1) { \ph $\cdots$ };
\node at (-2.5,-1) { \ph $\subset$ };
\node at (-2,-1) { \ph $\Fil^2 \Bpmu$ };
\node at (-1.5,-1) { \ph $\subset$ };
\node at (-1,-1) { \ph $\Fil^1 \Bpmu$ };
\node at (-0.5,-1) { \ph $\subset$ };
\node at (0,-1) { \ph $\phantom{\Bpmu = {}}\Fil^0 \Bpmu = \Bpmu$ };
\node[rotate=90] at (-2,-0.45) { \ph $\subset$ };
\node[rotate=90] at (-1,-0.45) { \ph $\subset$ };
\node[rotate=90] at (0,-0.45) { \ph $\subset$ };
\end{tikzpicture}
\end{center}

\noindent
The reader should be very careful that the inclusion $\Bpmu \subset 
\Fil^0 \Bmu$ is strict.
Let us first focus on the case where $\mu = \max$.
We recall that, in Remark~\ref{rem:genkertheta}, we have set $\gamma = 
[p^\flat] - p \in \Bpmax$ and noticed that $t = \gamma \gamma'$ 
for some $\gamma' \in \Bpmax$. The element $\gamma'$ is not invertible 
in $\Bpmax$ but we have seen in Remark~\ref{rem:genkertheta2} that it 
is invertible in $\BpdR$. Besides, since $\gamma'$ is a divisor of 
$t$, it is invertible in $\Bmax$. Now consider $\frac 1{\gamma'} \in 
\Bmax$. It does not lie in $\Bpmax$. However, its image in $\BdR$ 
falls in $\BpdR$, so that $\frac 1{\gamma'} \in \Fil^0 \Bmax$.
Actually, one can (easily) prove that $\Fil^0 \Bmax = \Bpmax[\frac 1 
{\gamma'}]$.
A similar description is also possible for a general $\mu$. Precisely
let $S$ be the multiplicative part consisting of all divisors in
$\Bpmu$ of some power of $t$. Then $\Fil^0 \Bmu = \Bpmu[S^{-1}]$.

\begin{rem}
As discussed in Remark~\ref{rem:genkertheta}, what happens
here is very similar to the following very classical situation: assume 
that $\Z$ is endowed with the filtration $\Fil^m \Z = p^m \Z$, which 
induces the usual valuation filtration on $\Zp$ and $\Qp$ after 
completion. Now consider the localization $\Z[\frac 1{2p}]$; it is a 
subring of $\Qp$ and then inherits the valuation filtration. For this
filtration, we have $\Fil^0\:\Z[\frac 1{2p}] = \Z[\frac 1 2]$.
\end{rem}

\begin{rem}
The reader may wonder why we defined $\Bmu$ as $\Bpmu[\frac 1 t]$
and not $\Bpmu[\frac 1 \gamma]$ in order to avoid the small
unpleasantness discussed above. One reason is that the Frobenius
does not extend on $\Bpmu[\frac 1 \gamma]$ because the ideal
$\ker \theta_\mu$ is not stable under Frobenius. 
Formula~\eqref{eq:prodt} shows that inverting $t$ is very natural 
if our objective is to keep an action of the Frobenius.
\end{rem}

\begin{prop}
For $\mu \geq 1$ or $\mu = \crys$, the inclusions
$\Bmu \subset \BmuK \subset \BdR$ induce 
$G_K$-equivariants isomorphisms of rings
$\gr\,\Bmu \simeq
\gr\,\BmuK \simeq
\gr\,\BdR \simeq B_\HT$.
\end{prop}

\begin{proof}
The fact that $\gr\,\BdR$ is isomorphic to $B_\HT$ has been
already noticed.
Now, consider the composite
$f : \gr\,\Bpmu \to \gr\,\Bmu \to \gr\,\BmuK \to
\gr\,\BdR \simeq \gr\,\BpdR$
in which all maps are injective. Since $\BpdR$ is the completing
of $\Bpmu$ with respect to the de Rham filtration, the map $f$
has to be an isomorphism. The proposition follows.
\end{proof}

\subsection{$\Bcrys$ and $\BdR$ as period rings}

In order to apply Fontaine's general strategy (discussed in 
\S\ref{ssec:Fontainestrategy}) with the $\Bmu$'s (for $\mu \geq 1$ or 
$\mu = \crys$ or $\mu = \dR$)---and then ``promote'' these rings at the 
level of genuine period rings---a final couple of verifications still 
need to be done; precisely
we need to check that the $\Bmu$'s satisfy Fontaine's 
hypotheses (H1), (H2) and (H3) introduced in~\S\ref{sssec:criterium},
and we need to compute the invariants under the $G_K$-action.

We start with $\BdR$ which is easier. First, since it is a field,
Fontaine's hypotheses are obviously fulfilled. Concerning the
computation of the fixed points, we have the following theorem.

\begin{theo}
\label{theo:bdrgk}
We have $(\BdR)^{G_K} = K$.
\end{theo}

\begin{proof}
We have already seen that $K$ embeds into $\BdR$, so that
$K \subset (\BdR)^{G_K}$. The reverse inclusion follows from
the fact that $(\gr\, \BdR)^{G_K} = (B_\HT)^{G_K} = K$.
\end{proof}

We now move to the crystalline setting, that is the ring $\Bcrys$ and 
its variant $\Bmu$ with $\mu \geq 1$.

\begin{theo}
\label{theo:bmugk}
For $\mu \geq 1$ or $\mu = \crys$, we have
$(\Bmu)^{G_K} = K_0$ and
$(\BmuK)^{G_K} = K$.
\end{theo}

\begin{proof}
We have already seen that $K_0 \subset (\Bmu)^{G_K}$ and
$K \subset (\BmuK)^{G_K}$.
From $\BmuK \subset \BdR$, we deduce that $(\BmuK)^{G_K} 
\subset (\BdR)^{G_K} = K$, the latter equality resulting from 
Theorem~\ref{theo:bdrgk}. Hence we have proved that $(\BmuK)^{G_K} = 
K$.
Now remember that, by definition, $\BmuK = K \otimes_{K_0}
\Bmu$. Taking the $G_K$-invariants, we obtain $K = K \otimes_{K_0}
(\Bmu)^{G_K}$, from which we deduce $(\Bmu)^{G_K} = K_0$.
\end{proof}

\begin{prop}
For $\mu \geq 1$ or $\mu = \crys$, 
the rings $\Bmu$ and $\BmuK$ satisfy Fontaine's hypotheses.
\end{prop}

\begin{proof}
It is clear that $\Bmu$ and $\BmuK$ are domains since they
both embed into $\BdR$ which is a field. Repeating the proof of
Theorem~\ref{theo:bmugk}, we find that $(\Frac \Bmu)^{G_K} = K_0$
and $(\Frac \BmuK)^{G_K} = K$. Hence $\Bmu$ and $\BmuK$
satisfy hypothesis (H2). 

Let us now prove that $\Bmu$ satisfies Fontaine's hypothesis (H3). 
Let $x \in \Bmu$, $x \neq 0$ and assume that the line $\Qp x$ is stable 
under the action of~$G_K$. We have to prove that $x$ is invertible in 
$\Bmu$. In what follows, we will consider $x$ as an element on $\BdR$.
Replacing possibly $x$ by $t^n x$ for some integer $n$ (which is safe
since $t$ is invertible in $\Bmu$), we may 
assume that $x \in \BpdR$ and $x \not\in \Fil^1 \BpdR$. The morphism 
$\theta_\dR$ then induces a $G_K$-equivariant embedding $\Qp x 
\hookrightarrow \Cp$. Thus the representation $\Qp x$ is 
$\Cp$-admissible. By Theorem~\ref{theo:Cpadm}, the inertia subgroup 
$I_K$ of $G_K$ acts on $x$ through a finite quotient. Therefore there 
exists a positive integer $n$ such that $I_K$ acts trivially on $y = 
x^n$. The line $\Qp y$ then inherits an action of $G_K/I_K = 
\Gal(K^\ur/K) \simeq \Gal(K_0^\ur/K_0)$.
Applying Proposition~\ref{prop:H90unram} to the 
$\Gal(K_0^\ur/K_0)$-representation $\hat K_0^\ur y$ (recall that 
$\hat K_0^\ur \subset \Bpinf \subset \BdR$), we find that there 
exists $\lambda \in \hat K_0^\ur$ such that $\lambda y$ is fixed
by $G_K$. By Theorem~\ref{theo:bmugk}, we obtain $\lambda y \in K_0$
and then $y \in \hat K_0^\ur$.
We deduce that $y$ is invertible in $\Bmu$, and so also is $x$.

The fact that $\BmuK$ satisfies (H3) is proved in a similar fashion.
\end{proof}

We conclude this section by stating another important property
of the rings $\Bmu$. 

\begin{prop}
\label{prop:bmuphi1}
Let $\mu \geq 1$ and $\mu = \crys$.
Let $x \in \Fil^0 \Bmu$ such that $\varphi(x) = x$, then 
$x \in \Qp$.
\end{prop}

\begin{proof}
We first prove the proposition when $\mu = \mu_0 = \frac p{p-1}$.
In this case, it is easily checked that $A_{\mu_0}$ is the $p$-adic 
completion of $\Ainf[\frac t p]$. This implies that
$A_{\mu_0} \subset \Ainf + \frac t p A_{\mu_0}$ and thus, inverting $p$, we 
find $B^+_{\mu_0} \subset \Bpinf + t B^+_{\mu_0}$.

Let $x \in \Fil^0 B_{\mu_0}$ such that $\varphi(x) = x$. By definition
of $B_{\mu_0}$, we can write $x = t^{-m} y$ with $m \in \N$ and $y \in 
B^+_{\mu_0}$. We choose $m$ minimal with this property. We assume by
contradiction that $m > 0$. By the first paragraph of the proof, we 
can write
$y = a + tb$ with $a \in \Bpinf$ and $b \in B^+_{\mu_0}$. Besides,
for any nonnegative integer $n$, we have $\varphi^n(y) = p^{nm} y$
and then:
$$\theta \circ \varphi^n(a) 
= \theta \circ \varphi^n \big(y - t b\big)
= \theta \big(p^{mn} y - p^n t \varphi^n(b)\big) 
= p^{mn} \cdot \theta_{\mu_0}(y) = 0$$
the last equality coming from the fact that $y = t^m x \in \Fil^m
B_{\mu_0} \subset \Fil^1 B_{\mu_0}$.
By Proposition~\ref{prop:kerthetaphi}, we find that $[\Ueps]-1$
divides $a$ in $\Bpinf$. On the other hand, from the definition 
of $t$, we have:
$$pt = \varphi(t) = 
\big([\Ueps]^p-1\big) \cdot \sum_{i=1}^\infty
(-1)^{i-1} \frac{([\Ueps]^p{-}1)^{i-1}} i,$$
from what we derive that $t$ and $[\Ueps]^p-1$ differ by a unit 
in $B^+_{\mu_0}$. From the divisibility observed above, we deduce 
that $[\Ueps]^p-1$ divides $\varphi(a) = p^m a + p^m t b - pt
\varphi(b)$ in $B^+_{\mu_0}$. Therefore $t$ must divide $a$ in 
$B^+_{\mu_0}$, which contradicts the minimality of $m$. As a
conclusion, we find $m = 0$, \emph{i.e.} $x \in B^+_{\mu_0}$.

Write $x = a + tb$ with $a \in \Bpinf$ and $b \in B^+_{\mu_0}$.
The equality $x = \varphi(x)$ gives $x = \varphi^n(a) + p^n t 
\varphi^n(b)$ for all~$n$. Therefore $\varphi^n(a)$ converges to $x$
when $n$ goes to infinity. Since $\Bpinf$ is closed in $B^+_{\mu_0}$,
we deduce that $x \in \Bpinf$. Finally, remembering that
$\Bpinf = W(\OCpflat)[\frac 1 p]$, we obtain $x \in
W(\Fp)[\frac 1 p]$, that is $x \in \Qp$.

We now move to the general case. 
Let $x \in (\Fil^0 \Bmu)^{\varphi = 1}$. In particular $x \in
\Fil^0 \Bmax$ and therefore $x = \varphi(x) \in \Fil^0 B_p
\subset \Fil^0 B_{\mu_0}$. 
The conclusion now follows by the first part of the proof.
\end{proof}

\begin{rem}
Proposition~\ref{prop:bmuphi1} can be written in the shorter form:
$$\big(\Fil^0 \Bmu\big)^{\varphi = 1} = \Qp$$
where the exponent ``$\varphi{=}1$'' means that we are taking the 
subspace of fixed points under $\varphi$.
The reader should be aware that restricting to $\Fil^0$ is
essential: $\Bmu^{\varphi = 1}$ is much bigger than $\Qp$. 
Precisely, we have the so-called fundamental exact sequence:
$$0 \to \Qp \to \Bmu^{\varphi = 1} \to \BdR/\BpdR \to 0$$
where the map $\Bmu^{\varphi = 1} \to \BdR/\BpdR$ is induced
by the natural inclusion $\Bmu \hookrightarrow \BdR$.
\end{rem}

\section{Crystalline and de Rham representations}
\label{sec:crysdRrep}

We keep the general notations of the previous section: the letter $K$ 
denotes a finite extension of $\Qp$, $G_K$ is its Galois group, 
\emph{etc.} So far, we have defined the periods rings $\Bcrys$ and 
$\BdR$ (together with some variants). By Fontaine's formalism 
(cf~\S\ref{ssec:Fontainestrategy}), these rings cut out full 
subcategories of $\Rep_{\Qp}(G_K)$. The objective of this section is 
to study these categories and to demonstrate that they are relevant 
for geometric purpose.
We begin with a definition.

\begin{deftn}
Let $V$ be a finite dimension $\Qp$-linear representation
of $G_K$.

\vspace{-2mm}

\begin{enumerate}[(i)]
\renewcommand{\itemsep}{0pt}
\item We say that $V$ is \emph{crystalline} if it is
$\Bcrys$-admissible.
\item We say that $V$ is \emph{de Rham} if it is
$\BdR$-admissible.
\end{enumerate}
\end{deftn}

\noindent
Rephrasing the definition of $B$-admissibilty and using
Theorems~\ref{theo:bdrgk} and~\ref{theo:bmugk}, we have:
$$\begin{array}{lcl}
V \text{ is crystalline}
& \Longleftrightarrow & 
\dim_{K_0} \big(\Bcrys \otimes_{\Qp} V)^{G_K} = \dim_{\Qp} V,
\smallskip \\
V \text{ is de Rham}
& \Longleftrightarrow & 
\dim_K \big(\BdR \otimes_{\Qp} V)^{G_K} = \dim_{\Qp} V.
\end{array}$$
Moreover, since $\Bcrys$ is a subring of $\BdR$, any crystalline 
representation is de Rham.

\subsection{Comparison theorems: statements}
\label{ssec:comparison}

We start by discussing the geometric relevance of the notion of 
crystalline and de Rham representations. Our ambition is only to state 
the relevant theorems in this direction and definitely not to prove 
them. The most important ingredients of the proofs will be presented and 
discussed in Yamashita's lecture~\cite{yamashita} and Andreatta and 
al.'s lecture~\cite{andreatta} in this volume.
From now on, we fix a proper smooth variety $X$ defined over $\Spec K$. 
(At least) two different cohomology theories taking coefficients in 
$\Qp$ can be naturally attached to $X$, namely:

\vspace{-2mm}

\begin{itemize}
\renewcommand{\itemsep}{0pt}
\item the (algebraic) de Rham cohomology $H_\dR^\bullet(X)$ of $X$: each 
component $H_\dR^r(X)$ is a $K$-vector space endowed with a descreasing 
filtration, denoted by $\Fil^m H^r_\dR(X)$, with $\Fil^0 H^r_\dR(X) =
H^r_\dR(X)$ and $\Fil^{r+1} H^r_\dR(X) = 0$

\item the $p$-adic \'etale cohomology $H_\et^\bullet(X_{\bar K}, \Qp)$ 
where $X_{\bar K} = \Spec\:\bar K \times_{\Spec K} X$: each component 
$H_\et^r(X_{\bar K}, \Qp)$ is a $\Qp$-vector space endowed with a 
continuous action of $\Gal(\bar K/K)$.
\end{itemize}

\noindent
In the early 1970's, Grothendieck~\cite{Gro71} wondered whether one can 
compare these cohomology groups. More precisely, he raised the so-called 
\emph{problem of the mysterious functor}, asking for the existence of a 
purely algebraic recipe to recover $H^r_\dR(X)$ from 
$H^r_\et(X_{\bar K}, \Qp)$.
When $X_\C$ is a complex variety, the problem of the ``mysterious'' 
functor has been solved for a long time; indeed, the de Rham comparison 
theorem ensures that $H^r_\dR(X_\C)$ is isomorphic to the singular 
cohomology of $X_\C(\C)$ with coefficients in $\C$ (which plays the role 
of the étale cohomology). As we shall see, the $p$-adic case is more 
subtle.

Using standard arguments, one proves that $H_\dR^r(X)$ and 
$H_\et^r(X_{\bar K}, \Qp)$ have the same dimension for all $r$. Thus 
$K \otimes_{\Qp} H_\et^r(X_{\bar K}, \Qp)$ has to be isomorphic to 
$H_\dR^r(X)$ as abstract $K$-vector spaces. However there does not 
exist any \emph{functorial} isomorphism between
them. Therefore the coincidence of dimensions cannot be considered
as a satisyfing answer to Grothendieck's question.

Hodge-like decomposition theorems discussed in \S\ref{ssec:motivations} 
(see in particular Eq.~\eqref{eq:hodgeabelian}) constitute a significant 
process towards Grothendieck's problem. Indeed they show, for some
particular $X$'s, that $H^r_\et(X_{\bar K}, \Qp)$ is isomorphic to
the \emph{graded} module of $H_\dR^r(X)$ after extending scalars to 
$\Cp$.
However the de Rham filtration on $H_\dR^r(X)$ is not canonically
split in the $p$-adic setting; therefore some information is lost
when passing to the graduation.
The point, which was first formulated by Fontaine and Jannsen,
is that we can recover this missing information by extending scalars
to the larger field $\BdR$.
This is the content of the $C_\dR$ theorem\footnote{This result is 
sometimes referred to as the $C_\dR$-conjecture (even if it is now 
proved) since it has been a conjecture for a long time. The letter ``C'' 
in $C_\dR$ stands for ``comparison'' or ``conjecture''.}:

\begin{theo}[$C_\dR$]
\label{theo:cdr}
Let $X$ be a proper smooth variety over $\Spec K$.
For all $r$, there exists a canonical isomorphism:
\begin{equation}
\label{eq:cdr}
\gamma_\dR(X) : \BdR \otimes_K H_\dR^r(X) \simeq 
\BdR \otimes_{\Qp} H_\et^r(X_{\bar K}, \Qp)
\end{equation}
which respects filtrations and Galois action on both sides. Moreover
$\gamma_\dR(X)$ is functorial in $X$.
\end{theo}

\noindent
In the above theorem, the filtration on the source of $\gamma_\dR(X)$ is
the ``convolution'' filtration:
$$\Fil^m \big(\BdR \otimes H_\dR^r(X) \big) = \sum_{a+b=m}
\Fil^a \BdR \otimes_K \Fil^b H_\dR^r(X)$$
whereas, on the target, the filtration comes only from that on $\BdR$.
In the same way, the Galois action on the source (resp. on the target)
of \eqref{eq:cdr} is the diagonal action (resp. the action coming from
that on $\BdR$).

We observe that Theorem~\ref{theo:cdr} implies readily that the
$\Qp$-linear representation $H_\et^r(X_{\bar K}, \Qp)$ is de Rham.
Moreover, taking $G_K$-invariants on both side of~\eqref{eq:cdr}, we 
find a natural isomorphism:
\begin{equation}
\label{eq:ettodr}
H_\dR^r(X) \simeq \big(\BdR \otimes_{\Qp}
H_\et^r(X_{\bar K},\Qp)\big)^{G_K}
\end{equation}
which gives a satisfactory answer to Grothendieck's mysterious functor
problem. Similarly, passing to the graduation in~\eqref{eq:cdr}, we 
obtain the following Hodge-like decomposition:
\begin{equation}
\Cp \otimes_{\Qp} H^r_\et(X_{\bar K}, \Qp) 
\, \simeq \, \bigoplus_{a+b=r} \Cp(\chi_\cycl^{-a})
 \otimes_K H^b(X, \O_X)
\end{equation}
extending Tate's theorem on abelian varieties (\emph{cf} 
\S\ref{ssec:motivations}). Observe in addition that the above 
isomorphism gives the Hodge--Tate decomposition of 
$H_\et^r(X_{\bar K}, \Qp)$. In particular, we see that all the
Hodge--Tate weights of $H_\et^r(X_{\bar K}, \Qp)$ are in the 
range $[-r,0]$.

\paragraph{The Fontaine--Mazur conjecture}

A classical application of Theorem \ref{theo:cdr} is to prove that a 
representation $V$ does \emph{not} come from geometry: if we can prove 
that $V$ is not de Rham (or not Hodge--Tate), it can't arise as the 
\'etale cohomology of a proper smooth variety. One may ask for the 
converse: does any de Rham representation arise as a subquotient of a 
Tate twist of the \'etale cohomology of some variety? In the local 
situation considered up to now, the answer is negative. Nevertheless, a 
``global'' variant of this question is conjectured to admit a positive 
answer. It is the so-called Fontaine--Mazur conjecture, which first 
appeared in~\cite{FM95}.

Let $F$ be a number field, that is a finite extension of $\Q$. 
For any prime ideal $\mathfrak p$ in $\O_F$ 
(the ring of integers of $F$), one can consider the field $F_{\mathfrak p}$ 
defined as the completion of $F$ with respect to the $\mathfrak p$-adic 
topology. If $p$ is the prime number defined by $p \Z = \Z \cap 
\mathfrak p$, the field $F_{\mathfrak p}$ is a finite extension of $\Qp$.
Moreover its absolute Galois group $\Gal(\Qpbar/F_{\mathfrak p})$ embeds 
into $\Gal(\bar\Q/F)$. This embedding is not unique but it is up to conjugacy 
by an element of $\Gal(\bar\Q/F)$.
Therefore, if $V$ is a $\Qp$-representation of $\Gal(\bar\Q/F)$, its 
restrction to $\Gal(\Qpbar/F_{\mathfrak p})$ is well defined and it makes sense 
to wonder whether it is de Rham or not. 
In the same way, a representation of $\Gal(\bar\Q/F)$ (with coefficients 
in any ring) is said to be \emph{unramified} at $\mathfrak p$ if its restriction to
$\Gal(\Qpbar/F_{\mathfrak p})$ is unramified (\emph{i.e.} if the inertia subgroup
of $\Gal(\Qpbar/F_{\mathfrak p})$ acts trivially on it).

\begin{conj}[Fontaine--Mazur]
We fix a number field $F$ and a prime number $p$.
Let $V$ be a finite dimensional $\Qp$-representation of 
$\Gal(\bar\Q/F)$. We assume that:

\vspace{-2mm}

\begin{enumerate}[(i)]
\renewcommand{\itemsep}{0pt}
\item for almost\footnote{``almost all'' means ``all expect possibly
a finite number of them''} all prime ideals $\mathfrak p \in \O_F$, the
representation $V$ is unramified at $\mathfrak p$,
\item for all primes $\mathfrak p$ above $p$ (\emph{i.e.} such that
$\Z \cap \mathfrak p = p \Z$), the representation $V_{|\Gal(\Qpbar / 
F_{\mathfrak p})}$ is de Rham.
\end{enumerate}

\vspace{-2mm}

\noindent
Then $V$ appears as a subquotient of some $H_\et^r(X_{\bar\Q}, \Qp)
(\chi_\cycl^m)$ where
$r$ is a nonnegative integer, $X$ is a proper smooth variety defined
over $\Spec F$ and $m$ is an integer.
\end{conj}

When a representation $V$ satisfies the conclusion of the above 
conjecture, we usually say that $V$ \emph{comes from geometry}. From 
the $C_\dR$-theorem, we derive that every representation of the shape 
$H_\et^r(X_{\bar\Q}, \Qp) (\chi_\cycl^m)$ comes from geometry. The 
Fontaine--Mazur conjecture then appears as a purely algebraic criterium 
to recognize representations coming from geometry among all 
representations.

We would like to emphasize that the Fontaine--Mazur conjecture might 
look surprising at first glance. Indeed it has been known for a long 
time that the Galois action on the \'etale cohomology satisfies many 
additional properties: for instance, the eigenvalues of the Frobenius 
acting on the \'etale cohomology have to take very particular values, 
known as Weyl numbers. However, these properties are not required in 
Fontaine--Mazur conjecture. It means that, assuming the conjecture to be 
true, they are implied by the unramified and the de Rham conditions, 
which is \emph{a priori} rather unexpected.

Nowadays, the Fontaine--Mazur conjecture is still open. It was recently 
proved by Emerton \cite{Eme} and Kisin \cite{Kis09} for two-dimensional 
representations of $\Gal(\bar\Q/\Q)$ (satisfying some additional mild 
conditions) using the most recent developments in algebraic number 
theory (\emph{e.g.} modularity lifting theorems, $p$-adic Langlands 
program). As far as we know, beyond the dimension $2$, nothing is known.

\paragraph{The $C_\crys$-theorem}

We now go back to the local setting and examine the case of the variety 
$X$ has good reduction.
We recall that this means that there exists a proper smooth 
variety $\calX$ over $\Spec \O_K$ whose generic fiber is $X$. We 
emphasize that the model $\calX$ is required to be smooth; it is the 
crucial assumption.

When $X$ has good reduction, the de Rham cohomology of $X$ carries more 
structures. Indeed, assuming that $X$ has good reduction, one can fix a 
model $\calX$ as above and consider its special fiber $\bar \calX$. It 
is a proper smooth scheme defined over $\Spec k$.
To $\bar \calX$, one can attach a third 
cohomology group: its \emph{crystalline} cohomology $H^r_\crys 
(\bar\calX)$, defined by Berthelot~\cite{Ber74}.
We refer to~\cite{Ber74,BO78} for a complete exposition of the
crystalline theory.
For this article, let us just recall \emph{very} briefly that, for all 
positive integer $r$, the crystalline cohomology $H^r_\crys(\bar\calX)$ 
is a module over $W(k)$ endowed with an endomorphism $\varphi :
H^r_\crys(\bar\calX) \to H^r_\crys(\bar\calX)$ which is semi-linear with 
respect to the Frobenius on $W(k)$.
In addition, the crystalline cohomology of $\bar\calX$ is closely 
related to the de Rham cohomology of $X$ through the Hyodo--Kato 
isomorphism $K \otimes_{W(k)} H^r_\crys(\bar\calX) \simeq H^r_\dR(X)$. 
Putting $K_0 = W(k)[\frac 1 p]$ as before, we see that Hyodo--Kato 
isomorphism defines a $K_0$-structure in $H^r_\dR(X)$, namely
$K_0 \otimes_{W(k)} H^r_\crys(\bar\calX)$. One can prove that this 
structure is canonical in the sense that it does not depend on the 
choice of a proper smooth model $\calX$ of $X$.

\begin{theo}[$C_\crys$]
\label{theo:ccris}
Let $X$ be a proper smooth variery over $\Spec\:K$ with good 
reduction. Let $\calX$ denote a proper smooth
model of $X$ over $\Spec\:\O_K$.
For all $r$, there exists a canonical and functorial isomorphism:
\begin{equation}
\label{eq:ccris}
\gamma_\crys(X) : \Bcrys \otimes_W H_\crys^r(\bar\calX) \simeq 
\Bcrys \otimes_{\Qp} H_\et^r(X_{\bar K}, \Qp)
\end{equation}
which respects Galois action and Frobenius action on both sides and such 
that $\BdR \otimes \gamma_\crys(X)$ respects filtrations.
\end{theo}

\noindent
Here, the Frobenius action is defined on the 
source (resp. the target) of $\gamma_\crys(X)$ as the diagonal action 
(resp. the action given by $\varphi \otimes \text{id}$).
Theorem~\ref{theo:ccris} shows that the representation 
$H_\et^r(X_{\bar K}, \Qp)$ is crystalline as soon as $X$ has good
reduction. (Remember that we already knew that this representation
was de Rham thanks to the $C_\dR$-theorem.)
More precisely, taking $G_K$-invariants 
on both sides of \eqref{eq:ccris}, we obtain:
\begin{equation}
\label{eq:ettocrys}
H_\crys^r(\calX) \simeq \big(\Bcrys \otimes_{\Qp}
H_\et^r(X_{\bar K},\Qp)\big)^{G_K}
\end{equation}
which shows that, when $X$ has good reduction, the étale cohomology of 
$X$ not only determines its de Rham cohomology but also its canonical
$K_0$-structure coming from the crystalline cohomology.
After the results of \S\ref{ssec:crys} (to come up), it turns out that 
the converse also holds true: the crystalline cohomology, equipped with 
its Frobenius and the de Rham filtration after scalar extension to $K$, 
determines the étale cohomology.

\paragraph{A brief history of the $C_\dR$-theorem}

Theorems~\ref{theo:cdr} and~\ref{theo:ccris} were first stated as
conjecture by Fontaine and 
Jannsen just after Fontaine introduced
the corresponding periods rings $\BdR$ and $\Bcrys$ respectively.
Fontaine also designed a strategy to prove these conjectures. Very 
roughly, it can be summarized as follows:

\vspace{-2mm}

\begin{enumerate}
\renewcommand{\itemsep}{0pt}
\item prove the $C_\crys$-conjecture;
\item extend the $C_\crys$-conjecture to the semi-stable 
case\footnote{We say that a variety $X$ over $K$ has 
\emph{semi-stable} reduction if it has a proper model $\calX$ over
$\Spec\:\O_K$ whose generic fibre is a divisor with normal
crossings.};
\item derive the $C_\dR$-conjecture by reduction to the semi-stable
case.
\end{enumerate}

\vspace{-2mm}

\noindent
The case of $C_\crys$ looks easier than that of $C_\dR$ because
the isomorphism~\eqref{eq:ccris} can be understood as a kind of
Kunneth formula. Indeed, the period ring $\Bcrys$ has a nice
cohomological interpretation, that is $\Bcrys = H^0_\crys\big(\O_{\bar 
K}/p\O_{\bar K}\big)$. It 
then becomes plausible that $\Bcrys \otimes_{K_0} 
H^r_\crys(\bar\calX)$ could have something to do with the cohomology of 
$X_{\bar K}$.
Beyond this remark, it remained to find the way to go back and forth 
between the crystalline and the étale cohomologies. To this end, 
Fontaine and Messing 
proposed to use a third cohomology, the syntomic cohomology, and
to compare it to both sides of the isomorphism~\eqref{eq:cdr}. 
Using these ideas, they managed to prove the $C_\crys$-theorem 
under the additional assumption that $X$ has dimension at most 
$\frac{p-1}2$~\cite{FM87}.

Regarding the second step, Fontaine and Illusie introduced and proposed 
to develop log geometry. The main feature of log geometry is that it 
sees a normal crossing divisor as a log-smooth scheme. It then should be 
the right framework to perform local computations in the semi-stable 
case and then, hopefully, to extend the proof by Fontaine and Messing to 
all varieties admitting semi-stable reduction. The development of log 
geometry was achieved by the Japanese school \cite{K89,HK94}, who 
defined an analogue of the crystalline cohomology in this setting --- 
the so-called log-crystalline cohomology --- and related it to the de 
Rham cohomology \emph{via} a log-analogue of the Hyodo--Kato 
isomorphism.

The initial idea for the third step was to prove that every proper 
smooth variety over $\Spec K$ admits semi-stable reduction after a 
finite extension. Unfortunately, this problem turns out to be quite 
difficult and is still open nowadays. Nevertheless, de Jong~\cite{dJ96, 
dJ97, Ber97} proved a weaker result which was enough to complete the 
last step of Fontaine's strategy. In the very long paper~\cite{tsuji}, 
Tsuji gathered all these inputs and finally came up with a complete 
proof of the $C_\dR$-theorem. The main ingredients of the proof will
be presented in Yamashita's lecture~\cite[\S 2]{yamashita} in this 
volume.

In the meanwhile, Faltings published another proof of the $C_\crys$
and $C_\dR$-theorem~\cite{Fal} (but did not state a semi-stable version).
Faltings' strategy is quite different from Fontaine's one and relies on 
\emph{almost mathematics}, a theory specifically developped by Faltings 
for this application, which can be thought of as a wild generalization
of Tate--Sen's methods presented in \S\ref{sec:Cp}.
The common idea which unifies these two proofs is, roughly speaking, to 
develop advanced methods to control extensions obtained by extracting 
$p$-th roots: in Fontaine's approach, it is achieved by the syntomic 
topology\footnote{A morphism of schemes obtained by extraction of a 
$p$-th 
root of some function turns out to be a covering for the syntomic 
cohomology.} whereas Faltings' initial idea is to work over infinite 
extensions obtained by extracting successive $p$-th roots and to use 
almost mathematics as the main tool to study the cohomology of varieties 
defined over such extensions.

More recently, Scholze designed a very powerful framework to do geometry 
over many ``very ramified'' bases including those obtained from usual 
$\Zp$-schemes by adjoining iterated $p$-th roots: it is the theory of 
\emph{perfectoid spaces}~\cite{Sch12}.
Based on this, he obtained in~\cite{Sch13} a new proof of the 
$C_\dR$-theorem which extends readily to \emph{analytic} varieties 
(without any hypothesis of type K\"ahler)!
This proof will be sketched in the article of Andreatta and al.'s in 
this volume~\cite{andreatta}.

\subsection{More on de Rham representations}
\label{ssec:dR}

Now we have seen the relevance of crystalline and de Rham
representations, it looks important to study systematically their
properties. We start with the de Rham case.
Let $\Rep_{\Qp}^\dR(G_K)$ denote the category of $\Qp$-linear de Rham 
representations of $G_K$. By Fontaine's general formalism, we know 
that $\Rep_{\Qp}^\dR(G_K)$ is a full abelian subcategory of $\Rep_{\Qp}
(G_K)$. It is moreover stable under direct sums, duals, tensor products, 
subobjects and quotients.

\begin{theo}
\label{theo:CpdR}
Any finite dimensional $\Cp$-admissible representation of
$G_K$ is de Rham.
\end{theo}

\begin{proof}
Let $V$ be a finite dimensional $\Cp$-admissible representation
of~$G_K$.
By Remark~\ref{rem:Cpadm}, there exists a finite extension $L$
of $K^\ur$ such that $V$ is $(L {\cdot}\hat K^\ur)$-admissible.
Since $L{\cdot}\hat K^\ur \subset \BdR$, we conclude that $V$ is
de Rham.
\end{proof}

Another interesting result is that de Rham representations can be 
detected by looking at the restriction to open subgroups. Precisely,
we have the following theorem.

\begin{theo}
\label{theo:dRrest}
Let $L$ be a finite extension of $K$ and 
let $V$ be a finite dimensional $\Qp$-linear representation of $G_K$.
Then $V$ is de Rham if and only if $V_{|G_L}$ is de Rham.
\end{theo}

\begin{rem}
In other terms, Theorem~\ref{theo:dRrest} says that the
following diagram is cartesian.
$$\xymatrix @C=3em {
\Rep_{\Qp}^\dR(G_K) \ar@{^(->}[r] \ar[d]_-{\text{restriction}} &
\Rep_{\Qp}(G_K) \ar[d]^-{\text{restriction}} \\
\Rep_{\Qp}^\dR(G_L) \ar@{^(->}[r] &
\Rep_{\Qp}(G_L)}$$
\end{rem}

\begin{proof}[Proof of Theorem~\ref{theo:dRrest}]
By definition, if $V$ is de Rham, the $\BdR$-semi-linear 
representation $\BdR \otimes_{\Qp} V$ is trivial as a 
$G_K$-representation. It is then a fortiori trivial as a
$G_L$-representation, which means that $V_{|G_L}$ is de Rham.

Conversely, let us assume that $V_{|G_L}$ is de Rham. Without
loss of generality, we may assume that the extension $L/K$ is
Galois (if not, replace $L$ by its Galois closure).
Define $D = (\BdR \otimes_{\Qp} V)^{G_L}$, so that we have $\dim_L 
D = \dim_{\Qp} V$. Moreover $D$ inherits a semi-linear action of 
$\Gal(L/K)$. By Hilbert's theorem 90 (\emph{cf} Theorem~\ref{theo:H90}), 
$D$ is spanned by a basis of fixed vectors. In other words, 
$\dim_K D^{\Gal(L/K)} = \dim_L D = \dim_{\Qp} V$. Since
$D^{\Gal(L/K)} = (\BdR \otimes_{\Qp} V)^{G_K}$, we have proved
that $V$ is de Rham.
\end{proof}

\paragraph{The function $D_\dR$.}

If $V$ is a de Rham representation of $G_K$, we define:
\begin{equation}
\label{eq:ddr}
D_\dR(V) = \big(\BdR \otimes_{\Qp} V\big)^{G_K}
= \Hom_{\Qp[G_K]} \big(V^\star, \BdR\big)
\end{equation}
where $\Hom_{\Qp[G_K]}$ refers to the set of $\Qp$-linear
$G_K$-equivariant morphisms and $V^\star$ is the dual representation
of $V$. Fontaine's formalism shows that we have a canonical 
isomorphism:
\begin{equation}
\label{eq:dRrepcomp}
\BdR \otimes_{\Qp} V \,\simeq\, \BdR \otimes_K D_\dR(V).
\end{equation}

\begin{rem}
When $V$ is the étale cohomology of a proper smooth variety $X$
over $\Spec K$, the isomorphism~\eqref{eq:dRrepcomp} is the 
isomorphism \eqref{eq:cdr} of the $C_\dR$-theorem. Notably, we have
$H^r_\dR(X) = D_\dR(H^r_\et(X_{\bar K}, \Qp))$ for all integer~$r$.
\end{rem}

Formula~\eqref{eq:ddr} defines a functor $D_\dR : \Rep_{\Qp}^\dR(G_K) 
\to \Vect_K$ where $\Vect_K$ is the category of finite dimensional 
vector spaces over $K$.
One can actually be more precise and endow $D_\dR(V)$ with a 
filtration coming from the filtration on $\BdR$. Precisely, for 
an integer $m \in \Z$, we define:
$$\Fil^m D_\dR(V) = \big(\Fil^m \BdR \otimes_{\Qp} V\big)^{G_K}
= \Hom_{\Qp[G_K]} \big(V^\star, \Fil^m \BdR\big).$$
Clearly $\Fil^m D_\dR(V)$ is sub-$K$-vector space of $D_\dR(V)$
and $\Fil^{m+1} D_\dR(V) \subset \Fil^m D_\dR(V)$ for all $m$.
Moreover observe that:
\begin{align*}
\bigcap_{m \in \Z} \Fil^m D_\dR(V)
& = {\textstyle 
\Hom_{\Qp[G_K]} \Big(V^\star,\, \bigcap_{m \in Z} \Fil^m \BdR\Big)}
= 0 \\
\text{and} \quad
\bigcup_{m \in \Z} \Fil^m D_\dR(V)
& = {\textstyle 
\Hom_{\Qp[G_K]} \Big(V^\star,\, \bigcup_{m \in Z} \Fil^m \BdR\Big)}
= D_\dR(V),
\end{align*}
the second equality coming from the fact that $D_\dR(V)$ has finite
dimension over $K$. Since again $D_\dR(V)$ is finite dimensional, we
deduce that $\Fil^m D_\dR(V) = 0$ for $m \gg 0$ and $\Fil^m D_\dR(V)
= D_\dR(V)$ for $m \ll 0$; we say that the filtration of $D_\dR(V)$
is \emph{separated} and \emph{exhaustive}.

With the above construction, we have promoted $D_\dR$
to a functor $D_\dR : \Rep^\dR_{\Qp}(G_K) \to \MF_K$ where
$\MF_K$ denotes the category of finite dimension $K$-vector
spaces equipped with a nonincreasing separated and exhaustive
filtration by sub-$K$-vector spaces.
This functor has an extra remarkable property given by the next
proposition.

\begin{prop}
\label{prop:dRrepcompfil}
For any $V \in \Rep^\dR_{\Qp}(G_K)$ and any integer $m$,
the isomorphism~\eqref{eq:dRrepcomp} identifies
$\Fil^m (\BdR \otimes_{\Qp} V)$ with $\Fil^m (\BdR \otimes_K
 D_\dR(V))$ where, by definition:
\begin{align*}
\Fil^m \big(\BdR \otimes_{\Qp} V\big)
 & = \Fil^m \BdR \otimes_{\Qp} V \\
\text{and} \quad
\Fil^m \big(\BdR \otimes_K D_\dR(V)\big)
 & = \sum_{a+b = m} \Fil^a \BdR \otimes_K \Fil^b D_\dR(V).
\end{align*}
\end{prop}

\begin{proof}
The inclusion $\Fil^m (\BdR \otimes_K D_\dR(V)) \subset 
\Fil^m (\BdR \otimes_{\Qp} V)$ is easily checked. It is then
enough to show that mapping:
$$f : \gr \big(\BdR \otimes_{\Qp} D_\dR(V)\big) \longrightarrow
\gr \big(\BdR \otimes_{\Qp} V\big) \simeq B_\HT \otimes_{\Qp} V$$
induced by the inverse of~\eqref{eq:dRrepcomp} is an isomorphism.
For this, we consider the exact sequence
$0 \to \Fil^{m+1} \BdR \to \Fil^m \BdR \to \Cp(\chi_\cycl^m)
\to 0$.
Tensoring it by $V$ and taking the $G_K$-invariants, we obtain an
injective morphism $h_m : \gr^m D_\dR(V) \hookrightarrow
(\Cp(\chi_\cycl^m) \otimes_{\Qp} V)^{G_K}$. Taking the direct
sum of the $h_m$'s, we end up with an injective $K$-linear
mapping
$h : \gr\,D_\dR(V) \hookrightarrow (B_\HT \otimes_{\Qp} V)^{G_K}$.
Now observe that $\dim_K \gr\,D_\dR(V) = \dim_K D_\dR(V) = 
\dim_{\Qp} V$ since $V$ is de Rham. On the other hand, 
$\dim_K (B_\HT \otimes_{\Qp} V)^{G_K} \leq \dim_{\Qp} V$ by the
general Fontaine's formalism.
As a consequence, $h$ must be an isomorphism.
We conclude the proof by remarking that $B_\HT \otimes h = f \circ g$ 
where $g$ is the canonical mapping:
$$g : B_\HT \otimes_{\Qp} \gr\,D_\dR(V) \longrightarrow
\gr \big(\BdR \otimes_{\Qp} D_\dR(V)\big).$$
By definition of the filtration on $\BdR \otimes_{\Qp} D_\dR(V)$, $g$ 
is surjective. Since $h$ is a bijection, we deduce, first, that $g$ is 
an isomorphism and, then, that $f$ is an isomorphism as well.
\end{proof}

As a byproduct of the above proof, we obtain the following
quite interesting corollary.

\begin{cor}
\label{cor:dRHT}
Let $V$ de a de Rham representation of $G_K$.
Then $V$ is Hodge--Tate and its Hodge--Tate weights are
the integers $m$ for which $\gr^{-m} D_\dR(V) \neq 0$, the
multiplicity of $m$ being equal to $\dim_K \gr^{-m} D_\dR(V)$.
\end{cor}

\begin{proof}
The corollary follows from the isomorphism
$B_\HT \otimes_{\Qp} V \simeq B_\HT \otimes_K \gr\,D_\dR(V)$,
which was established in the
proof of Proposition~\ref{prop:dRrepcompfil}.
\end{proof}

For one dimensional representations, the converse of 
Corollary~\ref{cor:dRHT} holds. Indeed, if $\chi : G_K \to \Qp^\times$ 
is a Hodge--Tate character, then there exists some integer $m$ for which 
$\chi \cdot \chi_\cycl^m$ is $\Cp$-admissible. 
By Theorem~\ref{theo:CpdR}, we deduce that $\chi \cdot \chi_\cycl^m$ 
is de Rham. Hence $\chi$ is de Rham as well.
However for higher dimensional representation, there do exist
Hodge--Tate representations which are not de Rham. 

\paragraph{$\BdR$-representations.}

After what we have achieved so far, it is quite tempting to study 
$\BdR$-semi-linear representations on their own in the spirit of Sen's 
theory (presented in \S\ref{ssec:Sen}).
This work was achieved by Fontaine in~\cite{fontaine-bdr}. Let us give 
rapidly a few details on Fontaine's results. Let $K_\infty$ denote the 
$p$-adic cyclotomic extension of $K$.
Generalizing Sen's arguments, Fontaine first shows that any 
$\BdR$-semi-linear representation of $G_K$ descends to 
$K_\infty(\!(t)\!)$. We are then reduced to study the 
$K_\infty(\!(t)\!)$-semi-linear representations of $\Gamma = 
\Gal(K_\infty/K)$. Fontaine then defines an analogue of the Sen's 
operator which is no longer a linear map, but instead a derivation.
More precisely, given a $K_\infty(\!(t)\!)$-semi-linear representation 
$W$ of $\Gamma$, Fontaine shows that, for $\gamma \in \Gamma$ 
sufficiently closed to the identity, the formula $\frac{\log \gamma} 
{\log \chi_\cycl(\gamma)}$ defines a $K_\infty$-linear mapping 
$\nabla_W : W \to W$ which satisfies the Leibniz rule, \emph{i.e.}
$$\nabla_W(f w) = \frac{df}{dt} \cdot w + f \cdot \nabla_W(w)
\qquad (f \in K_\infty(\!(t)\!), \, w \in W).$$
Moreover, as in Sen's theory, this construction is functorial 
and the datum of $\nabla_W$ caracterizes the representation~$W$.
For much more details, we refer to Fontaine's original 
paper~\cite{fontaine-bdr}.

\subsection{More on crystalline representations}
\label{ssec:crys}

Let $\Rep^\crys_{\Qp}(G_K)$ be the category of $\Qp$-linear
crystalline representations of $G_K$. It is a abelian 
subcategory of $\Rep^\dR_{\Qp}(G_K)$, which is stable by direct
sums, duals,
tensor products, subobjects and quotients.
Unlike the de Rham case, the fact that a representation is 
crystalline cannot be detected on the restriction to an open subgroup 
in general. Nevertheless, we have a weaker result in this direction.

\begin{prop}
\label{prop:crysrest}
Let $V$ be a finite dimensional $\Qp$-linear representation of 
$G_K$. Then:

\vspace{-2mm}

\begin{enumerate}[(i)]
\renewcommand{\itemsep}{0pt}
\item if $V$ is unramified (\emph{i.e.} the inertia subgroup
acts trivially on $V$), then $V$ is crystalline,
\item if there exists a finite \emph{unramified} extension $L$
of $K$ such that $V_{|G_L}$ is crystalline, then $V$ is
crystalline.
\end{enumerate}
\end{prop}

\begin{proof}
By Proposition~\ref{prop:H90unram}, if $V$ is unramified then
it is $\hat K^\ur$-admissible. Since $\hat K^\ur \subset \Bcrys$,
it is then \emph{a fortioti} crystalline. This proves (i).

We now assume that $V_{|G_L}$ is crystalline for some finite
unramified extension $L$ of $K$. Without loss of generality, we
may assume that $L/K$ is Galois. We let $L_0$ be the maximal
unramified extension of $\Qp$ inside $L$. Then $\Gal(L/K)
\simeq \Gal(L_0/K_0)$. Set $D = (\Bcrys \otimes_{\Qp} V)^{G_L}$;
it is a $L_0$-vector space endowed with a semi-linear action of
$\Gal(L/K) \simeq \Gal(L_0/K_0)$. By Hilbert's theorem 90, we
have $\dim_{K_0} D^{\Gal(L/K)} = \dim_{L_0} D$. Moreover since 
$V_{|G_L}$ is crystalline, we know that $\dim_{L_0} D = \dim_{\Qp} V$. 
Consequently $\dim_{K_0} D^{\Gal(L/K)} = \dim_{\Qp} V$, which proves 
that $V$ is crystalline because $D^{\Gal(L/K)} = (\Bcrys \otimes_{\Qp} 
V)^{G_K}$.
\end{proof}

We insist again on the fact that the assumption that $L/K$ is unramified 
is crucial in Proposition~\ref{prop:crysrest}.(ii). For example, one can 
prove (using Proposition~\ref{prop:crysunram} below for example) that a 
character is crystalline if and only if it is the product of an 
unramified character by a power of the cyclotomic character. In 
particular the finite order character $\omega_\cycl = [\chi_\cycl \mod 
p]$ of $G_{\Qp}$ is not crystalline.

A finite dimensional $\Qp$-linear representation that becomes 
crystalline over a finite extension (non necessarily ramified) is called 
\emph{potentially crystalline}. Combining Theorems~\ref{theo:Cpadm}, 
\ref{theo:CpdR}, \ref{theo:dRrest} and Proposition~\ref{prop:crysrest}, 
we obtain the following diagram of implications:
\begin{center}
\begin{tikzpicture}[xscale=3.2,yscale=1.3]
\node at (0,0) { \ph unramified };
\node at (0.5,0) { \ph $\Longrightarrow$ };
\node[rotate=90] at (0,0.5) { \ph $\Longrightarrow$ };
\node at (1,0) { \ph crystalline };
\node[rotate=90] at (1,0.5) { \ph $\Longrightarrow$ };
\node at (0,1) { \ph $\Cp$-admissible };
\node at (0.55,1) { \ph $\Longrightarrow$ };
\node at (1,1) { \ph pot. crys. };
\node at (1.5,1) { \ph $\Longrightarrow$ };
\node at (2,1) { \ph de Rham };
\node at (2.5,1) { \ph $\Longrightarrow$ };
\node at (3,1) { \ph Hodge--Tate };
\end{tikzpicture}
\end{center}

\begin{prop}
\label{prop:crysunram}
A representation which is at the same time crystalline
and $\Cp$-admissible is unramified.
\end{prop}

\begin{rem}
Recall that, for a Hodge--Tate representation, $\Cp$-admissibility
means that all Hodge--Tate weights are $0$.
Proposition~\ref{prop:crysunram} then says that any crystalline 
representation with Hodge--Tate weights $0$ is unramified.
\end{rem}

\begin{proof}[Proof of Proposition~\ref{prop:crysunram}]
Let $V$ be a crystalline $\Cp$-admissible representation.
From Remark~\ref{rem:Cpadm}, we derive that there exists a finite 
extension $L$ of $K$ such that $V$ is $(L {\cdot} \hat 
K_0^\ur)$-admissible.

Let $f : V^\star \to \BdR$ be a $G_K$-equivariant $\Qp$-linear 
morphism. Since $V$ is crystalline, we know that $f(V^\star) \subset \Bcrys$. 
Similarly, using that $V$ is $(L {\cdot} \hat K_0^\ur)$-admissible, we 
find $f(V^\star) \subset (L {\cdot} \hat K_0^\ur)$. On the other hand, we know 
that $L \otimes_{L_0} \Bcrys$ embebs into $\BdR$. The 
canonical morphism $(L {\cdot} \hat K_0^\ur) \otimes_{\hat K_0^{\ur}} 
\Bcrys \to \BdR$ is then injective. As a consequence $(L {\cdot} \hat 
K_0^\ur) \cap \Bcrys = \hat K_0^\ur$ and we deduce that $f$ takes its 
values in $\hat K_0^\ur$. As a conclusion, $\Hom_{\Qp[G_K]}(V^\star, 
\BdR) = \Hom_{\Qp[G_K]}(V^\star, \hat K_0^\ur)$.

Since $V$ is de Rham, we deduce from the above equality that $V$ is 
$\hat K_0^\ur$-admissible. In particular $V$ embeds into a direct sum of
copies of $\hat K_0^\ur$. Since the inertia subgroup acts trivially
on $\hat K_0^\ur$, it acts trivially on $V$ as well.
\end{proof}

\begin{ex}
We give an example of a two dimensional representation which is
de Rham but not crystalline. 
For any positive integer $n$, let $\varepsilon_n \in \bar K$ be a 
primitive $p^n$-th root of unity. Similarly, let $\varpi_n \in \bar K$ 
be a $p^n$-root of $p$. For any $g \in G_{\Qp}$, there exists a unique 
element $c(g) \in \Zp$ such that $g \varpi_n = \varepsilon_n^{c(g)} 
\varpi_n$ for all $n$. In the language of \S\ref{sec:dRcrys}, the
previous equation reads:
\begin{equation}
\label{eq:gpflat}
g p^\flat = \Ueps^{c(g)} \cdot p^\flat
\qquad (g \in G_{\Qp})
\end{equation}
where $p^\flat = (p, \bar p_1, \bar p_2, \ldots)$ and 
$\Ueps = (1, \bar \varepsilon_1, \bar \varepsilon_2, \ldots)$ are 
the elements of $\OCpflat$ defined in \S\ref{ssec:Ainf}.
A direct computation shows that $c(gh) = c(g) + \chi_\cycl(g) \cdot 
c(h)$ (we say that $c$ is a cocycle). From this observation, we
deduce that the function:
$$G_{\Qp} \to \GL_2(\Qp), \qquad 
g \mapsto \left(\begin{matrix} \chi(g) & c(g)\\ 0 & 1 \end{matrix} \right)$$
is a group homomorphism and then defines a two dimensional 
$\Qp$-linear representation $V$ of~$G_{\Qp}$.
We are going to compute the space $D = \Hom_{\Qp[G_{\Qp}]}(V, \BdR)$.
By the general theory, we know that:

\vspace{-2mm}

\begin{enumerate}[(i)]
\renewcommand{\itemsep}{0pt}
\item $D$ is a $K$-vector space of dimension at most~$2$,
\item $V^\star$ is de Rham if and only if $\dim_K D = 2$,
\item $V^\star$ is crystalline if and only if it is de Rham and
any morphism in $D$ falls in $\Bcrys$.
\end{enumerate}

\noindent
On the other hand, $D$ is canonically in bijection with the set of pairs 
$(x,y) \in \BdR^2$ such that:
\begin{equation}
\label{eq:Bdrst}
g x = \chi_\cycl(g) x
\quad \text{and} \quad
g y = y + c(g) x
\end{equation}
for all $g \in G_{\Qp}$. The pair $(0,1)$ is obviously a solution 
of~\eqref{eq:Bdrst}. Taking Teichmüller representatives and then
passing to the logarithm in~\eqref{eq:gpflat}, we find that
$(t, \log [p^\flat])$ (where we recall that $t = \log [\Ueps]$) 
is formally another solution of~\eqref{eq:Bdrst}. 
It remains to justify that
$\log [p^\flat]$ makes sense in $\BdR$. To this end, we observe
that it can be defined as follows:
$$\log [p^\flat] = \log \frac{[p^\flat]} p =
- \sum_{i=1}^\infty \frac 1 i \cdot \Bigg( 1 - \frac{[p^\flat]} p
\Bigg)^i.$$
(here we have chosen the convention that $\log p = 0$).
Note that the series converges in $\Fil^1 \BpdR$ because $1 - 
\frac{[p^\flat]} p \in \Fil^1 \BpdR$).
The space $D$ is two dimensional and spanned by $(0,1)$ and $(t, \log 
[p^\flat])$. Hence $V^\star$ is de Rham. The fact that $\log [p^\flat] 
\not\in \Bcrys$, \emph{i.e.} that $V^\star$ is not crystalline can be 
checked as follows. Assume by contradiction that $\log [p^\flat] \in 
B_\crys$. Then, it would lies in $\Fil^1 B_\crys$, so that $a = 
\frac{\log [p^\flat]} t \in \Fil^0 B_\crys$. Moreover, we would have 
$\varphi(a) = a$ since the Frobenius takes $[p^\flat]$ to $[p^\flat]^p$.
By Proposition~\ref{prop:bmuphi1}, this would implies that $a \in \Qp$.
Applying Galois to the relation $\log [p^\flat] = a t$, we would 
obtain $a + c(g) = \chi_\cycl(g)$ for all $g \in G_K$, which is
obviously not true. Finally, we deduce that $V^\star$ is not 
crystalline.
\end{ex}

\begin{rem}
The representation $V$ of the previous example is the prototype of 
\emph{semi-stable} representations. On the geometric side, it
corresponds to the Tate curve, which is the prototype of elliptic
curve without good reduction. Semi-stable representations
will be introduced and widely discussed in Brinon's lecture.
In particular, it will be proved in~\cite[Proposition~2.7]{brinon} 
is actually transcendantal over $\Frac B_\crys$.
\end{rem}

\paragraph{About $\Bmu$-admissibility.}

Recall that, in \S\ref{sec:dRcrys}, we have introduced a whole family of 
rings $\Bmu$'s (where $\mu \geq 1$ is a real paramater); there rings
serve as variants of $\Bcrys$, which have the advantage of exhibiting 
more pleasant properties from the algebraic and analytic point of view. 
The next theorem shows that changing $\Bcrys$ by $\Bmu$ does not affect 
the notion of crystalline representation.

\begin{theo}
\label{theo:bmuadm}
Let $\mu \geq 1$ and 
let $V$ be a finite dimension $\Qp$-linear representation
of $G_K$. Then $V$ is \emph{crystalline} if and only if it
is $\Bmu$-admissible.
\end{theo}

\begin{proof}
Since the $\Bmu$'s form a decreasing sequence of rings and
$\Bmu \subset \Bcrys \subset B_{p-1}$ for each $\mu < p{-}1$,
it is enough to show that $\Bmu$-admissibility implies 
$B_{p\mu}$-admissibility for all $\mu\geq 1$. But the latter
assertion follows from the fact that the Frobenius induces a 
Galois equivariant ring isomorphism $\Bmu \stackrel\sim\to 
B_{p\mu}$ and therefore an isomorphism
$(\Bmu \otimes_{\Qp} V)^{G_K} \simeq (B_{p\mu} \otimes_{\Qp} V)^{G_K}$.
\end{proof}

\paragraph{The functor $D_\crys$.}

When $V$ is a crystalline representation of $G_K$, we set:
$$D_\mu(V) = \big(\Bmu \otimes_{\Qp} V\big)^{G_K}
= \Hom_{\Qp[G_K]}\big(V^\star,\,\Bmu\big)$$
for $\mu \geq 1$, $\mu = \crys$ or $\mu = \max$ (which is, we
recall, a redundant notation for $\mu = 1$).
By Theorem~\ref{theo:bmuadm}, $D_\mu(V)$ is a vector space over 
$K_0$ of dimension $\dim_{\Qp} V$.
Observe in addition that $D_\mu(V)$ is equipped with a Frobenius map 
$\varphi : D_\mu(V) \to D_\mu(V)$ which is semi-linear with respect to 
the Frobenius on~$K_0$. Moreover, one checks easily that:
$$K \otimes_{K_0} D_\mu(V) 
= \big(\BmuK \otimes_{\Qp} V\big)^{G_K}
= \big(\BdR \otimes_{\Qp} V\big)^{G_K} = D_\dR(V).$$
Therefore $K \otimes_{K_0} D_\mu(V)$ comes equipped with a
filtration, namely the de Rham filtration.

The inclusion $\Bmu \subset \Bmax$ induces an injective $K_0$-linear 
mapping $f_\mu : D_\mu(V) \to D_\max(V)$, which commutes with all 
additional structures.
Since the source and the target of $f_\mu$ are both $K_0$-vector spaces 
of dimension $\dim_{\Qp} V$, we conclude that $f_\mu$ is an isomorphism. 
In other words, the functor $D_\mu$ does not depend on the choice of 
$\mu$; in what follows, we will prefer the notation $D_\crys$ (in order 
to make apparent the fact that we are considering the crystalline case) 
but the reader should keep in mind that $D_\crys = D_\mu$ for all $\mu$.

The above constructions motivate the following definition.

\begin{deftn}
A \emph{filtered $\varphi$-module} over $K$ is a $K_0$-vector space $D$ 
equipped with a semi-linear endomorphism $\varphi : D \to D$ and a 
nonincreasing, exhaustive and separated filtration on $K \otimes_{K_0} 
D$.
\end{deftn}

We denote by $\MF_K(\varphi)$ the category of filtered
$\varphi$-modules over~$K$ (the morphisms are the $K_0$-linear
mappings commuting with $\varphi$ and preserving the filtration
after scalar extension to $K$).
We have a natural functor $\MF_K(\varphi) \to \MF_K$ taking $D$ to $K 
\otimes_{K_0} D$ equipped with its filtration.
Besides, the previous constructions give rise to a functor
$$D_\crys : \Rep^\crys_{\Qp}(G_K) \to \MF_K(\varphi)$$
whose composite with $\MF_K(\varphi) \to \MF_K$ is $D_\dR$.

\begin{theo}
\label{theo:dcrys}
The function $D_\crys$ is exact and fully faithful.
\end{theo}

\begin{proof}
The fact that $D_\crys$ is exact follows directly by a dimension
argument.

Let $V \in \Rep^\crys_{\Qp}(G_K)$ and set $D = D_\crys(V)$.
We then have a canonical isomorphism
$\Bcrys \otimes_{\Qp} V \simeq \Bcrys \otimes_{K_0} V$
which commutes with Frobenius and respects the filtration after 
extending scalars to $\BdR$. Taking the $\Fil^0$ and the fixed 
points under the Frobenius and using Proposition~\ref{prop:bmuphi1},
we obtain:
\begin{equation}
\label{eq:vcrys}
V = \big(\Bcrys \otimes_{K_0} D\big)^{\varphi=1} \,\cap\,
\Fil^0\big(\BdR \otimes_K D_K\big)
\end{equation}
(where the supscript ``$\varphi{=}1$'' means that we are taking
fixed points under the Frobenius).
Formula~\eqref{eq:vcrys} defines a functor $V_\crys : \MF_K(\varphi)
\to \Rep_{\Qp}(G_K)$ and we have just proved that $V_\crys \circ
D_\crys$ is the identity. This is enough to ensure that 
$D_\crys$ is fully faithful.
\end{proof}

\begin{rem}
In the proof of Theorem~\ref{theo:dcrys}, instead of $\BdR$, we could 
have used the smaller ring $\BcrysK$. Similarly, we could have 
replaced everywhere the subscript ``$\crys$'' by $\mu$ for any real 
number $\mu \geq 1$.
\end{rem}

\begin{rem}
Let $A$ is an abelian variety over $K$ with good reduction and let 
$A[p^\infty]$ be the $p$-divisible groups of its points of 
$p^\infty$-torsion.
The étale cohomology (resp. the crystalline cohomology) of $A$ is 
then identified with the Tate module (resp. the Dieudonné module) 
of $A[p^\infty]$. The fact that $D_\crys$ is fully faithful then
reflects the fact that Dieudonné modules (equipped with the de Rham
filtration) classify $p$-divisible groups.
\end{rem}

\paragraph{Admissibility for $\varphi$-modules.}

We say that a filtered $\varphi$-module over $K$ is \emph{admissible} 
if it belongs to the essential image of $D_\crys$, and we denote by 
$\MF_K^\adm(\varphi)$ the full subcategory of $\MF_K(\varphi)$ 
consisting of admissible filtered $\varphi$-modules. 
Theorem~\ref{theo:dcrys} indicates that $D_\crys$ induces an equivalence 
of categories $\Rep^\crys_{\Qp}(G_K) \simeq \MF_K^\adm (\varphi)$. 
This result provides a very concrete description of crystalline 
representations as soon as we are able to recognize admissible filtered 
$\varphi$-modules among all filtered $\varphi$-modules.

This is actually possible: there exists an easy numerical criterium that 
caracterizes admissibility. We would like to conclude this article by
stating it (without proof).
Let $D \in \MF_K(\varphi)$ and set $d = \dim_{K_0} D$. The maximal 
exterior product $\det\:D = \bigwedge^d D$ has a natural structure of 
filtered $\varphi$-module: the Frobenius on it is $\bigwedge^d \varphi$ 
(where the latter $\varphi$ is the Frobenius acting to $D$) and:
$$\Fil^m (K \otimes_{K_0} \det\:D) = \sum_{m_1 + \ldots + m_d = m}
\Fil^{m_1} D_K \wedge \Fil^{m_2} D_K \wedge \cdots \wedge \Fil^{m_d} 
D_K$$
where we have set $D_K = K \otimes_{K_0} D$.
Since $\det\:D$ is one dimensional, there exists a unique integer
$m$ for which $\Fil^m (K \otimes_{K_0} \det\:D) = K \otimes_{K_0}
\det\:D$ and $\Fil^{m+1} (K \otimes_{K_0} \det\:D) = 0$. This
integer is called the \emph{Hodge number} of $D$ and is usually
denoted by $t_H(D)$. It is an easy exercise to check that we have 
the following alternative formula for $t_H(D)$:
$$t_H(D) = \sum_{m \in \Z} m \cdot \dim_K \gr^m D_K.$$
Similarly, we can assign an integer to $D$ measuring the action
of the Frobenius. Precisely, if $v$ is nonzero element of $\det\:D$, 
we have $\bigwedge^d
\varphi(v) = \lambda v$ for some $\lambda\in K_0$. One checks
easily that $v_p(\lambda)$ does not depend on the choice of $V$.
We call it the \emph{Newton number} of $D$ and denote by $t_N(D)$.

\begin{theo}
\label{theo:phimodadm}
A filtered $\varphi$-module $D$ over $K$ is admissible if and
only if the two following conditions hold:

\vspace{-2mm}

\begin{enumerate}[(i)]
\renewcommand{\itemsep}{0pt}
\item $t_H(D) = t_N(D)$
\item for all sub-$K_0$-vector space
$D' \subset D$ stable by the Frobenius, we have $t_H(D') \leq 
t_N(D')$, where $D'$ is endowed with the induced filtration
defined by:
$$\Fil^m (K \otimes_{K_0} D') = (K \otimes_{K_0} D') \cap
\Fil^m (K \otimes_{K_0} D)
\qquad (m \in \Z).$$
\end{enumerate}
\end{theo}

\noindent
Theorem~\ref{theo:phimodadm} was first conjectured by Fontaine 
in~\cite{fontaine-rennes}. It has been proved first by Colmez and Fontaine in 
\cite{colmez-fontaine} about twenty years later. Today, other proofs are 
been proposed by different authors~\cite{berger,kisin,fargues-fontaine}, 
but Theorem~\ref{theo:phimodadm} remains a difficult result in all
cases.
Kisin's proof~\cite{kisin} will be sketched in Brinon's lecture in this 
volume~\cite{brinon} (in the more general framework of filtered 
$(\varphi,N)$-modules).

\begin{ex}
As an easy example, let us give a complete classification of 
filtered $\varphi$-modules of dimension~$1$. Let then $D \in
\MF_K(\varphi)$ with $\dim_{K_0} D = 1$; write $D_K = K 
\otimes_{K_0} D$.
Let $e$ be a basis of $D$. Then, there exists $\lambda \in K_0$ such 
that $\varphi(e) = \lambda e$. Observe that if $e$ is changed to $u e$ 
(with $u \in K_0$), $\lambda$ becomes $\lambda \cdot \frac 
{\varphi(u)}{u}$. By Hilbert's theorem 90, the elements of the form $\frac 
{\varphi(u)}{u}$ are exactly the elements of norm~$1$ over $\Qp$. 
Therefore $N_{K_0/\Qp}(\lambda)$ does not depend on a choice of $e$ and 
is a complete invariant classying the possible $\varphi$'s on~$D$.
Concerning the filtration, there exists a unique integer $r$ such 
that $\Fil^m D_K = D_K$ if $m \leq r$ and $\Fil^m D_K = 0$ otherwise.

One sees immediately that the $D$ is admissible if and only if 
$v_p(\lambda) = r$. Moreover, when admissibility holds, an easy 
computation shows that the attached Galois representation $V_\crys(D)$ 
is given by the character $\chi_\cycl^{-r} \cdot \mu_\alpha^{-1}$ with 
$\alpha = N_{K_0/\Qp}(p^{-r}\lambda)$.
\end{ex}

\begin{ex}
We now investigate the admissible filtered $\varphi$-modules of 
dimension~$2$ over $\Qp$.
We then consider $D \in \MF^\adm_{\Qp}(\varphi)$ with $\dim_{\Qp} D = 
2$.
The filtration on $D$ is easy to describe: there exist two integers 
$r$ and $s$ with $r \leq s$ and a line $L \subset D$ such that
$\Fil^m D = D$ if $m \leq r$, $\Fil^m D = L$ if $r < m \leq s$
and $\Fil^m D = 0$ if $m > s$.
If $r = s$, it follows from Proposition~\ref{prop:crysunram} that the 
crystalline representation associated to $D$ (if $D$ is admissible) has 
the form $V(\chi_\cycl^{-r})$ for an unramified representation $V$.
We leave this case to the reader and assume now that $r < s$.
Then $L$ is uniquely determined.

We want to describe the action of the Frobenius $\varphi : D \to D$. 
Let us first notice that $\varphi$ is a \emph{linear} mapping because the 
Frobenius acts trivially on $\Qp$.
We first assume that $L$ is stable by the Frobenius. We let
$\alpha \in \Qp$ be the scalar by which $\varphi$ acts on $L$ and
we let $\beta$ be the second eigenvalue of~$\varphi$. 
From the admissibility condition, we deduce $v_p(\alpha) + v_p(\beta) = 
r+s$ and $v_p(\alpha) \geq s$. Therefore $v_p(\beta) \leq r < s$.
Hence $\alpha \neq \beta$ and $\varphi$ is diagonalizable. If $L'$
denotes the eigenspace associated to $\beta$, we have $t_N(L') =
v_p(\beta)$ and $t_H(L') = r$. By the admissibility condition,
this implies that $v_p(\beta) = r$ and then $v_p(\alpha) = s$.
Then $L$ and $L'$ are themselves \emph{admissible} filtered $\varphi$-modules 
of dimension~$1$ and $D$ splits as a direct sum $D = L \oplus L'$.
The attached Galois representation is then a direct sum of two
crystalline characters.

We now assume that $L$ is not stable under $\varphi$.
Let $e_1$ be a nonzero vector in $L$. Define $e_2 \in D$ by the
equality $\varphi(e_1) = p^s e_2$. The family $(e_1, e_2)$ is a
basis of $D$ in which the matrix of $\varphi$ has the form:
$$\Phi = 
\left( \begin{matrix} 0 & p^r a \\ p^s & p^r b \end{matrix}\right)$$
for $a,b \in \Qp$. The admissibility condition implies 
$v_p(\det\:\Phi) = r + s$, and then $a \in \Zp^\times$. It also
implies that any eigenvalue of $\Phi$ must have valuation at 
least $r$. But if $v_p(b) < 0$, we see on the Newton polygon
of the characteristic polynomial of $\Phi$, that $\Phi$ has an
eigenvalue of valuation strictly less than $r$. Therefore, we conclude 
that $v_p(b) \geq 0$, \emph{i.e.} $b \in \Zp$. When $b \in \Zp^\times$,
$\varphi$ has two eigenvalues of valuation $r$ and $s$ respectively.
Let $L_r$ and $L_s$ be the corresponding eigenspaces. Since $e_1$
is not an eigenvector, we have $t_H(L_r) = t_H(L_s) = r$. Hence,
$L_r$ is admissible and we have the exact sequence 
$0 \to L_r \to D \to D/L_r \to 0$ is $\MF^\adm_{\Qp}(\varphi)$.
Passing to Galois representations, we find that $V_\crys(D)$
is a non split extension of $\Qp(\chi_\cycl^{-s}\mu_\alpha)$ by
$\Qp(\chi_\cycl^{-r}\mu_\beta)$ with $\alpha, \beta \in \Zp^\times$.

On the contrary, when $v_p(b) > 0$, $D$ is admissible and irreducible in 
the category $\MF^\adm_{\Qp}(\varphi)$. It then gives rise to an 
irreducible crystalline representation of dimension~$2$ of $G_{\Qp}$, 
whose Hodge--Tate weights are $r$ and~$s$.
\end{ex}


\end{document}